\documentclass[final]{siamltex}

\usepackage[notcite,notref]{showkeys}\usepackage[notcite,notref]{showkeys}

\usepackage{amsfonts}
\usepackage{amscd}
\usepackage{amsmath}
\usepackage{amssymb}
\usepackage{epic}
\usepackage{multirow}
\usepackage{makeidx}
\usepackage{mathscinet}
\usepackage{hyperref}
\usepackage{graphicx}
\usepackage{subfigure}
\usepackage{fancyhdr}
\usepackage[english]{babel}
\usepackage[latin1]{inputenc}
\usepackage{srcltx}
\usepackage{epstopdf}
\usepackage[usenames]{color}
\usepackage{mathdots}
\DeclareMathOperator{\heads}{\mathrm{heads}}

\usepackage{color}
\usepackage{mathtools}
\usepackage{enumerate}
\usepackage{verbatim}
\usepackage{arydshln}

\DeclareMathOperator{\csf}{\mathrm{csf}}
\DeclareMathOperator{\rev}{\mathrm{rev}}

\definecolor{Myblue2}{rgb}{0.7,0,1}

\newcommand{\red}{\color{red} }

\usepackage{exscale,cmmib57}

\usepackage{mathrsfs}
\usepackage{times}

\newtheorem{remark}{Remark}[section]
\newtheorem{example}{Example}[section]

\title{A unified approach to  Fiedler-like pencils via strong block minimal bases pencils.}

\author{
M. I. Bueno\thanks{Department of Mathematics and College of Creative Studies,
University of California, Santa Barbara, CA 93106, USA
({\tt mbueno@math.ucsb.edu}). The research of M. I. Bueno was partially supported by NSF grant DMS-1358884 and partially supported by Ministerio de Economia y Competitividad of Spain through grants MTM2015-65798-P.}
\and
F. M. Dopico \thanks{Departamento de Matem\'aticas, Universidad Carlos III de Madrid,
        Avda.\ Universidad 30, 28911 Legan\'es, Spain  ({\tt dopico@math.uc3m.es}). The research of F. M. Dopico was partially supported by the Ministerio de Econom\'{i}a y Competitividad of Spain through grant MTM2015-65798-P and MTM2015-68805-REDT.}
\and
J. P\'erez \thanks{Department of Computer Science, KU Leuven, Celestijnenlaan 200A, 3001 Heverlee, Belgium.({\tt javierperez@kuleuven.be}). The research of J. P\'erez was partially supported by KU Leuven Research Council grant OT/14/074 and the Interuniversity Attraction Pole DYSCO, initiated by the Belgian State Science Policy Office}
\and
R. Saavedra \thanks{Dos Pueblos High School, Goleta, CA 93117, USA ({\tt rmsaavedra@umail.ucsb.edu}). The research
of R. Saavedra was supported by NSF grant DMS-1358884.}
\and
B.  Zykoski \thanks{ Department of Mathematics, University of Virginia, Charlottesville, VA 22903, USA  ({\tt bpz8nr@virginia.edu}). The research of B. Zykoski was supported by NSF grant DMS-1358884} }

\begin{document}

\maketitle

\bibliographystyle{plain}

\begin{abstract}
The standard way of solving the polynomial eigenvalue problem associated with a matrix polynomial is to embed the matrix polynomial into a matrix pencil, transforming the problem into an equivalent generalized eigenvalue problem.
Such pencils are known as linearizations. 
Many of the families of  linearizations for matrix polynomials available in the literature are  extensions of the so-called family of Fiedler pencils.
These  families are known as generalized Fiedler pencils, Fiedler pencils with repetition and generalized Fiedler pencils with repetition, or Fiedler-like pencils for simplicity.
The goal of this work is to unify the Fiedler-like pencils approach with the more recent one based on strong block minimal bases pencils introduced in \cite{canonical}.
To this end, we introduce a family of pencils that we have named extended block Kronecker pencils, whose members are, under some generic nonsingularity conditions, strong block minimal bases pencils, and show that, with the exception of the non proper generalized Fiedler pencils, all Fiedler-like pencils belong to this family modulo permutations.
As a consequence of this result, we obtain a  much simpler theory for Fiedler-like pencils than the one available so far.
Moreover, we expect this unification to allow for further developments in the theory of Fiedler-like pencils such as global or local backward error analyses and eigenvalue conditioning analyses of polynomial eigenvalue problems solved via Fiedler-like linearizations.
   \end{abstract}



\pagestyle{myheadings}
\thispagestyle{plain}
\markboth{M. I. Bueno,  F. M. Dopico, J. Perez, R. Saavedra, and B. Zykoski}{ Fiedler-like pencils as strong block minimal bases}

\section{Introduction}

\emph{Matrix polynomials} and their associated \emph{polynomial eigenvalue problems} appear  in many areas of applied mathematics, and they have received in the last years considerable attention.
For example, they are ubiquitous in a wide range of problems in engineering, mechanic, control theory, computer-aided graphic design, etc.
For detailed discussions of different applications of matrix polynomials, we refer the reader to the classical references \cite{Lancaster_book,Kailath,Rosenbrock}, the modern surveys \cite[Chapter 12]{Volker_book} and \cite{NEV,quadratic}  (and their references therein), and the references \cite{Fiedler_Steve,4m-vspace,GoodVibrations}.
For those readers not familiar with the theory of matrix polynomials and polynomial eigenvalue problems, those topics are briefly reviewed in Section \ref{sec:intro}.

The standard way of solving the polynomial eigenvalue problem associated with a matrix polynomial is to \emph{linearize} the polynomial into a matrix pencil (i.e., matrix polynomials of grade 1), known as linearization \cite{DTDM14,Strong_lin,Lancaster_book}.
The linearization process transforms the polynomial eigenvalue problem into an equivalent generalized eigenvalue problem, which, then, can be solved using  mature and well-understood eigensolvers such as the QZ algorithm or the staircase algorithm, in the case of singular matrix polynomials  \cite{QZ,staircase,VanDooren83}.
Ideally, to make a set of linearizations desirable for numerical applications, it should satisfy the following list of properties:
\begin{itemize}
\item[\rm (i)] the linearizations should be \emph{strong linearizations}, regardless whether the matrix polynomial is regular or singular;
\item[\rm (ii)] the linearizations should be easily constructible from the coefficients of the matrix polynomials (ideally, without any matrix operation other than scalar multiplication);
\item[\rm (iii)] eigenvectors of regular matrix polynomials should be easily recovered from those of the linearizations;
\item[\rm (iv)] minimal bases of singular matrix polynomials should be easily recovered from those of the linearizations;
\item[\rm (v)] there should exist simple relations between the minimal indices of singular matrix polynomials and the minimal indices of the linearizations, and such relations should be robust under perturbations;
\item[\rm (vi)] guarantee global backward stability of polynomial eigenvalue problems solved via  linearization.
\end{itemize}
Additionally, some authors like to add to the above list the following property:
\begin{itemize}
\item[\rm (vii)] the linearizations should present one-sided factorizations (as those used in \cite{Framework}), which are useful for performing residual local (i.e., for each particular computed eigenpair) backward error and eigenvalue conditioning  analyses of regular polynomial eigenvalue problems solved by linearizations \cite{BackErrors,Conditioning}.
\end{itemize}
Furthermore, matrix polynomials that appear in applications usually present algebraic structures, which are reflected in their spectra (see \cite[Section 7.2]{DTDM14} or \cite{GoodVibrations}, for example).
If the spectrum of such a polynomial is computed without taking into account the algebraic structure of the polynomial, the rounding errors inherent to numerical computations may destroy qualitative properties of these spectra.
Thus, to the mentioned list of desirable properties that a set of linearizations should satisfy, the following property should be added:
\begin{itemize}
\item[\rm (viii)] the linearizations of a matrix polynomial $P(\lambda)$ should preserve any algebraic structure that $P(\lambda)$ might posses \cite{GoodVibrations},
\end{itemize} 
and property (vi) should be replaced in the structured case by
\begin{itemize}
\item[\rm (vi-b)] guaranteed structured and global backward stability of structured polynomial eigenvalue problems solved via structure-preserving linearizations \cite{structured_canonical}.
\end{itemize}

In practice, the linearizations used to solve polynomial eigenvalue problems are the well known \emph{Frobenius companion forms}. 
It has been proved that these pencils satisfy properties (i)--(vii) (see \cite{DTDM14,DTT16,VanDooren83}, for example).
However, Frobenius companion forms do not preserve the algebraic structure that might be present in the matrix polynomials they are associated with, that is, they do not satisfy property (viii).
This important drawback of the Frobenius companion forms has motivated an intense activity on the theory of linearizations of matrix polynomials, with the  goal of finding linearizations satisfying properties (i)--(vii) that, additionally,  retain whatever structure the matrix polynomial might possess.
See, for example, \cite{AV04,FPR1,GFPR,FPR2,FPR3,DTDM11,FPS16,Philip,symmetric,LP16,
GoodVibrations,ChebyFiedler,RVVD16,ant-vol11}, which is a small sample of recent references in this area.

Although not the first approach (see \cite{Lancaster_sym,Lancaster_poly_book}), the first systematic approach for constructing structure-preserving linearizations was based on pencils belonging to the vector space $\mathbb{DL}(P)$.
This vector space was introduced in \cite{4m-vspace} and further analyzed in \cite{BackErrors,symmetric,Conditioning,GoodVibrations,NNT}. 
The pencils in this vector space are easily constructible from the matrix coefficients of the matrix polynomial, and most of them are strong linearizations when the polynomial is regular.
However, none of these pencils is a strong linearization when the polynomial is singular \cite{singular_lin} (i.e., these linearizations do not satisfy property (i)), which ``questions the utility of such space also when the polynomials are regular but very close to be singular'' \cite{GFPR}. Even more, none of these pencils, when constructed for a symbolic arbitrary matrix polynomial is always a strong linearization of all regular matrix polynomials.

The second systematic approach for constructing structure-preserving linearizations is based on different extensions of \emph{Fiedler pencils}.
Fiedler pencils were introduced in \cite{Fiedler03} for monic scalar polynomials, and then generalized to regular matrix polynomials in \cite{AV04}, to square singular matrix polynomials in \cite{DTDM10}, and to rectangular matrix polynomials in \cite{DTDM12}.
These pencils were baptized as \emph{Fiedler companion pencils} (or Fiedler companion linearizations) in \cite{DTDM10}.
Later on, with the aim of constructing large families of structure-preserving linearizations, the definition of Fiedler companion pencils was extended to include the families of \emph{generalized Fiedler pencils}, \emph{Fiedler pencils with repetition} and \emph{generalized Fiedler pencils with repetition} \cite{AV04,GFPR,ant-vol11}, unifying, in addition, the $\mathbb{DL}(P)$ approach with the Fiedler pencils approach (we recall that the standard basis of $\mathbb{DL}(P)$ consists of Fiedler pencils with repetition \cite{FPR1,ant-vol11}).
Furthermore, the definition of Fiedler pencils has been extended, also, to allow the construction of linearizations for matrix polynomials that are expressed in other non-monomial bases without any conversion to the monomials \cite{Bernstein,ChebyFiedler}.

The family of Fiedler companion linearizations is a remarkable set of pencils, since, as it has been proven in \cite{DTDM10,DTDM12,DTT16}, it satisfies properties (i)--(v) and property (vii). 
However, the proofs of these properties were extremely involved, being the stem of the difficulties the implicit way Fiedler pencils are defined, either in terms of a product of matrices for square polynomials or as the output of a symbolic
algorithm for rectangular ones.
Moreover, the proofs of the extensions of some of these results to the different generalizations of Fiedler pencils mentioned in the previous paragraph are even much more involved, and they only work for square matrix polynomials \cite{recovery_all_Fid,recovery_gen_Fid}.

With the double aim of proving for the first time that Fiedler companion linearizations  provide  global backward stable methods for polynomial eigenvalue problems (i.e., they satisfy property (vi)), and obtaining a simplified theory for them, in \cite{canonical}, the authors introduced the family of \emph{block minimal bases pencils} and the subfamily of block minimal bases pencils named \emph{block Kronecker pencils}.
The introduction of these families of pencils has been an interesting recent advance in the theory of linearizations of matrix polynomials, since the families of block minimal bases pencils and  block Kronecker pencils have been shown to be a fertile source of linearizations satisfying properties (i)-(vii) \cite{canonical}, and also (vi-b) and (viii)  in the structured matrix polynomial case \cite{structured_canonical,LP16,RVVD16}.

The reason why the theory developed for block minimal bases pencils works for Fiedler pencils as well is that, up to permutations of rows and columns, Fiedler pencils  are block Kronecker pencils \cite[Theorem 4.5]{canonical}, a particular type of block minimal bases pencils.
In this work, we extend the result obtained for Fielder pencils in \cite{canonical} to all the other families of Fiedler-like pencils (generalized Fiedler pencils, Fiedler pencils with repetition and generalized Fiedler pencils with repetition).
More explicitly, our main result is that, with rare exceptions, the pencils in all these families  belong, up to permutations, to the family of \emph{extended block Kronecker pencils}, which, under some generic nonsingularity conditions,   are block minimal bases pencils.
This result allows the application of all the tools and machinery developed for block minimal bases pencils to most Fielder-like pencils too.
For example, these tools could be used to  try to determine whether or not solving a polynomial eigenvalue problem via a Fiedler-like linearization is backward stable from the point of view of the polynomial, which is still an open problem (except for the particular case of Fiedler companion pencils), or to perform  local residual backward error and eigenvalue conditioning analyses. 

 We want to remark that the family of extended block Kronecker pencils has been introduced independently in \cite{Philip2016} under the name of block-Kronecker ansatz spaces motivated, as in our case, by the results in \cite{canonical}. However, the goal of \cite{Philip2016} is different than ours. While the goal of \cite{Philip2016} is to construct a new source of strong linearizations of matrix polynomials over the real numbers and to establish connections of the ``block-Kronecker ansatz spaces''  with other ansatz spaces of potential strong linearizations previously available in the literature \cite{DLP}, our goal is to provide a unified approach to all the families of Fiedler-like pencils in any field via the more general concept of strong block minimal bases pencils. We emphasize again in this context that all the extended block Kronecker pencils and, so, all pencils in \cite{Philip2016} that are strong linearizations are particular cases of the (strong) block minimal bases pencils introduced in \cite{canonical} which seem to be a key unifying concept for studying linearizations of matrix polynomials.

 The paper is structured as follows.
In Section \ref{sec:intro}, we introduce the definitions and notation used throughout the paper, and present some basic results needed in other sections.
In Section \ref{sec:block_minimal_bases_pencils}, we review the framework of (strong) block minimal bases pencils and the family of block Kronecker pencils.
We introduce the family of extended block Kronecker pencils that will be used to express all the Fiedler-like pencils into the framework of block minimal bases pencils.
In this section, we also state, informally, our main results, and we illustrate them with some illuminating examples.
In Section  \ref{sec:index_tuples}, we review  the families of (square) Fiedler pencils, generalized Fiedler pencils, Fiedler pencils with repetition and generalized Fiedler pencils with repetition,  introduce the index tuple notation, which is  needed to define and work in an effective way with the Fiedler-like pencils families, and present some auxiliary results used in the subsequent sections. 
In Section 5, we introduce some technical lemmas regarding matrix pencils that are block-permutationally equivalent to extended block Kronecker pencils. These technical results are used to prove the main theorems in the paper but are interesting by themselves.
In Sections \ref{sec:FP_as_minimal_bases_pencils},  \ref{sec:GF_as_block_Kron_pen} and \ref{sec:GFPR_as_minimal_bases_pencils}, we present and prove our main results.
We show that, up to permutations, proper generalized Fiedler pencils, Fiedler pencils with repetition and generalized Fiedler pencils with repetition are extended block Kronecker pencils. 
In particular, we show that proper generalized Fiedler pencils are block Kronecker pencils modulo permutations.
Finally, in Section \ref{sec:conclusion}, we summarize our results.
 
\section{Notation, definitions, and basic results}\label{sec:intro}

Throughout the paper, given two integers  $a$ and $b$,  we denote
$$a:b:=\left\{ \begin{array}{ll} a, a+1,\ldots, b, & \textrm{if $a\leq b$,}\\
\emptyset, & \textrm{if $a>b$.}
 \end{array}\right. $$

We will use $\mathbb{F}$ to denote an arbitrary field, $\mathbb{F}[\lambda]$ to denote the ring of polynomials with coefficients from the field $\mathbb{F}$, and $\mathbb{F}(\lambda)$
to denote the field of rational functions over $\mathbb{F}$.
The algebraic closure of the field $\mathbb{F}$ is denoted by $\overline{\mathbb{F}}$.
The set of $m\times n$ matrices with entries in $\mathbb{F}[\lambda]$ is denoted by $\mathbb{F}[\lambda]^{m\times n}$. 
Any $P(\lambda)\in\mathbb{F}[\lambda]^{m\times n}$ is called an \emph{$m\times n$ matrix polynomial}, or, just a matrix polynomial when its size is clear from the context or is not relevant.
Moreover, when $m=1$ (resp. $n=1$), we refer to $P(\lambda)$ as a \emph{row vector polynomial} (resp. \emph{column vector polynomial}).
A matrix polynomial $P(\lambda)$ is said to be \emph{regular} if it is square and the scalar polynomial $\det P(\lambda)$ is not identically equal to the zero polynomial, and \emph{singular} otherwise. 
If $P(\lambda)$ is regular and $\det P(\lambda)\in\mathbb{F}$, then $P(\lambda)$ is said to be \emph{unimodular}.

A matrix polynomial $P(\lambda)\in\mathbb{F}[\lambda]^{m\times n}$ is said to have \emph{grade} $k$ if it can be expressed
in the form
\begin{equation}\label{pol}
P(\lambda)=\sum_{i=0}^k A_i \lambda^i, \quad \mbox{with} \quad A_0,\hdots,A_k\in \mathbb{F}^{m\times n},
\end{equation}
where any of the coefficients, including $A_k$, can be zero. 
We call the degree of $P(\lambda)$ the maximum integer $d$ such that $A_d\neq 0$. 
The degree of a matrix polynomial $P(\lambda)$ is denoted by $\deg(P(\lambda))$.
Note, in addition, that the degree of $P(\lambda)$ is fixed while the grade of a matrix polynomial, which is always larger than or equal to its degree, is a choice.
A matrix polynomial of grade 1 is called a \emph{matrix pencil}, or, sometimes for simplicity, a pencil. 

For any $k\geq \deg(P(\lambda))$, the \emph{$k$-reversal matrix polynomial} of $P(\lambda)$ is the matrix polynomial
\[
\rev_k P(\lambda):=\lambda^k P(\lambda^{-1}).
\]
Note that the $k$-reversal operation maps matrix polynomials of grade $k$ to matrix polynomials of grade $k$. 
 When the degree of a matrix polynomial $P(\lambda)$ is clear from the context, we write $\rev P(\lambda)$ to denote the $\deg(P(\lambda))$-reversal of $P(\lambda)$.

The \emph{complete eigenstructure} of a regular matrix polynomial consists of its finite and infinite elementary divisors (spectral structure), and for a singular matrix polynomial it consists of its finite and infinite elementary divisors together with its right and left minimal indices (spectral structure and singular structure). The singular structure of matrix polynomials will be briefly reviewed later in the paper.
For more detailed definitions of the spectral  structure of matrix polynomials, we refer the reader to \cite[Section 2]{DTDM14}.
The problem of computing the complete eigenstructure of a matrix polynomial is called the \emph{complete polynomial eigenvalue problem}.

The standard approach to solving a complete polynomial eigenvalue problem is via linearizations.
 A matrix pencil $L(\lambda)$ is a \emph{linearization} of a matrix polynomial $P(\lambda)$ of grade $k$ if, for some $s\geq 0$, there exist two unimodular matrix polynomials $U(\lambda)$ and $V(\lambda)$ such that 
$$U(\lambda) L(\lambda) V(\lambda) = \left[ \begin{array}{cc} I_s & 0 \\ 0 & P(\lambda) \end{array} \right],$$
where $I_s$ denotes the identity matrix of size $s$.
Additionally, a linearization $L(\lambda)$ of $P(\lambda)$ is a \emph{strong linearization} if $\rev_1 L(\lambda)$ is a linearization of $\rev_k P(\lambda)$.
We recall that the key property of any strong linearization $L(\lambda)$ of $P(\lambda)$ is that $P(\lambda)$ and $L(\lambda)$ have the same finite and infinite elementary divisors, and the same numbers of right and left minimal indices.

Minimal bases and minimal indices play an important role in the developments of this work, so we briefly review them here (for a more complete description of minimal bases and their properties see \cite{Forney}).
Any subspace $\mathcal{W}$ of $\mathbb{F}(\lambda)^n$  has bases consisting entirely of vector polynomials. 
The \emph{order} of a vector polynomial basis of $\mathcal{W}$ is defined as the sum of the degrees of its vectors. 
Among all of the possible polynomial bases of $\mathcal{W}$, those with least order are called \emph{minimal bases} of $\mathcal{W}$. 
In general, there are many minimal bases of $\mathcal{W}$, but the ordered list of degrees of the vector polynomials in any of its minimal bases  is always the same. 
This list of degrees is called the list of \emph{minimal indices} of $\mathcal{W}$.

When an $m\times n$ matrix polynomial $P(\lambda)$ is singular, it has nontrivial right and/or left \emph{rational null spaces}:
\begin{align*}
&\mathcal{N}_r(P):=\{x(\lambda)\in \mathbb{F}(\lambda)^{n\times 1} : \; P(\lambda)x(\lambda)=0\}, \\
&\mathcal{N}_{\ell}(P):=\{y(\lambda)^T\in \mathbb{F}(\lambda)^{1\times m} : \; y(\lambda)^T P(\lambda)=0\}.
\end{align*}
The left (resp. right) minimal bases and minimal indices of $P(\lambda)$ are defined as those of the rational subspace $\mathcal{N}_\ell(P)$ (resp. $\mathcal{N}_r(P)$).

In order to give a useful characterization of the minimal bases of a rational subspace, we introduce the concept of \emph{highest degree coefficient matrix}. 
Here and thereafter, by the \emph{$i$th row degree} of a matrix polynomial $Q(\lambda)$ we denote the degree of the $i$th row of $Q(\lambda)$.
\begin{definition}
Let $Q(\lambda) \in \mathbb{F}[\lambda]^{m\times n}$ be a matrix polynomial with row degrees $d_1, d_2, \ldots, d_m$. The \emph{highest row degree coefficient matrix} of $Q(\lambda)$, denoted $Q_h$, is the $m\times n$ constant matrix whose $i$th entry in the $j$th row is the coefficient of $\lambda^{d_j}$ in the $(i,j)$th entry of  $Q(\lambda)$, for $i=1:m$ and $j=1:n$. 
\end{definition}

We illustrate the concept of highest degree coefficient matrix in the following example.
In particular, we show that the highest degree coefficient matrix must not be confused with the leading coefficient of a matrix polynomial.
\begin{example}
Consider the matrix polynomial
\[
Q(\lambda)  = \left[ \begin{matrix} 2 \lambda^2 -3 \lambda & 3- 5 \lambda & \lambda^2\\ \lambda +1 & \lambda +2 & -3\lambda^2   \end{matrix} \right].
\]
The highest row degree coefficient of $Q(\lambda)$ is
$
Q_h = \left[\begin{smallmatrix} 2 & 0 & 1 \\ 0 & 0 & -3  \end{smallmatrix} \right].
$
Notice that, in this case, the highest row degree coefficient of $Q(\lambda)$ coincides with its leading coefficient.
However, this is not always true.
For example, consider the matrix polynomial
\[
P(\lambda)  = \left[ \begin{array}{ccc} 2 \lambda^2 -3 \lambda & 3- 5 \lambda & \lambda^2\\ \lambda +1 & \lambda +2 & 1   \end{array} \right].
\]
Then, the highest row degree coefficient of $P(\lambda)$ is
$
P_h = \left[\begin{smallmatrix} 2 & 0 & 1 \\ 1 & 1 & 0 \end{smallmatrix} \right].
$
Notice that, in this case, the leading coefficient of $P(\lambda)$ and $P_h$ are not equal.
\end{example}

 In this paper, we arrange  minimal bases as the rows of  matrix polynomials. 
Then, we  say that an $m\times n$ matrix polynomial, with $m\leq n$, is a minimal basis if its rows form a minimal basis for the rational space they span.

\begin{theorem}\label{minimal-basis}{\rm \cite[Theorem 2.2]{canonical}}
A matrix polynomial $Q(\lambda)\in \mathbb{F}[\lambda]^{m\times n}$ is a minimal basis if and only if $Q(\lambda_0) \in \overline{\mathbb{F}}^{m\times n}$ has full row rank for all $\lambda_0\in \overline{\mathbb{F}}$ and the highest row degree coefficient matrix $Q_h$ of $Q(\lambda)$ has full rank.
\end{theorem}

Next, we introduce the definition of \emph{dual minimal bases}, a concept that plays a key role in the construction of the matrix pencils in Section \ref{sec:block_minimal_bases_pencils}.
\begin{definition}
Two matrix polynomials $K(\lambda) \in \mathbb{F}[\lambda]^{m_1\times q}$ and $N(\lambda)\in \mathbb{F}[\lambda]^{m_2\times q}$ are called \emph{dual minimal bases} if $m_1+m_2=q$ and $K(\lambda)$ and $N(\lambda)$ are both minimal bases satisfying  $K(\lambda) N(\lambda)^T=0$. 
We also say that $K(\lambda)$ is a dual minimal basis to
$N(\lambda)$, or vice versa.
\end{definition}

We introduce in Example \ref{ex:dual_bases1} a simple pair of dual minimal bases that plays an important role in this paper.
Here and throughout the paper we omit occasionally some, or all, of the entries of a matrix.
\begin{example}\label{ex:dual_bases1}
Consider the following matrix polynomials:
\begin{equation}\label{eq:Ls}
L_s(\lambda):=
\begin{bmatrix}
-1 & \lambda \\
& -1 & \lambda \\
& & \ddots & \ddots \\
& & & -1 & \lambda
\end{bmatrix}\in\mathbb{F}[\lambda]^{s\times (s+1)},
\end{equation}
and
\begin{equation}\label{Lambda}
\Lambda_s(\lambda) :=\begin{bmatrix} \lambda^s & \lambda^{s-1} & \cdots & \lambda & 1  \end{bmatrix}\in\mathbb{F}[\lambda]^{1\times (s+1)}.
\end{equation}
It follows easily from Theorem \ref{minimal-basis} that the matrix polynomials $L_s(\lambda)$ and $\Lambda_s(\lambda)$ are both minimal bases.
Also, since $L_s(\lambda)\Lambda_s(\lambda)^T=0$, we conclude that $L_s(\lambda)$ and $\Lambda_s(\lambda)$ are dual minimal bases.
Additionally, it follows from basic properties of the Kronecker product  that $L_s(\lambda)\otimes I_n$ and $\Lambda_s(\lambda)\otimes I_n$ are also dual minimal bases.
\end{example}
\begin{remark}
The vector polynomial $\Lambda_s(\lambda)$ in \eqref{Lambda} is very well known in the theory of linearizations of matrix polynomials, and plays an essential role, for instance, in  \cite{canonical,structured_canonical,4m-vspace,GoodVibrations}. 
However, notice that in those references $\Lambda_s(\lambda)$ is defined as a column vector, while in this work, for convenience, we define it as a row vector.
\end{remark}

In Section \ref{sec:block_minimal_bases_pencils}, we will  work with minimal bases dual to $\Lambda_s(\lambda)\otimes I_n$ other than $L_s(\lambda)\otimes I_n$.
This  motivates the following definition.
\begin{definition}\label{def:wing_pencil}
Let $n,s\in\mathbb{N}$, and let $\Lambda_s(\lambda)$ be the matrix polynomial in \eqref{Lambda}.
A  matrix pencil $K(\lambda)\in\mathbb{F}[\lambda]^{sn\times (s+1)n}$ is called a \emph{$(s,n)$-wing pencil} or, simply, a wing pencil, if $K(\lambda)(\Lambda_s(\lambda)^T\otimes I_n)=0$.
\end{definition}

We characterize in Theorem \ref{thm:charac_dual-pencils} all the $(s,n)$-wing pencils that are dual minimal bases to the matrix polynomial $\Lambda_s(\lambda)\otimes I_n$. 
 A characterization of the wing pencils can also be found in \cite[Lemma 3]{Philip2016}, a work that was produced independently and simultaneously to our work.
\begin{theorem}\label{thm:charac_dual-pencils}
Let $n,s\in\mathbb{N}$, and let $L_s(\lambda)$ and $\Lambda_s(\lambda)$ be the matrix polynomials in \eqref{eq:Ls} and \eqref{Lambda}, respectively.
Then, a $(s,n)$-wing  pencil $K(\lambda)\in\mathbb{F}[\lambda]^{sn\times (s+1)n}$ is a dual minimal basis to $\Lambda_s(\lambda)\otimes I_n$ if and only if $K(\lambda)=B(L_s(\lambda)\otimes I_n)$ for some nonsingular matrix $B\in\mathbb{F}^{sn\times sn}$.
\end{theorem}
\begin{proof}
Assume that $K(\lambda)=B(L_s(\lambda)\otimes I_n)$ for some nonsingular matrix $B\in\mathbb{F}^{sn\times sn}$.
Since the condition $K(\lambda)(\Lambda_s(\lambda)^T\otimes I_n)=0$  holds, the pencil $K(\lambda)$ is a $(s,n)$-wing pencil.
Thus, we only need to prove that $K(\lambda)$ is a minimal basis.
To prove this, we will use Theorem \ref{minimal-basis}.
First, we have to show that $K(\lambda_0)$ has full row rank for any $\lambda_0\in\overline{\mathbb{F}}$.
Indeed, let $\lambda_0\in\overline{\mathbb{F}}$.
Since $L_s(\lambda)\otimes I_n$ is a minimal basis, we know that $L_s(\lambda_0)\otimes I_n$ has full row rank. 
Since the product of a nonsingular matrix by a full row rank matrix is a full row rank matrix, we obtain that $K(\lambda_0)=B(L_s(\lambda_0)\otimes I_n)$ has full row rank. 
Now, we prove that the highest row degree coefficient matrix of $K(\lambda)$ has full row rank as well. 
Writting $L_s(\lambda) = \lambda F+E$, it is not difficult to check that the highest row degree coefficient matrix of $K(\lambda)$ is given by $B(F\otimes I_n)$.
The matrix $B$ is nonsingular and $F\otimes I_n$ has full row rank, therefore the matrix $B(F\otimes I_n)$ has full row rank.
From Theorem \ref{minimal-basis}, we conclude that $K(\lambda)$ is a minimal basis.
Thus, $K(\lambda)$ and $\Lambda_s(\lambda)\otimes I_n$ are dual minimal bases.

Assume now that $K(\lambda)$ and $\Lambda_s(\lambda)\otimes I_n$ are dual minimal bases.
We will show that $K(\lambda)=B(L_s(\lambda)\otimes I_n)$ for some nonsingular matrix $B$.
To this purpose, let us partition the pencil $K(\lambda)$ as a $s\times (s+1)$ block-pencil with blocks $\lambda [K_1]_{ij}+[K_0]_{ij} \in \mathbb{F}[\lambda]^{n\times n}$, for $i=1:s$ and $j=1:s+1$. 
Let
\[
\begin{bmatrix}
\lambda [K_1]_{i1} + [K_0]_{i1} & \lambda [K_1]_{i2} +  [K_0]_{i2} & \cdots & \lambda [K_1]_{i,s+1}+  [K_0]_{i,s+1}
\end{bmatrix}
\]
denote  the $i$th block-row of $K(\lambda)$.
From  $K(\lambda) (\Lambda_s(\lambda)^T \otimes I_n)=0$, we get
\begin{align*}
0=&\sum_{j=0}^s \lambda^j\left(\lambda [K_1]_{i,s+1-j}+[K_0]_{i,s+1-j}\right)\\
=&\lambda^{s+1} [K_1]_{i1} + \sum_{j=1}^s \lambda^j \left([K_0]_{i,s+1-j} + [K_1]_{i,s+2-j}\right) + [K_0]_{i,s+1},
\end{align*}
or equivalently,
\[
[K_1]_{i1}=0, \quad [K_1]_{ij}= - [K_0]_{i,j-1}, \quad  j=2:s+1, \quad [K_0]_{i,s+1}=0,
\]
which shows that $K(\lambda)$ is of the form
\[
-\begin{bmatrix}
[K_0]_{11} & \cdots & [K_0]_{1s} \\
\vdots & \ddots & \vdots \\
[K_0]_{s1} & \cdots & [K_0]_{ss}
\end{bmatrix}
\begin{bmatrix}
-I_n & \lambda I_n \\
 & \ddots & \ddots \\
 & & -I_n & \lambda I_n
\end{bmatrix}=:
B(L_s(\lambda)\otimes I_n).
\]
To finish the proof, we only need to show that $B$ is nonsingular.
The proof proceeds by contradiction.
Assume that $B$ is singular, i.e, there exists a nonzero vector $x\in\mathbb{F}^{sn\times 1}$ such that $x^TB=0$.
Let $\lambda_0\in\overline{\mathbb{F}}$.
Then, we have $x^T K(\lambda_0) = x^T B(L_s(\lambda_0)\otimes I_n)=0$.
Thus, $K(\lambda_0)$ does not have full row rank.
But this contradicts the fact that $K(\lambda)$ is a minimal basis.
Therefore, the matrix $B$ is nonsingular.
\end{proof}
\begin{remark}\label{rem:block_structure}
It is straightforward to check that a pencil $B(\lambda)\in\mathbb{F}[\lambda]^{(q+1)n\times (p+1)n}$  satisfies $B(\lambda)(\Lambda_p(\lambda)^T\otimes I_n)=0$  if and only it is of the form
\begin{equation}\label{gen-wing}
B(\lambda)=\left[ \begin{array}{ccccc} B_{11} & -\lambda B_{11} + B_{12} & -\lambda B_{12} + B_{13} & \cdots & -\lambda B_{1p}\\
B_{21} & -\lambda B_{21} + B_{22} & -\lambda B_{22} + B_{23} & \cdots & -\lambda B_{2p}\\
\vdots & \vdots & \vdots & \ddots & \vdots \\
B_{q+1,1} & -\lambda B_{q+1,1} + B_{q+1,2} & -\lambda B_{q+1,2} + B_{q+1,3} & \cdots & -\lambda B_{q+1,p}\\
\end{array} \right],
\end{equation}
for some matrices $B_{ij}\in \mathbb{F}^{n\times n}$, with  $i=1:q+1$ and $j=1:p$ (see, for example, {\rm \cite[Lemma 3]{Philip2016}}). 
In particular, 
any $(s,n)$-wing pencil has that form, as extracted from the proof of Theorem \ref{thm:charac_dual-pencils}. 
In what follows, any pencil as in (\ref{gen-wing}), with $q+1\neq p$, will be called \emph{generalized wing pencil}.
\end{remark}


As an immediate corollary of Theorem \ref{thm:charac_dual-pencils}, we obtain the following  result for minimal basis $(s,n)$-wing pencils, whose simple proof is omitted.
\begin{corollary}\label{cor:wing_pencils}
Let $M\in\mathbb{F}^{sn\times sn}$ be a nonsingular matrix.
If $K(\lambda)\in\mathbb{F}[\lambda]^{ns\times n(s+1)}$ is a minimal basis $(s,n)$-wing pencil, then $MK(\lambda)$ is also a minimal basis $(s,n)$-wing pencil.
\end{corollary}

The next theorem is the last auxiliary result about minimal basis wing pencils  that we present in this section.
\begin{theorem}\label{thm:minimal_bases_concatenation}
Let $n,s_1,s_2\in\mathbb{N}$,  and let $K(\lambda)$ and $L(\lambda)$ be, respectively, a minimal basis $(s_1,n)$-wing pencil and a minimal basis $(s_2,n)$-wing pencil.
If the pencils $K(\lambda)$ and $L(\lambda)$ are partitioned as follows
\[
K(\lambda)=\begin{bmatrix}
K_1(\lambda) & k(\lambda)
\end{bmatrix} \quad \mbox{and} \quad
L(\lambda)=\begin{bmatrix}
\ell(\lambda) & L_1(\lambda)
\end{bmatrix},
\]
where $K_1(\lambda)\in\mathbb{F}[\lambda]^{s_1n\times s_1n}$ and $L_1(\lambda)\in\mathbb{F}[\lambda]^{s_2n\times s_2n}$, then the pencil
\[
S(\lambda):=\begin{bmatrix}
K_1(\lambda) & k(\lambda) & 0 \\
0 & \ell(\lambda) & L_1(\lambda)
\end{bmatrix}
\]
is a minimal basis $(s_1+s_2,n)$-wing pencil.
\end{theorem}
\begin{proof}
Set $s:=s_1+s_2$.
By Theorem \ref{thm:charac_dual-pencils}, the pencils $K(\lambda)$ and $L(\lambda)$ can be written, respectively, as $K(\lambda)=B(L_{s_1}(\lambda)\otimes I_n)$ and $L(\lambda)=C(L_{s_2}(\lambda)\otimes I_n)$, for some nonsingular matrices $B\in\mathbb{F}^{s_1n\times s_1n}$ and $C\in\mathbb{F}^{s_2n\times s_2n}$.
Then, notice that the pencil $S(\lambda)$ can be written as
\[
S(\lambda) =
 \begin{bmatrix}
B & 0 \\ 0 & C
\end{bmatrix}
(L_{s}(\lambda)\otimes I_n).
\]
Thus, the desired result follows applying Theorem \ref{thm:charac_dual-pencils} to the pencil $S(\lambda)$.
\end{proof}

\bigskip

In future sections we will work with block-row and block-column permutations of
block-pencils, which will involve the use of block-permutation matrices. 
Next, we introduce the notation used in this paper for block-permutation matrices.
\begin{definition}\label{block-perm}
Let $k,n\in \mathbb{N}$. 
Let $\mathbf{c}=(c_1, c_2, \ldots, c_k)$ be a permutation of the set $\{1:k\}$. 
Then, we call the block-permutation matrix associated with $(\mathbf{c}, n)$, and denote it by $\Pi^n_{\mathbf{c}}$, the $k\times k$ block-matrix whose $(c_i,i)$th block-entry is $I_n$, for $i=1:k$, and having $0_n$ in every other block-entry. 
In particular, we denote by $\boldsymbol{\mathrm{id}}$ the identity permutation given by $\boldsymbol{\mathrm{id}}=(1:k)$. 
\end{definition}
\begin{example}
Let $k=4$ and let $\mathbf{c}=(2, 4,3, 1)$, which is a permutation of $\{1, 2, 3, 4\}$. Then, for every $n\in \mathbb{N}$, 
$$\Pi_{\mathbf{c}}^n = \left[ \begin{array}{cccc} 0 & 0 & 0 & I_n \\ I_n & 0 & 0 & 0\\ 0 & 0 & I_n & 0\\ 0 & I_n & 0 & 0\end{array} \right].$$
\end{example}

A particular block-permutation matrix that will be very useful in the proofs of our main results is the block version of the \emph{sip matrix} (standard involutory permutation).
 Such $sn\times sn$ block-permutation matrix is denoted by $R_{s,n}$ and is defined as follows
\begin{equation}\label{eq:sip}
R_{s,n}:=\begin{bmatrix}
0 & \cdots & I_n \\ \vdots & \iddots & \vdots \\I_n & \cdots & 0
\end{bmatrix}\in\mathbb{F}^{sn\times sn}.
\end{equation}

\begin{remark}
When the scalar $n$ is clear in the context, we will write $\Pi_{\mathbf{c}}$ and $R_s$  instead of $\Pi_{\mathbf{c}}^n$ and $R_{s,n}$, respectively, to simplify the notation.
\end{remark}

Finally, we introduce the concept of \emph{block-transposition} of a block-matrix.
Let $H=[H_{ij}]_{i=1:p, j=1:q}$ be a $p\times q$ block-matrix with block-entries $H_{ij}\in \mathbb{F}^{n\times n}$. 
We define the block-transpose $H^{\mathcal{B}}$ of $H$ as the $q\times p$ block-matrix having the block $H_{ji}$ in the  block-position $(i, j)$, for $1\leq i \leq p$ and $1\leq j \leq q$. 

\section{The block minimal bases pencils framework for Fiedler-like pencils}\label{sec:block_minimal_bases_pencils}

In \cite{canonical}, Fiedler pencils were expressed as block Kronecker pencils, a particular type of block minimal bases pencils. 
In this section we recall the definition of block minimal bases  pencils and introduce the family of extended block Kronecker pencils, which contains the family of block Kronecker pencils. 
The pencils in this new family are block minimal bases pencils under some nonsingularity conditions and allow us to express all Fiedler-like pencils   into the block minimal bases framework.

\subsection{Block minimal bases pencils}

We discuss in this section  the family of block minimal bases pencils, recently introduced in \cite{canonical}, and its main properties.

\begin{definition}{\rm \cite[Definition 27]{canonical}}
A matrix pencil 
\begin{equation}\label{BMBP}
C(\lambda)=\left[ \begin{array}{c|c} M(\lambda) & K_2(\lambda)^T\\ \hline K_1(\lambda) & 0 \end{array} \right]
\end{equation}
is called a \emph{block minimal bases pencil} if $K_1(\lambda)$ and $K_2(\lambda)$  are both minimal bases. If, in addition, the row degrees of $K_1(\lambda)$ are all equal to 1, the row degrees of $K_2(\lambda)$ are all equal to 1, the row degrees of a minimal basis dual to $K_1(\lambda)$ are all equal, and the row degrees of a minimal basis dual to $K_2(\lambda)$ are all equal, then $C(\lambda)$ is called a \emph{strong block minimal bases pencil}. 
\end{definition}

In the following theorem, which will be used in the proofs of some of our results, we show that every block minimal bases pencil is a linearization of at least one matrix polynomial.

\begin{theorem}\label{linearization}{\rm \cite[Theorem 3.3]{canonical}}
Let $K_1(\lambda)$ and $N_1(\lambda)$ be a pair of dual minimal bases, and let $K_2(\lambda)$ and $N_2(\lambda)$ be another pair of dual minimal bases. Let $C(\lambda)$ be a block minimal bases pencil as in (\ref{BMBP}).
 Consider the matrix polynomial 
\begin{equation}\label{eq:Qpolinminbaslin}
Q(\lambda):=N_2(\lambda) M(\lambda) N_1(\lambda)^T.\end{equation}
Then:
\begin{itemize} 
\item[\rm(i)] $C(\lambda)$ is a linearization of $Q(\lambda)$.
\item[\rm(ii)] If $C(\lambda)$ is a strong block minimal bases pencil, then $C(\lambda)$  is a strong linearization of $Q(\lambda)$, considered as a polynomial with grade  $1 + \deg(N_1(\lambda))+\deg(N_2(\lambda))$. 
\end{itemize}
\end{theorem}

Next we show the relationship between the left and right minimal indices of the matrix polynomial $Q(\lambda)$ in
\eqref{eq:Qpolinminbaslin} and those of any of its strong block minimal bases pencil linearizations.
\begin{theorem} {\rm \label{thm:indicesminbaseslin}\cite[Theorem 3.7]{canonical}}
Let $C(\lambda)$ be a strong block minimal bases pencil as in \eqref{BMBP}, let $N_1(\lambda)$ be a minimal basis dual to $K_1 (\lambda)$, let $N_2(\lambda)$ be a minimal basis dual to $K_2 (\lambda)$, and let $Q(\lambda)$ be the matrix polynomial defined in \eqref{eq:Qpolinminbaslin}. 
Then the following hold:
\begin{enumerate}
\item[\rm (a)] If $0 \leq \epsilon_1 \leq \epsilon_2 \leq \cdots \leq \epsilon_s$ are the right minimal indices of $Q(\lambda)$, then
\[
\epsilon_1 + \deg(N_1(\lambda)) \leq \epsilon_2 + \deg(N_1(\lambda)) \leq \cdots \leq \epsilon_s +  \deg(N_1(\lambda))
\]
are the right minimal indices of $C(\lambda)$.
\item[\rm (b)] If $0 \leq \eta_1 \leq \eta_2 \leq \cdots \leq \eta_r$ are the left minimal indices of $Q(\lambda)$, then
\[
\eta_1 + \deg(N_2(\lambda)) \leq \eta_2 + \deg(N_2(\lambda)) \leq \cdots \leq \eta_r +  \deg(N_2(\lambda))
\]
are the left minimal indices of $C(\lambda)$.
\end{enumerate}
\end{theorem}

\subsection{Block Kronecker pencils}

Next, we recall the subfamily of block minimal bases pencils called block Kronecker pencils, which was introduced in \cite{canonical} and  has proven to be fruitful in providing a simple block-structure characterization of Fiedler pencils up to permutation of rows and columns. 
\begin{definition} \label{def:Kronecker_pencils}{\rm \cite[Definition 5.2]{canonical}}
 Let $L_s (\lambda)$ be the matrix pencil defined in \eqref{eq:Ls} and let $M(\lambda)$ be an arbitrary pencil. 
Then any matrix pencil of the form
\begin{equation}
  \label{eq:linearization_general}
  \begin{array}{cl}
  C(\lambda)=
  \left[
    \begin{array}{c|c}
      M(\lambda) &L_{q}(\lambda)^{T}\otimes I_{m}\\\hline
      L_{p}(\lambda)\otimes I_{n}&0
      \end{array}
    \right]&
    \begin{array}{l}
      \left. \vphantom{L_{\mu}^{T}(\lambda)\otimes I_{m}} \right\} {\scriptstyle (q+1)m}\\
      \left. \vphantom{L_{\epsilon}(\lambda)\otimes I_{n}}\right\} {\scriptstyle p n}
    \end{array}\\
    \hphantom{C(\lambda)=}
    \begin{array}{cc}
      \underbrace{\hphantom{L_{\epsilon}(\lambda)\otimes I_{n}}}_{(p+1)n}&\underbrace{\hphantom{L_{\mu}^{T}(\lambda)\otimes I_{m}}}_{q m}
    \end{array}
  \end{array}
  \>,
\end{equation}
is called a $(p,n,q,m)$-block Kronecker pencil or, simply, a {\em block Kronecker pencil}.
\end{definition}

Fiedler pencils are, modulo permutations, block Kronecker pencils (we recall this result in Theorem \ref{FP-col-row}).
In Section \ref{sec:GF_as_block_Kron_pen}, we will show that this result is also true for  the most important subfamily of generalized Fiedler pencils (i.e., proper generalized Fiedler pencils).
However, we will see that this family is not large enough to include the other families of Fiedler-like pencils (Fiedler pencils with repetition and generalized Fiedler pencils with repetition).
 To provide a block minimal bases pencils framework large enough to include all families of Fiedler-like pencils is the goal of the following section, where we introduce the family of \emph{extended block Kronecker pencils}.
 These pencils are block
minimal bases pencils under some generic nonsingularity conditions.

\subsection{Extended block Kronecker pencils and the AS condition}\label{sec:Lambda-dual_pencils}
Here we introduce a new  family of pencils named \emph{extended block Kronecker pencils}, such that most of its members are strong block minimal bases pencils.
As we will show, this family of pencils provides a unified approach to all Fielder-like pencils (Fiedler pencils, generalized Fiedler pencils, Fiedler pencils with repetition and generalized Fiedler pencils with repetition) in the sense that every Fiedler-like pencil is, up to permutations, an extended block
Kronecker pencil. 
 This family was simultaneously introduced in \cite[Definition 1, Theorem 4]{Philip2016}. In this section we point out the connections between the theory that we develop here and the results that were obtained independently in \cite{Philip2016}.

\begin{definition}\label{def:Lambda-dual-pencils}
Let $K_1(\lambda)$ and $K_2(\lambda)$ be, respectively, a $(p,n)$-wing pencil and a $(q,m)$-wing pencil (recall Definition \ref{def:wing_pencil}), and let $M(\lambda)$ be an arbitrary pencil.
A matrix pencil
\begin{equation}\label{LDP}
  \begin{array}{cl}
  C(\lambda)=
  \left[
    \begin{array}{c|c}
      M(\lambda) &K_2(\lambda)^T\\\hline
      K_1(\lambda)&0
      \end{array}
    \right]&
    \begin{array}{l}
      \left. \vphantom{K_2(\lambda)^T} \right\} {\scriptstyle (q+1)m}\\
      \left. \vphantom{K_1(\lambda)}\right\} {\scriptstyle p n}
    \end{array}\\
    \hphantom{C(\lambda)=}
    \begin{array}{cc}
      \underbrace{\hphantom{M(\lambda)}}_{(p+1)n}&\underbrace{\hphantom{K_2(\lambda)^Taa}}_{q m}
    \end{array}
  \end{array}
  \>,
\end{equation}
is called an extended \emph{$(p,n,q,m)$-block Kronecker pencil}, or, simply, an \emph{extended block Kronecker pencil}.
Moreover, the block $M(\lambda)$ is called the \emph{body} of $C(\lambda)$.  
\end{definition}

\begin{remark}
In Definition \ref{def:Lambda-dual-pencils}, we allow the border cases $p=0$ and $q=0$, i.e., pencils of the form 
\[
C(\lambda)=
\left[\begin{array}{c|c}
M(\lambda) & K(\lambda)^T 
\end{array}\right] \quad \mbox{or} \quad 
C(\lambda)=
\left[\begin{array}{c}
M(\lambda) \\ \hline
K(\lambda)
\end{array}\right],
\]
where $K(\lambda)$ is a wing pencil, are also considered extended block Kronecker pencils.
\end{remark}

\begin{remark}
In view of Remark \ref{rem:block_structure} we notice that the extended block Kronecker pencil \eqref{LDP} can be factorized as
\begin{equation}\label{eq:extendedBKP}
\begin{bmatrix}
I_{(q+1)m} & 0 \\ 0 & A
\end{bmatrix}
\left[\begin{array}{c|c}
\lambda M_1+M_0 & L_q(\lambda)^T\otimes I_m \\ \hline
L_p(\lambda)\otimes I_n & 0
\end{array}\right]
\begin{bmatrix}
I_{(p+1)n} & 0 \\ 0 & B
\end{bmatrix},
\end{equation}
for some matrices $A\in\mathbb{F}^{np\times np}$ and $B\in\mathbb{F}^{mq\times mq}$.
Taking $A=I_{np}$ and $B=I_{mq}$ in \eqref{eq:extendedBKP} the extended block Kronecker pencil reduces to a block Kronecker pencil. 
Thus, we obtain that the family of $(p,n,q,m)$-block Kronecker pencils is included in the family of extended $(p,n,q,m)$-block Kronecker pencils.
 The factorization (\ref{eq:extendedBKP}) is equivalent to (18) in \cite{Philip2016}.
\end{remark}


\begin{remark}\label{rem:different_partitions}
When $n=m$,  a pencil may be partitioned in more than one way as an extended $(p,n,q,n)$-block Kronecker pencil.
For example, let $A_0,A_1,A_2,A_3$ be arbitrary $n\times n$ matrices.
 The pencil 
$$C(\lambda)= \left[ \begin{array}{ccc} \lambda A_3+A_2 & A_1 & A_0 \\ A_1 & -\lambda A_1+A_0 & -\lambda A_0 \\ A_0 & -\lambda A_0 & 0 \end{array} \right]$$
can be seen as an extended $(1,n,1,n)$-block Kronecker pencil since it can be partitioned as follows
$$\left[ \begin{array}{cc|c} \lambda A_3+A_2 & A_1 & A_0 \\ A_1 & -\lambda A_1+A_0 & -\lambda A_0 \\ \hline A_0 &- \lambda A_0 & 0 \end{array} \right].$$
But notice that we could also partition $C(\lambda)$ in the next two alternative ways
$$\left[ \begin{array}{c|cc} \lambda A_3+A_2 & A_1 & A_0 \\ A_1 & -\lambda A_1+A_0 & -\lambda A_0 \\ A_0 &- \lambda A_0 & 0 \end{array} \right], \quad \left[ \begin{array}{ccc} \lambda A_3+A_2 & A_1 & A_0 \\ \hline A_1 &- \lambda A_1+A_0 &- \lambda A_0 \\ A_0 &- \lambda A_0 & 0 \end{array} \right].$$
Thus, the pencil $C(\lambda)$ can also be seen as an extended $(0,n,2,n)$-block Kronecker pencil or as an extended $(2,n,0,n)$-block Kronecker pencil.
Some consequences of this ambiguity are studied in Theorem \ref{thm:different_partitions} and Remark \ref{rem:only_for_regular}. 
 The  phenomenon presented in this example is called Superpartition Principle in \cite{Philip2016}.
\end{remark}

As an immediate corollary of Theorems \ref{linearization} and \ref{thm:indicesminbaseslin}, we obtain that, under some generic nonsingularity conditions, extended block Kronecker pencils are strong linearizations of some matrix polynomials, and that the minimal indices of those polynomials and those of the pencils are related by uniform shifts. 
 Part (a) of the following theorem corresponds with Theorem 6 in \cite{Philip2016}.
\begin{theorem}\label{thm:linearization_Lambda}
Let $C(\lambda)$ be an extended block Kronecker pencil as in \eqref{eq:extendedBKP},  with body  $M(\lambda)$. 
Let  
\begin{equation}\label{eq:Qpoli_Lambda}
Q(\lambda):= (\Lambda_q(\lambda)\otimes I_m)M(\lambda)(\Lambda_p(\lambda)^T\otimes I_n)\in\mathbb{F}[\lambda]^{m\times n},
\end{equation}
viewed as a matrix polynomial with grade $p+q+1$.
If $A$ and $B$ are nonsingular, then the following statements hold:
\begin{enumerate}
\item[\rm (a)] The pencil $C(\lambda)$ is a strong linearization of $Q(\lambda)$.
\item[\rm (b1)] If $0\leq \epsilon_1\leq \epsilon_2\leq \cdots \leq \epsilon_s$ are the right minimal indices of $Q(\lambda)$, then
\[
\epsilon_1+p\leq \epsilon_2+p\leq \cdots \leq \epsilon_s+p
\]
are the right minimal indices of $C(\lambda)$.
\item[\rm (b2)] If $0\leq \eta_1\leq \eta_2\leq \cdots \leq \eta_r$ are the left minimal indices of $Q(\lambda)$, then
\[
\eta_1+q\leq \eta_2+q\leq \cdots \leq \eta_r+q
\]
are the left minimal indices of $C(\lambda)$.
\end{enumerate}
\end{theorem}

Theorem \ref{thm:linearization_Lambda} shows that most of the extended block Kronecker pencils as in (\ref{eq:extendedBKP}) are strong linearizations of  the associated  matrix polynomial $Q(\lambda)$ since the nonsingularity conditions on $A$ and $B$ are generic in $\mathbb{F}^{np\times np}\times \mathbb{F}^{qm\times qm}$.

Now, we address the inverse problem, that is, given a matrix polynomial $P(\lambda)$, how to construct extended block Kronecker pencils that are strong linearizations of $P(\lambda)$.
With this goal in mind, we introduce next some useful concepts. 
\begin{definition}\label{AS-def}
Let  $M(\lambda)=\lambda M_1 + M_0\in \mathbb{F}[\lambda]^{(q+1)m \times (p+1)n}$ be a matrix pencil and set $k:=p+q+1$. 
Let us denote by $[M_0]_{ij}$ and $[M_1]_{ij}$ the $(i,j)$th block-entries of $M_0$ and $M_1$, respectively, when $M_1$ and $M_0$ are partitioned as $(q+1)\times (p+1)$ block-matrices with blocks of size $m\times n$. 
We call the \emph{antidiagonal sum of $M(\lambda)$ related to $s\in \{ 0:k\}$} the matrix
\[
\mathrm{AS}(M, s):=\sum_{i+j=k+2-s} [M_1]_{ij} +\sum_{i+j=k+1-s} [M_0]_{ij}.
\]
Additionally, given a matrix polynomial $P(\lambda)=\sum_{i=0}^k A_i \lambda^i \in\mathbb{F}[\lambda]^{m\times n}$, we say that $M(\lambda)$ satisfies the \emph{antidiagonal sum condition (AS condition)} for $P(\lambda)$ if 
\begin{equation}\label{AS-condition}
\mathrm{AS}(M, s)= A_s, \quad s=0:k.
\end{equation}
\end{definition}


The main result of this section is Theorem \ref{thm:Lambda-dual-pencil-linearization}, where we give sufficient conditions on the body  of an extended block Kronecker pencil as in \eqref{eq:extendedBKP} with $A,B$ nonsingular, to be a strong linearization of a given matrix polynomial. 
 It is not hard to see that this result also follows from Theorems 4 and 6 in \cite{Philip2016}.
\begin{theorem}\label{thm:Lambda-dual-pencil-linearization}
Let $P(\lambda)=\sum_{i=0}^k A_i\lambda^i\in\mathbb{F}[\lambda]^{m\times n}$, and let $C(\lambda)$ be an extended $(p,n,q,m)$-block Kronecker pencil as in \eqref{eq:extendedBKP} with $p+q+1=k$ and $A,B$ nonsingular.
If the body $M(\lambda)\in\mathbb{F}[\lambda]^{(q+1)m\times (p+1)n}$ of $C(\lambda)$ satisfies the AS condition for $P(\lambda)$, then $C(\lambda)$ is a strong linearization of $P(\lambda)$, the right minimal indices of $C(\lambda)$ are those of $P(\lambda)$ shifted by $p$, and the left minimal indices of $C(\lambda)$ are those of $P(\lambda)$ shifted by $q$.
\end{theorem}
\begin{proof}
By Theorem \ref{thm:linearization_Lambda}, the pencil $C(\lambda)$ is a strong linearization of \eqref{eq:Qpoli_Lambda}.
Some simple algebraic manipulations show that the AS condition implies $Q(\lambda)=P(\lambda)$. 
Moreover, the results for the minimal indices just follow from parts (b1) and (b2) in Theorem \ref{thm:linearization_Lambda}.
\end{proof}

Our main focus in future sections will be on block-pencils whose blocks are square.
In this case, surprisingly, given an extended block Kronecker pencil $C(\lambda)$ whose body satisfies the AS condition for a given matrix polynomial $P(\lambda)\in\mathbb{F}[\lambda]^{n\times n}$, the whole pencil $C(\lambda)$ satisfies also the AS condition but for a shifted version of $P(\lambda)$, as we show next.  
\begin{theorem}\label{CAS}
Let $P(\lambda)= \sum_{i=0}^k A_i \lambda^i \in\mathbb{F}[\lambda]^{n\times n}$  and let $C(\lambda)\in \mathbb{F}[\lambda]^{kn \times kn}$ be an extended block Kronecker pencil as in \eqref{LDP} with $k=p+q+1$ whose body satisfies the AS condition for $P(\lambda)$. 
Then, the pencil $C(\lambda)$ satisfies the AS condition for $\lambda^{k-1} P(\lambda)$, that is, 
$$\mathrm{AS}(C, s)=  B_s,\quad s=0:2k-1,$$
where $B_s=0$, for $s=0:k-2$, and $B_s = A_{s-k+1} $, for $s= k-1:2k-1$. 
\end{theorem}
\begin{proof}
Let us write $C(\lambda)=\widetilde{M}(\lambda)+\widetilde{K}_1(\lambda)+\widetilde{K}_2(\lambda)^T$, where
\[
 \widetilde{M}(\lambda):= \left[ \begin{array}{cc} M(\lambda) & 0 \\ 0 & 0 \end{array} \right], \quad 
\widetilde{K}_2(\lambda)^T: = \left[ \begin{array}{cc} 0 & K_2(\lambda)^T \\ 0 & 0 \end{array} \right], \quad 
\widetilde{K}_1(\lambda) := \left[ \begin{array}{cc} 0 & 0 \\ K_1(\lambda) & 0 \end{array} \right], 
\]
Then, due to the linearity of the antidiagonal sum $\mathrm{AS}(\cdot,s)$, we get
\begin{align*}
\mathrm{AS}(C, s) & = \mathrm{AS}(\widetilde{M}, s) + \mathrm{AS}(\widetilde{K}_2^T, s)+ AS(\widetilde{K}_1, s), \quad s=0:2k-1.
\end{align*}
Additionally, because of the  block-structure of the wing pencils explained in Remark \ref{rem:block_structure}, it is clear that   $\mathrm{AS}(\widetilde{K}_1, s)=0$ and $ \mathrm{AS}(\widetilde{K}_2^T, s)=0$, for all $s$. 
Thus, $\mathrm{AS}(C, s)= \mathrm{AS}(\widetilde{M}, s)$ and the result follows taking into account that $M(\lambda)$ satisfies the AS condition for $P(\lambda)$.
\end{proof}

We pointed out in Remark \ref{rem:different_partitions} that a given pencil  $L(\lambda)$ may be partitioned in more that one way as an extended $(p,n,q,n)$-block Kronecker  pencil.
The next result shows that the AS condition of the body is preserved under different representations of $L(\lambda)$ as an extended block Kronecker pencil.  This result corresponds to Theorem 22 in \cite{Philip2016}.
\begin{theorem}\label{thm:different_partitions}
Let $P(\lambda)=\sum_{i=0}^k A_i \lambda^i \in\mathbb{F}[\lambda]^{n\times n}$  and let $C(\lambda)\in \mathbb{F}[\lambda]^{kn \times kn}$ be an extended block Kronecker pencil as in \eqref{LDP} with $k=p+q+1$.
If the body of $C(\lambda)$ satisfies the AS condition for $P(\lambda)$, then the body of any other partition of $C(\lambda)$ as an extended block Kronecker pencil satisfies the AS condition for $P(\lambda)$ as well.
\end{theorem}
\begin{proof}
The result is an immediate  consequence of Theorem \ref{CAS}.
Assume that the extended block Kronecker pencil $C(\lambda)$ can be partitioned as in \eqref{LDP} and  also  as follows
\[
C(\lambda)=\left[ \begin{array}{c|c} \widehat{M}(\lambda) & \widehat{K}_2(\lambda)^T \\ \hline \widehat{K}_1(\lambda) & \phantom{\Big{(}} 0 \phantom{\Big{(}} \end{array} \right],
\]
where  $\widehat{K}_1(\lambda)\in\mathbb{F}[\lambda]^{n\widehat{p}\times n(\widehat{p}+1)}$ and $\widehat{K}_2(\lambda)\in\mathbb{F}[\lambda]^{n\widehat{q}\times n(\widehat{q}+1)}$ are wing pencils, $\widehat{M}(\lambda)\in\mathbb{F}[\lambda]^{n(\widehat{q}+1)\times n(\widehat{p}+1)}$, and $\widehat{p}+\widehat{q}+1=k$.
Let
\[
\widetilde{M}(\lambda) := \left[ \begin{array}{cc} \widehat{M}(\lambda) & 0 \\ 0 & 0 \end{array} \right]. 
\]
By Theorem \ref{CAS} and the block-structure of wing pencils explained in Remark  \ref{rem:block_structure}, we obtain
\begin{equation}\label{ASCs}
\mathrm{AS}(C, s)= \mathrm{AS}(\widetilde{M}, s)= \left\{ \begin{array}{cc} 0,  & \textrm{for $s=0:k-2$}\\ 
A_{s+1-k}, & \textrm{for $s=k-1:2k-1$,} \end{array}\right.
\end{equation}
which shows that $\widehat{M}(\lambda)$ satisfies the AS condition for $P(\lambda)$.
 \end{proof}
\begin{remark}\label{rem:only_for_regular}
A pencil $C(\lambda)$ may be partitioned as an extended block Kronecker  pencil that is a strong linearization of  a given  matrix polynomial $P(\lambda)$ in more than one way only if  $P(\lambda)$ is regular.
Otherwise, the application of Theorem \ref{thm:Lambda-dual-pencil-linearization} to the different partitions would give contradictory values for the minimal indices of $C(\lambda)$.
\end{remark}

\begin{example}
Consider the pencil $C(\lambda)$ in Remark \ref{rem:different_partitions}. 
Note that $C(\lambda)$ is an extended block Kronecker pencil with a body satisfying the AS condition for  $P(\lambda)= \sum_{i=0}^3 A_i \lambda^i$ when $C(\lambda)$ is partitioned in any of the three ways presented in that remark. 
 Furthermore, note that the wing pencils in any of the three partitions of $C(\lambda)$ as an extended block Kronecker pencil are minimal bases if and only if $A_0$ is nonsingular.
Therefore, they are strong linearizations of $P(\lambda)$ provided that its trailing coefficient $A_0$ is nonsingular.
In this case, notice that $P(\lambda)$ is regular, in accordance with Remark \ref{rem:only_for_regular}.
\end{example}

The last result in this section is Lemma \ref{lemm:perturbation_body}.
 It shows that, if a pencil satisfying the AS condition for some  matrix polynomial $P(\lambda)$  is perturbed by the addition of a generalized wing pencil, the AS condition is preserved.
\begin{lemma}\label{lemm:perturbation_body}
Let $P(\lambda)=\sum_{i=0}^k A_i\lambda^i \in\mathbb{F}[\lambda]^{n\times n}$, let $p+q+1=k$, let $M(\lambda)\in\mathbb{F}[\lambda]^{(q+1)n\times (p+1)n}$ be a matrix pencil satisfying the AS condition for $P(\lambda)$, and let $B(\lambda)\in\mathbb{F}[\lambda]^{(q+1)n\times (p+1)n}$ be a generalized wing pencil. 
Then, the matrix pencil $M(\lambda)+B(\lambda)$ satisfies the AS condition for $P(\lambda)$.
\end{lemma}
\begin{proof}
Let $\widetilde{M}(\lambda):=M(\lambda)+B(\lambda)$. The block-structure of generalized wing pencils explained in Remark \ref{rem:block_structure}
 implies that $\mathrm{AS}(B,s)=0$, for $s=0:k$.
Therefore, by the linearity of the antidiagonal sum, we get $\mathrm{AS}(\widetilde{M},s)=\mathrm{AS}(M,s)+\mathrm{AS}(B,s)=\mathrm{AS}(M,s)=A_s$.
Thus, the pencil $\widetilde{M}(\lambda)$ satisfies the AS condition for $P(\lambda)$.
\end{proof}


\subsection{Illustrations and informal statements of the main results for Fiedler-like pencils}\label{sec:informal}

The goal of this section is twofold: first, to provide some examples that illustrate the main results of this paper, and, second, to state informally these results in Theorems \ref{thm:informal1} and \ref{thm:informal2}.
For the impatient reader, the precise statements of Theorems \ref{thm:informal1} and \ref{thm:informal2} can be found in Theorems \ref{FP-col-row}, \ref{thm:main:GFP} and  \ref{thm:main_GFPR}, however they require the index tuple notation that will be introduced in Section \ref{sec:index_tuples}.

Let us start by considering the Fiedler  pencil
\[
F_\mathbf{q}(\lambda)=\begin{bmatrix}
\lambda A_6+A_5 & -I_n & 0 & 0 & 0 & 0 \\
A_4 & \lambda I_n & A_3 & -I_n & 0 & 0 \\
-I_n & 0 & \lambda I_n & 0 & 0 & 0 \\
0 & 0 & A_2 & \lambda I_n & A_1 & -I_n \\
0 & 0 & -I_n & 0 & \lambda I_n & 0 \\
0 & 0 & 0 & 0 & A_0 & \lambda I_n
\end{bmatrix}
\]
 associated with a matrix polynomial $P(\lambda)=\sum_{i=0}^6 A_i\lambda^i\in\mathbb{F}[\lambda]^{n\times n}$ (this Fiedler pencil will be formally introduced in Example \ref{ex:FP}).
 Then, it is not difficult to see that we can transform $F_{\mathbf{q}}(\lambda)$ into a block Kronecker pencil via block-row and block- column permutations, given by  permutation matrices $\Pi_{\boldsymbol\ell_1}^n$ and $\Pi_{\mathbf{r}_1}^n$, to obtain
 \[
 (\Pi_{\boldsymbol\ell_1}^n)^\mathcal{B}F_\mathbf{q}(\lambda)\Pi_{\mathbf{r}_1}^n=
\left[\begin{array}{ccc|ccc}
\lambda A_6+A_5 & 0 & 0 & -I_n & 0 & 0 \\
A_4 & A_3 & 0 & \lambda I_n & -I_n & 0 \\
0 & A_2 & A_1 & 0 & \lambda I_n & -I_n \\
0 & 0 & A_0 & 0 & 0 & \lambda I_n \\ \hline
-I_n & \lambda I_n & 0 & 0 & 0 & 0 \\
0 & -I_n & \lambda I_n & 0 & 0 & 0
\end{array}\right]
 \]
 By Theorem \ref{thm:linearization_Lambda}, the above block Kronecker pencil is a strong linearization of
 \[
 \begin{bmatrix}
 \lambda^3 I_n & \lambda^2 I_n & \lambda I_n & I_n
 \end{bmatrix}
 \begin{bmatrix}
 \lambda A_6+A_5 & 0 & 0 \\
 A_4 & A_3 & 0 \\
 0 & A_2 & A_1 \\
 0 & 0 & A_0 
 \end{bmatrix}
 \begin{bmatrix}
 \lambda^2 I_n \\ \lambda I_n \\ I_n
 \end{bmatrix} = P(\lambda).
 \]
 Thus, this is an example of a Fiedler pencil associated with a matrix polynomial $P(\lambda)$ which is block-permutationally equivalent to a block Kronecker pencil. Note that  $ (\Pi_{\boldsymbol\ell_1}^n)^\mathcal{B}F_\mathbf{q}(\lambda)\Pi_{\mathbf{r}_1}^n $ is a strong linearization of $P(\lambda)$.
This result has been shown to be true for any Fiedler pencil \cite[Theorem 4.5]{canonical} (see, also, Theorem \ref{FP-col-row}).
Let us consider, now, the following proper generalized Fiedler pencil
\[
K_{\mathbf{q},\mathbf{z}}(\lambda)=
\begin{bmatrix}
-I_n & \lambda A_6 & 0 & 0 & 0 & 0 \\
\lambda I_n & \lambda A_5+A_4 & -I_n & 0 & 0 & 0 \\
0 & A_3 & \lambda I_n & A_2 & -I_n & 0 \\
0 & -I_n & 0 & \lambda I_n & 0 & 0 \\
0 & 0 & 0 & -I_n & 0 & \lambda I_n \\
0 & 0 & 0 & 0 & \lambda I_n & \lambda A_1+A_0
\end{bmatrix},
\]
associated also with the polynomial $P(\lambda)$ (this pencil will be formally introduced in Example \ref{ex:GFP}).
We can transform $K_{\mathbf{q},\mathbf{m}}(\lambda)$ into a block Kronecker pencil via block-row and block-column permutations, given by block-permutation matrices $\Pi_{\boldsymbol\ell_2}^n$ and $\Pi_{\mathbf{r}_2}^n$, to obtain
\[
(\Pi_{\boldsymbol\ell_2}^n)^\mathcal{B} K_{\mathbf{q},\mathbf{z}}(\lambda) \Pi_{\mathbf{r}_2}^n =
\left[\begin{array}{ccc|ccc}
\lambda A_6 & 0 & 0 & -I_n & 0 & 0 \\
\lambda A_5+A_4 & 0 & 0 & \lambda I_n & -I_n & 0 \\
A_3 & A_2 & 0 & 0 & \lambda I_n & -I_n \\
0 & 0 & \lambda A_1+A_0 & 0 & 0 & \lambda I_n \\ \hline 
-I_n & \lambda I_n & 0 & 0 & 0 & 0 \\
0 & -I_n & \lambda I_n & 0 & 0 & 0
\end{array}\right],
\]
which, by Theorem \ref{thm:linearization_Lambda}, is a strong linearization of
 \[
 \begin{bmatrix}
 \lambda^3 I_n & \lambda^2 I_n & \lambda I_n & I_n
 \end{bmatrix}
 \begin{bmatrix}
 \lambda A_6 & 0 & 0 \\
 \lambda A_5+A_4 & 0 & 0 \\
 A_3 & A_2 & 0 \\
 0 & 0 & \lambda A_1+A_0 
 \end{bmatrix}
 \begin{bmatrix}
 \lambda^2 I_n \\ \lambda I_n \\ I_n
 \end{bmatrix} = P(\lambda).
 \]
  Therefore, this is an example of a proper generalized Fiedler pencil associated with a matrix polynomial $P(\lambda)$ which is block-permutationally equivalent to a block Kronecker pencil that is a strong linearization of $P(\lambda)$.
  This result turns out to be true for all (proper) generalized Fiedler pencils (see Section \ref{sec:GF_as_block_Kron_pen}).
  Since Fiedler pencils are particular examples of proper generalized Fiedler pencils, we can state the following theorem.
\begin{theorem}\label{thm:informal1}
Let $L(\lambda)$ be a proper generalized Fiedler pencil associated with $P(\lambda)$.
 Then, up to permutations, $L(\lambda)$ is a block Kronecker pencil which is a strong linearization of $P(\lambda)$.
\end{theorem}

Let us consider, finally, the following generalized Fiedler pencil with repetition (in particular, it is a Fiedler pencil with repetition)
\[
F_P(\lambda) =
\begin{bmatrix}
0 & 0 & 0 & 0 & -A_6 & \lambda A_6 \\
0 & 0 & 0 & -A_6 & \lambda A_6-A_5 & \lambda A_5 \\
0 & 0 & -A_6 & \lambda A_6-A_5 & \lambda A_5-A_4 & \lambda A_4\\
0 & -A_6 & \lambda A_6-A_5 & \lambda A_5 -A_4 & \lambda A_4-A_3 & \lambda A_3 \\
-A_6 & \lambda A_6-A_5 & \lambda A_5-A_4 & \lambda A_4-A_3 & \lambda A_3-A_2 & \lambda A_2 \\
\lambda A_6 & \lambda A_5 & \lambda A_4 & \lambda A_3 & \lambda A_2 & \lambda A_1+A_0
\end{bmatrix}
\]
associated with the matrix polynomial $P(\lambda)=\sum_{i=0}^6 A_i\lambda^i\in\mathbb{F}[\lambda]^{n\times n}$.
It is clear that there are  no block-row and block-column permutations that transform the above pencil into a block Kronecker pencil.
However, it is not difficult to show that there exist block-permutation matrices, denoted by $\Pi_{\boldsymbol\ell_3}^n$ and $\Pi_{\mathbf{r}_3}^n$, such that 
\begin{align*}
(\Pi_{\boldsymbol\ell_3}^n)^\mathcal{B} L_P(\lambda)\Pi_{\mathbf{r}_3}^n &=
\left[\begin{array}{ccc|ccc}
\lambda A_6-A_5 & \lambda A_5-A_4 & \lambda A_4 & -A_6 & 0 & 0 \\
\lambda A_5-A_4 & \lambda A_4-A_3 & \lambda A_3 & \lambda A_6-A_5 & - A_6 & 0 \\
\lambda A_4-A_3 & \lambda A_3-A_2 & \lambda A_2 & \lambda A_5-A_4 & \lambda A_6-A_5 & -A_6 \\
\lambda A_3 & \lambda A_2 & \lambda A_1+A_0 & \lambda A_4 & \lambda A_5 & \lambda A_6 \\ \hline
-A_6 & \lambda A_6-A_5 & \lambda A_5 & 0 & 0 & 0 \\
0 & -\lambda A_6 & \lambda A_6 & 0 & 0 & 0
\end{array}\right]\\
&=:\left[\begin{array}{c|c}
M(\lambda) & K_2(\lambda)^T \\ \hline
K_1(\lambda) & 0 
\end{array}\right].
\end{align*}
Additionally, notice that 
\[
K_1(\lambda) = 
\begin{bmatrix}
-A_6 & \lambda A_6-A_5 & \lambda A_5 \\ 0 & -A_6 & \lambda A_6
\end{bmatrix}=
\begin{bmatrix}
A_6 & A_5 \\ 0 & A_6
\end{bmatrix}
\begin{bmatrix}
-I_n & \lambda I_n & 0 \\ 0 & -I_n & \lambda I_n
\end{bmatrix}
\]
and 
\[
K_2(\lambda)=
\begin{bmatrix}
-A_6 & 0 & 0 \\
\lambda A_6-A_5 & - A_6 & 0 \\
\lambda A_5-A_4 & \lambda A_6-A_5 & -A_6 \\
\lambda A_4 & \lambda A_5 & \lambda A_6
\end{bmatrix}^T=
\left(
\begin{bmatrix}
-I_n & 0 & 0 \\
\lambda I_n & -I_n & 0 \\
0 & \lambda I_n & -I_n \\
0 & 0 & \lambda I_n
\end{bmatrix}
\begin{bmatrix}
A_6 & 0 & 0 \\
A_5 & A_6 & 0 \\
A_4 & A_5 & A_6
\end{bmatrix}
\right)^T.
\]
Therefore, if the leading coefficient $A_6$ of $P(\lambda)$ is nonsingular, by Theorem \ref{thm:charac_dual-pencils}, the pencils $K_1(\lambda)$ and $K_2(\lambda)$ are both minimal  basis wing pencils.
Thus, the pencil $(\Pi_{\boldsymbol\ell_3}^n)^\mathcal{B} L_P(\lambda)\Pi_{\mathbf{r}_3}^n$ is an extended block Kronecker pencil which, by Theorem \ref{thm:linearization_Lambda}, is a strong linearization of 
\[
\begin{bmatrix}
 \lambda^3 I_n & \lambda^2 I_n & \lambda I_n & I_n
 \end{bmatrix}
\begin{bmatrix}
\lambda A_6-A_5 & \lambda A_5-A_4 & \lambda A_4 \\
\lambda A_5-A_4 & \lambda A_4-A_3 & \lambda A_3 \\
\lambda A_4-A_3 & \lambda A_3-A_2 & \lambda A_2 \\
\lambda A_3 & \lambda A_2 & \lambda A_1+A_0
\end{bmatrix}
 \begin{bmatrix}
 \lambda^2 I_n \\ \lambda I_n \\ I_n
 \end{bmatrix}
=P(\lambda).
\]
We conclude that  the pencil $L_P(\lambda)$ is not permutationally equivalent to a block Kronecker pencil, and it is only permutationally equivalent to a block minimal bases pencil under some nonsingularity conditions ($A_6$ nonsingular). 
In the latter case, though, $L_P(\lambda)$ is an extended block Kronecker pencil, which is a strong linearization of $P(\lambda)$.
This result can be generalized as follows. 

\begin{theorem}\label{thm:informal2}
Let $L(\lambda)$ be a generalized Fiedler pencil with repetition associated with $P(\lambda)$.
Then, up to permutations, $L(\lambda)$ is an extended block Kronecker pencil which, under generic nonsingular conditions,  is a strong linearization of $P(\lambda)$.
\end{theorem}

The rest of the paper is dedicated to introduce the needed notation and auxiliary technical results to define all the families of  Fiedler-like pencils, to state in a precise way Theorems \ref{thm:informal1} and \ref{thm:informal2}, and to prove these results.

\section{Index tuples, elementary matrices and Fiedler-like pencils families} \label{sec:index_tuples}

In this section, we introduce the notation needed to define and to work in an effective way with the families of (square) Fiedler pencils, generalized Fiedler pencils, Fiedler pencils with repetition and generalized Fiedler pencils with repetition.
We also present some results that will be used to prove our main theorems.

\subsection{Index tuples, elementary matrices and matrix assignments}

The square Fiedler pencils as well as the generalized Fiedler pencils, Fiedler pencils with repetition and generalized Fiedler pencils with repetition were defined originally by expressing their matrix coefficients as products of the so-called \emph{elementary matrices}, whose definition we recall  next.
\begin{definition}{\rm \cite{GFPR}}
Let $k\geq 2$ be an integer and let $B$ be an arbitrary $n\times n$ matrix. 
We call \emph{elementary matrices} the  following  block-matrices partitioned into $k\times k$ blocks of size $n\times n$:
\begin{equation*}
\,M_{0}(B):=\left[
\begin{tabular}
[c]{c|c}%
$I_{(k-1)n}$ & $0$\\\hline
$0$ & $B$%
\end{tabular}
\ \right]  ,\quad M_{-k}(B):=\left[
\begin{tabular}
[c]{c|c}%
$B$ & $0$\\\hline
$0$ & $I_{(k-1)n}$%
\end{tabular}
\ \right]  , \label{m0-mk}%
\end{equation*}

\begin{equation}
M_{i}(B):=\left[
\begin{tabular}
[c]{c|cc|c}%
$I_{(k-i-1)n}$ & $0$ & $0$ & $0$\\\hline
$0$ & $B$ & $I_{n}$ & $0$\\
$0$ & $I_{n}$ & $0$ & $0$\\\hline
$0$ & $0$ & $0$ & $I_{(i-1)n}$%
\end{tabular}
\ \right]  ,\ \ \ i=1:k-1, \label{mis}%
\end{equation}
\[
M_{-i}(B):=\left[
\begin{tabular}
[c]{c|cc|c}%
$I_{(k-i-1)n}$ & $0$ & $0$ & $0$\\\hline
$0$ & $0$ & $I_{n}$ & $0$\\
$0$ & $I_{n}$ & $B$ & $0$\\\hline
$0$ & $0$ & $0$ & $I_{(i-1)n}$%
\end{tabular}\ \ \ \right] \ \ \ i=1:k-1, 
\]
 and
\[
M_{-0}(B)=M_0(B)^{-1} \quad \mbox{and} \quad M_{k}(B)=M_{-k}(B)^{-1}.
\]
Notice that  both $0$ and $-0$ are used with different meanings.
\end{definition}

\begin{remark}
We note that, for $i=1:k-1$, the matrices $M_i(B)$ and $M_{-i}(B)$ are nonsingular for any $B$. Moreover, $(M_i(B))^{-1}=M_{-i}(-B)$. 
On the other hand, the matrices $M_{0}(B)$ and $M_{-k}(B)$ are nonsingular if and only if $B$ is nonsingular.
\end{remark}

Given a  matrix polynomial $P(\lambda)=\sum_{i=0}^k A_i \lambda^i\in\mathbb{F}[\lambda]^{n\times n}$, we consider the following abbreviated notation:
$$M_i^P := M_i(-A_i), \quad i=0:k-1,$$
and
$$M_{-i}^P:=M_{-i}(A_i), \quad i=1:k.$$ 
When the polynomial $P(\lambda)$ is understood in the context, we will simply write $M_i$ and $M_{-i}$, instead of $M_i^P$ and $M_{-i}^P$ to simplify the notation.

\begin{remark}\label{commutativity}
It is easy to check that the commutativity relation
\begin{equation}
M_{i}(B_1)M_{j}(B_2)=M_{j}(B_2)M_{i}(B_1)\label{commut}%
\end{equation}
holds for any $n\times n$ matrices $B_1$ and $B_2$ if   $||i|-|j||\neq 1$ and  $|i|\neq |j|$. 
\end{remark}

As mentioned earlier, the matrix coefficients of  Fiedler-like pencils are defined as products  of elementary matrices. In order to describe these products and their properties it is convenient to use index tuples, which are introduced in the following definition.
\begin{definition}{\rm \cite[Definition 3.1]{GFPR}}
We call an \emph{index tuple}  a finite ordered sequence of integer numbers.
Each of these integers is called an \emph{index} of the tuple.
\end{definition}

If $\mathbf{t}=(t_1, \ldots, t_r)$ is an index tuple, we denote  $-\mathbf{t}:=(-t_1,\hdots,-t_r)$, and, when $a$ is an integer, we denote $a+\mathbf{t}:=(a+t_1, a+t_2, \ldots, a+t_r)$. 
We call the \emph{reversal index tuple} of $\mathbf{t}$  the index tuple $\rev(\mathbf{t}):=(t_r, \hdots, t_2,  t_1).$

When an index tuple $\mathbf{t}$ is of the form $(a:b)$ for some integers $a$ and $b$, we call $\mathbf{t}$ a \emph{string}.

Additionally, given index tuples $\mathbf{t}_1,\hdots,\mathbf{t}_s$,  we denote by $(\mathbf{t}_1,\hdots,\mathbf{t}_s)$ the index tuple obtained by concatenating the indices in the index tuples $\mathbf{t}_1,\hdots,\mathbf{t}_s$ in the indicated order.

Given an index tuple $\mathbf{t}=(t_1, t_2, \ldots, t_r)$, we call an $n\times n$ \emph{matrix assignment for $\mathbf{t}$}  any ordered collection  $\mathcal{X}=(X_1, X_2,\ldots, X_r)$ of arbitrary $n\times n$ matrices \cite[Definition 4.1]{GFPR}.  
If the size $n$ of the matrices is not important in the context, we simply say that $\mathcal{X}$ is a matrix assignment for $\mathbf{t}$.
In addition, if the indices of $\mathbf{t}$ are in $\{-k:k-1\}$, we  define
\begin{align}\label{prod-elem}
M_{\mathbf{t}}(\mathcal{X}) & := M_{t_1}(X_1) M_{t_2}(X_2) \cdots M_{t_r}(X_r).
\end{align}

Given a matrix polynomial $P(\lambda)\in\mathbb{F}[\lambda]^{n\times n}$, we say that $\mathcal{X}=(X_1, \ldots X_r)$ is the trivial matrix assignment for $\mathbf{t}=(t_1, \ldots, t_r)$ associated with $P(\lambda)$ if 
$M_{t_j}(X_j) = M_{t_j}^P$ for $j=1:r$. Moreover, we define 
\begin{equation}
M^P_{\mathbf{t}} := M_{t_1}^P M_{t_2}^P \cdots M_{t_r}^P.
\end{equation}
When $P(\lambda)$ is clear from the context, we just write $M_{\mathbf{t}}$ instead of $M^P_{\mathbf{t}}$.
If $\mathbf{t}$ is the empty tuple, then $M_{\mathbf{t}}^P$ and $M_{\mathbf{t}}(\mathcal{X})$  are defined to be the identity matrix $I_{kn}$. 

Given a sequence $ \mathcal{X}= (X_1, \ldots, X_r)$ of  matrices,  we  denote $\rev(\mathcal{X}):=(X_r,X_{r-1},\hdots,X_0)$.

If a  tuple $\mathbf{t}$ is expressed as a concatenation of tuples $\mathbf{t}_1,\hdots,\mathbf{t}_s$ (i.e., $\mathbf{t}=(\mathbf{t}_1,\hdots,\mathbf{t}_s)$), and a matrix assignment $\mathcal{X}$ for $\mathbf{t}$ as a concatenation of matrix assignments $\mathcal{X}_1,\hdots,\mathcal{X}_s$ (i.e., $\mathcal{X}=(\mathcal{X}_1,\hdots,\mathcal{X}_s)$), where the number of indices in $\mathbf{t}_i$ is equal to the number of matrices in $\mathcal{X}_i$, for $i=1:s$, we say that \emph{$\mathcal{X}_i$ is the matrix assignment for $\mathbf{t}_i$ induced by $\mathcal{X}$}.

The following simple lemma gives the block-structure of $M_{\mathbf{t}}(\mathcal{X})$ when $\mathbf{t}$ is a string consisting of nonnegative indices.
We omit its proof since the result can be obtained by a direct computation.
\begin{lemma}\label{lem: string product}
Let $\mathbf{t}=(a:b)$ be a string with indices from $\{0:k-1\}$ and let $\mathcal{X} = (X_a,\dots,X_b)$ be an $n\times n$ matrix assignment for $\mathbf{t}$. Then,  the $kn \times kn$ matrix $M_{\mathbf{t}}(\mathcal{X})$  is a block-diagonal matrix  of the form
\begin{equation}\label{eq:string product block1}
M_{\mathbf t}(\mathcal{X}) \coloneqq
I_{n(k-b-1)} \oplus \begin{bmatrix}
 X_b & I_n & \phantom{X} & \phantom{X} & \phantom{X} \\
 X_{b-1} & \phantom{X} & I_n & \phantom{X} & \phantom{X} \\
 \vdots & \phantom{X} & \phantom{X} & \ddots & \phantom{X} \\
 X_a & \phantom{X} & \phantom{X} & \phantom{X} & I_n \\
 I_n & \phantom{X} & \phantom{X} & \phantom{X} & \phantom{X}
\end{bmatrix} \oplus I_{n(a-1)}
\end{equation}
if $a\neq 0$; and of the form 
\begin{equation}\label{eq:string product block2}
M_{\mathbf t}(\mathcal{X}) \coloneqq
I_{n(k-b-1)} \oplus \begin{bmatrix}
 X_b & I_n & \phantom{X} & \phantom{X} & \phantom{X} \\
 X_{b-1} & \phantom{X} & I_n & \phantom{X} & \phantom{X} \\
 \vdots & \phantom{X} & \phantom{X} & \ddots & \phantom{X} \\
 X_1 & \phantom{X} & \phantom{X} & \phantom{X} & I_n \\
 X_0  & \phantom{X} & \phantom{X} & \phantom{X} & \phantom{X}
\end{bmatrix}
\end{equation}
if $a=0$. 
\end{lemma}

Notice that the matrix $M_{\mathbf t}(\mathcal{X})$ in \eqref{eq:string product block1}--\eqref{eq:string product block2}  is \emph{operation-free}, that is, it is a block-matrix  whose blocks are of the form $0_n$, $I_n$, and the $n\times n$ matrices from the matrix assignment $\mathcal{X}$.
 Note also that the position where each of these blocks lies only depends on $\mathbf{t}$ \cite[Definition 4.5]{GFPR}.
To guarantee this operation-free property for more general index  tuples $\mathbf{t}$, we need the following definition.
\begin{definition}
\label{defSIP}{\rm \cite[Definition 7]{ant-vol11}} Let ${\mathbf{t}}=(i_{1},i_{2},\hdots,i_{r})$
be an index tuple of either nonnegative integers or  negative integers.
 Then, ${\mathbf{t}}$ is said to satisfy the \emph{Successor
Infix Property (SIP)} if for every pair of indices $i_{a},i_{b}\in{\mathbf{t}%
}$, with $1\leq a<b\leq r$, satisfying $i_{a}=i_{b}$, there exists at least
one index $i_{c}=i_{a}+1$ with $a<c<b$.
\end{definition}

\begin{remark}\label{remark:on_SIP}
We note  that any subtuple of consecutive indices of a tuple satisfying the SIP also satisfies the SIP. 
Moreover, the reversal of any tuple satisfying the SIP also satisfies the SIP.
\end{remark}

\begin{theorem}{\rm \cite[Theorem 4.6]{GFPR}}
Let $\mathbf{t}$ be a tuple with indices from either $\{0:k-1\}$ or $\{-k:-1\}$.
 The tuple $\mathbf{t}$ is operation free in the sense of {\rm \cite[Definition 4.5]{GFPR}} if and only if $\mathbf{t}$ satisfies the SIP.
\end{theorem}

SIP  also guarantees  that the block-transpose of
a product of elementary matrices behaves as the
regular transpose.
\begin{lemma}\label{revB}{\rm \cite[Lemma 4.8]{GFPR}}
Let $k$ be  a positive integer, let $\mathbf{t}$ be an index tuple satisfying the SIP with indices from $\{0:k-1\}$ and let $\mathcal{X}$ be a matrix assignment for  $\mathbf{t}$. 
Then, 
\[
M_{\mathbf{t}}(\mathcal{X})^{\mathcal{B}}= M_{\rev(\mathbf{t})}\left(\rev(\mathcal{X})\right).
\]
\end{lemma}




Due to the commutativity relations of elementary matrices explained in  Remark \ref{commutativity}, different index tuples $\mathbf{t}_1$ and $\mathbf{t}_2$ may give rise to the same product of elementary matrices, that is, given a matrix assignment $\mathcal{X}$ for both $\mathbf{t}_1$ and $\mathbf{t}_2$, we may have $M_{\mathbf{t}_1}(\mathcal{X})=M_{\mathbf{t}_2}(\mathcal{X})$.
Next, we introduce a canonical form for index tuples such that two index tuples with the same canonical form yield the same product of elementary matrices. 
But, first, notice that given an index tuple $\mathbf{t}=(t_1, t_2, \ldots, t_r)$ and a matrix assignment $\mathcal{X}=(X_1,\hdots,X_r)$ for $\mathbf{t}$, the elementary matrices $M_{t_i}(X_i)$ and $M_{t_{i+1}}(X_{i+1})$ commute whenever $||t_i|- |t_{i+1}|| \neq 1$.
For this reason, here and thereafter, we say that in this situation the \emph{indices $t_i$ and $t_{i+1}$ in $\mathbf{t}$ commute}.
\begin{definition}{\rm \cite[Definition 3.4]{GFPR}}\label{def:equivalence_tuples}
Given two index tuples $\mathbf{t}$ and $\mathbf{t}'$ of nonnegative indices, we say that $\mathbf{t}$ is equivalent to $\mathbf{t}'$ (and write $\mathbf{t} \sim \mathbf{t'}$), if $\mathbf{t}= \mathbf{t'}$ or $\mathbf{t}'$ can be obtained from $\mathbf{t}$ by interchanging a finite number of times two distinct commuting indices in adjacent positions, that is, indices $t_i$ and $t_{i+1}$ such that  $| t_i - t_{i+1}| \neq 1$ and $t_i \neq t_{i+1}$.
\end{definition}
\begin{remark}
The product of elementary matrices is invariant under the equivalence introduced in Definition \ref{def:equivalence_tuples}, i.e.,  given an index tuple $\mathbf{t}$ and a matrix assignment $\mathcal{X}$ for $\mathbf{t}$, if $\mathbf{t} \sim \mathbf{t'}$, then $M_{\mathbf{t}}(\mathcal{X})= M_{\mathbf{t'}}(\mathcal{X})$.
\end{remark}

It turns out that if $\mathbf{t}$ satisfies the SIP property, the index tuple $\mathbf{t}$ is equivalent to  a tuple with a special structure that we introduce in the following definition.  
 \begin{definition}
\label{FSC}{\rm \cite[Theorem 1]{ant-vol11}} Let ${\mathbf{t}}$ be an index tuple
with indices from $\{0:h\}$, $h \geq 0$. 
Then ${\mathbf{t}}$ is said to be in \emph{column standard form}  if
\[
{\mathbf{t}}=\left(  a_{s}:b_s,a_{s-1}:b_{s-1},\hdots,a_{2}:b_{2}%
,a_{1}:b_{1}\right),
\]
with $h\geq b_{s}>b_{s-1}>\cdots>b_{2}>b_{1}\geq0$ and $0\leq a_{j}\leq b_{j}%
$, for all $j=1:s$. We call  each subtuple of consecutive indices
$(a_{i}:b_{i})$ a \emph{string} of ${\mathbf{t}}$.
\end{definition}

In the next lemma, we show the connection between tuples satisfying the SIP and tuples in column standard form.
\begin{lemma}\label{SIP-csf}{\rm \cite[Theorem 2]{ant-vol11}}
Let $\mathbf{t}$ be an index tuple. 
\begin{itemize}
\item[\rm (i)]  If the indices of $\mathbf{t}$ are all  nonnegative integers, then ${\mathbf{t}}$ satisfies the SIP if and only if $\,{\mathbf{t}}$ is equivalent to a tuple in column standard form.
\item[\rm (ii)] If the indices of $\mathbf{t}$ are all negative integers and $a$ is the minimum index in $\mathbf{t}$, then $\mathbf{t}$ satisfies the SIP if and only if $-a+\mathbf{t}$ is equivalent to a tuple in column standard form.
\end{itemize}
\end{lemma}

Lemma \ref{SIP-csf}, together with \cite[Proposition 2.12]{FPR1}, allows us to introduce a canonical form for tuples satisfying the SIP under the equivalence relation introduced in Definition \ref{def:equivalence_tuples}.
\begin{definition}{\rm \cite[Definition 3.9]{GFPR}}
The unique index tuple in column standard form equivalent to an index tuple $\mathbf{t}$ of nonnegative integers satisfying the SIP is called the  \emph{column standard form of $\mathbf{t}$}  and is denoted by $\mathrm{csf}(\mathbf{t})$.
\end{definition}

In the next definition, we introduce some special indices of  tuples of nonnegative integers satisfying the SIP that will play a key role in the rest of the paper.
\begin{definition}\label{def;heads}
For an arbitrary index tuple $\mathbf t$ with $\csf(\mathbf t) = \left(  a_{s}:b_s,\hdots
,a_{1}:b_{1}\right)$, we define $\heads(\mathbf {t}) := \{b_i ~|~ 1 \leq i \leq s\}$. 
Furthermore, we denote by  $\mathfrak{h}(\mathbf{t})$ the cardinality of $\heads(\mathbf{t})$. 
\end{definition}

Given an index tuple $\mathbf{t}$, note that $\mathfrak{h}(\mathbf{t})$ gives not only the cardinality of the set $\heads(\mathbf {t})$, but also the number of strings in $\mathrm{csf}(\mathbf{t})$. 

\begin{example}
The tuples $\mathbf{t}_1=(3:4, 0:2)$ and $\mathbf{t}_2=(1:4, 0:3, 0:2, 1, 0)$ are tuples in column standard form with indices from $\{0:4\}$.
Note that $\heads(\mathbf{t}_1) = \{4, 2\}$ and $\mathfrak{h}(\mathbf{t}_1)=2$.
Similarly, $\heads(\mathbf{t}_2) =\{4, 3, 2, 1, 0\}$ and $\mathfrak{h}(\mathbf{t}_2)= 5$. 
\end{example}

We introduce in Definition \ref{tuple-type} an important distinction between some type of indices relative to a given index tuple.
\begin{definition}\label{tuple-type}
Given an index tuple $\mathbf t$ and an index $x$ such that $(\mathbf{t}, x)$ satisfies the SIP,  we say that $x$ is of \emph{Type~I} relative to $\mathbf t$ if   $\mathfrak{h}(\mathbf t, x)= \mathfrak{h}(\mathbf t)$, and of \emph{Type~II} otherwise. 
That is, $x$ is of Type I relative to $\mathbf{t}$ if $\csf(\mathbf{t}, x)$ has the same number of heads (and, therefore of strings) as $\csf(\mathbf{t})$, and of \emph{Type~II} relative to $\mathbf{t}$ otherwise.
\end{definition}

\begin{example}
Let $\mathbf{t} =(3:5, 2, 0:1)$. Then $3$ is an index of Type I relative to $\mathbf{t}$ since 
$$\csf(\mathbf{t}, 3)= (3:5, 2:3, 0:1)$$
has the same number of strings as $\mathbf{t}$ (recall that $3$ commutes with 0 and 1). 
However, $4$ is an index of Type II relative to $\mathbf{t}$ since
$$\csf(\mathbf{t}, 4) = ( 3:5, 4, 2, 0:1)$$
has more strings than $\mathbf{t}$. 
\end{example}

The following technical lemmas and results are used in future sections.
Here and thereafter we use the following notation.
If $\mathbf{t}$ is an index tuple and $a$ is an index of $\mathbf{t}$, we write $a\in \mathbf{t}$. 
Also, we denote by $e_i$ the $i$th column of the identity matrix of arbitrary size.  

In the first results  we determine the set of heads of the concatenation $(\mathbf{t}_1, \mathbf{t}_2)$ of two index tuples that satisfies the SIP when the set of heads of  either $\mathbf{t}_1$ or $\mathbf{t}_2$ is known. 
\begin{lemma}\label{txSIP}
Let $k$ be a positive integer and let $\mathbf{t}$ be an index tuple with indices from $\{0:k-1\}$.
Let $(a:b)$ be a string with indices from $\{0:k-2\}$.
Then, the tuple $(\mathbf{t},a:b)$ satisfies the SIP property if and only if $\mathbf{t}$ satisfies the SIP and $c\notin\mathrm{heads}(\mathbf{t})$, for all $c\in (a:b)$. 
\end{lemma}
\begin{proof}
 $\Longrightarrow$ Since $(\mathbf{t}, a:b)$ satisfies the SIP, so does $\mathbf{t}$. 
Thus, by Lemma \ref{SIP-csf}, we may assume, without loss of generality, that $\mathbf{t}$ is in column standard form, that is, $\mathbf{t} = (a_s: b_s, \ldots, a_1: b_1)$.
The proof proceeds by contradiction. 
Assume that  $c\in (a:b)$ and $c \in \heads(\mathbf{t})$. 
Then, there exists $1 \leq i \leq s$ such that $c= b_i$. 
Since $\mathbf{t}$ is a tuple in column standard form, the indices in the subtuple $(a_{i-1}: b_{i-1}, \ldots, a_1:b_1)$ of $\mathbf{t}$ are in $\{0: b_i-1\}$. 
Therefore, there is no index between $b_i\in\mathbf{t}$ and $b_i=c\in (a:b)$ equal to $c+1$ which implies that $(\mathbf{t}, a:c)$ does not satisfy the SIP, a contradiction.  

\medskip

\noindent  $\Longleftarrow$ To prove the implication, it is enough to show that,  between $c\in(a:b)$ and the rightmost index equal to $c$ in $\mathbf{t}$, if it exists,  there is an index equal to $c+1$. Let $c$ be such an index.
Since $\mathbf{t}$ satisfies the SIP, without loss of generality, we may assume that $\mathbf{t}=(a_s:b_s,\hdots,a_1:b_1)$, for some nonnegative integers $a_i$, $b_i$ and $i=1:s$, that is, $\mathbf{t}$ is in column standard form.
Let $(a_j:b_j)$ be the rightmost string of $\mathbf{t}$ such that $c\in(a_j:b_j)$. 
Since $c\neq b_i$, for all $i$, the tuple $(a_j:b_j)$ is of the form
\[
(c,c+1,\hdots,b_j) \quad \mbox{or} \quad  (a_j,\hdots, c,c+1,\hdots,b_j) \quad \mbox{or} \quad (a_j,\hdots, c,c+1),
\]
which shows the desired result.
\end{proof}

\medskip

\begin{proposition}\label{prop: Type I}
Let $\mathbf t =(a_s:b_s, \ldots, a_2:b_2, a_1:b_1)$ be a nonempty index tuple in column standard form with indices from $\{0:k-1\}$, for $k \geq 1$. Let $x$ be an index in $\{0:b_s-1\}$ such that $(\mathbf{t}, x)$ satisfies the SIP. Then $x$ is of Type~I relative to $\mathbf t$ if and only if  $x-1 \in \heads(\mathbf{t})$.
In particular,  $x=0$ is always an index of Type II relative to $\mathbf{t}$, for every nonempty index tuple $\mathbf{t}$ in column standard form with nonnegative indices.
\end{proposition}

\begin{proof}
Since $(\mathbf t, x)$ satisfies the SIP, we have $x \notin \heads(\mathbf{t})$ by Lemma \ref{txSIP}. Therefore, since $x \in \{0:b_s-1\}$, there is either some $1 \leq  \ell < s$ with $b_{\ell+1}>  x > b_{\ell}$  or $b_1 > x$.

Assume that  $x -1 \in \heads({\mathbf{t}} )$. Then, $x = b_j+1$ for some  $1 \leq j \leq s$. Thus, we have
\begin{align*}
(\mathbf t, x)& = (a_s:b_s,\dots, a_j:b_j,\dots, a_1:b_1, b_j+ 1) \\
& \sim (a_s:b_s,\dots,a_j:b_j+1,\dots, a_1:b_1) =\csf(\mathbf{t}, x),
\end{align*}
where the  equivalence holds because $x  > b_j >  b_{j-1}$, and hence $x$ commutes with every index of $\mathbf t$ ocurring to the right of $b_j$. Moreover, the second equality follows from the fact that $b_{j+1} \neq b_j+1$, otherwise, the SIP does not hold.  Since $\csf(\mathbf t, x)$ has the same number of strings (heads) as $\csf(\mathbf t)$, it follows that $x$ is of Type~I relative to $\mathbf t$.

Now suppose that  $x-1 \notin \heads(\mathbf{t})$. If  $b_1 > x $, then
 \begin{align*}
\csf(\mathbf t, x) = (\mathbf{t}, x),
\end{align*}
by definition of column standard form, which implies that $\csf(\mathbf{t}, x)$ has one more string than $\mathbf{t}$. Hence $x$ is not of Type~I relative to $\mathbf t$. \\
\indent If $b_{\ell+1}> x > b_{\ell}$ for some $\ell>0$ and $x-1 \notin \heads(\mathbf{t})$,  then $x > b_\ell+1$.  Therefore, $x$ commutes with  $(a_\ell: b_\ell, \ldots, a_1:b_1)$ since all the indices of this tuple  are smaller than or equal to $b_\ell$. 
Therefore, 
\begin{align*}
(\mathbf t, x) &= (a_s:b_s,\dots,a_\ell:b_\ell,\dots,a_1:b_1, x) \\
&\sim (a_s:b_s, \dots,a_{\ell+1}:b_{\ell+1}, x , a_\ell:b_\ell,\dots, a_1:b_1) = \csf(\mathbf t, x),
\end{align*}
Since the number of strings of $\csf(\mathbf{t}, x)$ is larger than the number of strings of $\mathbf{t}$, the index $x$ is not of Type~I relative to $\mathbf t$ and the result follows.
\end{proof}
\begin{remark}\label{rem: Type I and II heads}
Using the notation of Proposition \ref{prop: Type I}, note that this proposition implies that $\heads(\mathbf t, x) = (\heads(\mathbf t) \setminus \{b_j\}) \cup \{x\}$ when $x$ is an index of Type I relative to $\mathbf t$ and $x=b_j+1$.  Furthermore, the proof of Proposition \ref{prop: Type I} shows that if $x$ is of Type II relative to $\mathbf t$, then $\heads(\mathbf t,x) = \heads(\mathbf t) \cup \{x\}$. 
\end{remark}

\begin{lemma}\label{xtSIP}
Let $k$ be a positive integer and let $\mathbf t$ be an index tuple  containing each index in $\{0: k-1\}$ at least once. 
Let $(a:b)$ be a string with indices from $\{0:k-2\}$ such that $(a:b, \mathbf t)$ satisfies the SIP. 
Then   $\heads(a:b, \mathbf t) = \heads(\mathbf t)$.
\end{lemma}
\begin{proof}
Since $(a:b, \mathbf{t})$ satisfies the SIP, so does $\mathbf{t}$. 
Therefore, the tuple $\mathbf{t}$ is equivalent to a unique tuple $\mathrm{csf}(\mathbf{t}) = (a_s: b_s, \ldots, a_1: b_1)$ in column standard form by Lemma \ref{SIP-csf}. 
Because $(a:b, \mathbf t)$, and hence $(a:b, \csf(\mathbf t))$, satisfies the SIP, we know that the leftmost appearance of $b+1$ in $\csf(\mathbf t)$ is farther left than any appearance of $b$. 
We know that such appearances actually occur because each index in $ \{0:k-1\}$ appears in $\mathbf t$ at least once. 
It follows that the leftmost appearance of $b+1$ in $\csf(\mathbf t)$ must be of the form $a_j$ for some $1\leq j \leq s$, since otherwise $b$ would occur to the left of this appearance of  $b+1$. 
Moreover, since $(a:b, \mathbf{t})$ satisfies the SIP, all indices in $(a_s:b_s, \ldots, a_{j+1}: b_{j+1})$  are  larger than  $b+1$. 
Therefore, the indices $a,a+1,\hdots,b$ all commute with the indices in $(a_s:b_s,\hdots,a_{j+1}:b_{j+1})$ and we have
\begin{align*}
(a:b,\csf(\mathbf t)) &= (a:b, a_s:b_s,\dots,a_j:b_j,\dots) \\
&\sim (a_s:b_s,\dots,a:b, a_j:b_j,\dots) \\
&= (a_s:b_s,\dots,a:b_j,\dots),
\end{align*}
which implies $\heads(a:b, \mathbf t) = \heads(\mathbf t)$.
\end{proof}

We end this section with a result  that provides some structural properties concerning the block-columns and block-rows of the product of elementary matrices $M_{\mathbf{t}}(\mathcal{X})$ defined in \eqref{prod-elem}.
\begin{lemma}\label{lem: column locations}
Let  $k$ be a positive integer and let $\mathbf t $ be an index tuple  in column standard form with indices from $\{0:k-1\}$.
Let $\mathcal{X} $ be an $n\times n$ matrix assignment for $\mathbf{t}$,
and let us consider the $kn \times kn$ matrix $M_{\mathbf t}(\mathcal{X})$ as a $k\times k$ block-matrix with blocks of size $n\times n$.  
Then, all block-columns of $M_{\mathbf t}(\mathcal{X})$ 
are of the form $e_i \otimes I_n$ (for some $1 \leq i \leq k$) except for the $(k-h)$th block-column for each $h \in \heads(\mathbf t)$.
Moreover, all block-rows of the $kn \times kn$ matrix $M_{\mathbf t}(\mathcal{X})$ 
are of the form $e_i^T \otimes I_n$ (for some $1 \leq i \leq k$) except for the $(k-h)$th block-row for each $h \in \heads(\mathrm{rev}(\mathbf t))$.
\end{lemma}
\begin{proof}
The proof proceeds by induction on the number $s$ of strings of the tuple $\mathbf{t}$.
Assume, first, that $s=1$. 
In this case $\mathbf{t}=(a_1:b_1)$, for some $a_1\leq b_1$.
By Lemma \ref{lem: string product}, it is clear that all block-columns of $M_{t}(\mathcal{X})$ are of the form $e_i \otimes I_n$, except for the $(k-b_1)$th block-column. 
Therefore, the desired result for block-columns holds for $s = 1$. 

 Suppose now that the lemma holds for all $s$ less than some $s_0 > 1$ and assume that $\mathbf{t}$ has  $s_0$ strings. 
 Let $\mathbf t =(a_{s_0}: b_{s_0}, \ldots,  a_1 : b_1)$. 
For $1\leq j \leq s_0$, let $\mathbf{t}_j := (a_j: b_j)$  and let $\mathcal{X}_j $ be the matrix assignment for $\mathbf{t}_j$ induced by $\mathcal{X}$.
 Then 
$$M_{\mathbf t}(\mathcal{X}) = M_{\mathbf{t}_{s_0}}(\mathcal{X}_{s_0}) \cdots M_{\mathbf t_1}(\mathcal{X}_1).$$
By the inductive hypothesis, all block-columns of $M_{\mathbf t_{s_0}}(\mathcal{X}_{s_0}) \cdots M_{\mathbf t_{2}}(\mathcal{X}_{2})$, except for the $(k-b_j)$th block-column, for  $2 \leq j \leq s_0$, are of the form  $e_i \otimes I_n$.
Thus,
\[
M_{\mathbf t_{s_0}}(\mathcal{X}_{s_0}) \cdots M_{\mathbf t_{2}}(\mathcal{X}_{2}) = \begin{bmatrix} C & D \end{bmatrix},
\]
where $C$ is formed by the first $k-b_2$ block-columns of $M_{\mathbf t_{s_0}}(\mathcal{X}_{s_0}) \cdots M_{\mathbf t_{2}}(\mathcal{X}_{2})$, and all block-columns of $D$ are of the form $e_i \otimes I_n$. 
Recall that, since $\mathbf{t}$ is a tuple in column standard form, we have $b_{s_0} > b_{s_0-1}> \cdots > b_1$. Thus, from Lemma \ref{lem: string product}, multiplying 
$
M_{\mathbf t_{s_0}}(\mathcal{X}_{s_0}) \cdots M_{\mathbf t_{2}}(\mathcal{X}_{2})$ by $M_{\mathbf t_1}(\mathcal{X}_1)$ leaves unmodified the columns of $C$. Moreover, taking into account the block-structure of $M_{\mathbf{t}_1}(\mathcal{X}_1)$,  all the block-columns of $
M_{\mathbf t_{s_0}}(\mathcal{X}_{s_0}) \cdots M_{\mathbf t_{1}}(\mathcal{X}_{1})$ from the $(k-b_2+1)$th to the last, with the exception of the  $(k-b_1)$th column, are of the form $e_i \otimes I_n$ for some $i$, and the result  for the block-columns follows.


To finish the proof, note that the block-rows of $M_{\mathbf{t}}(\mathcal{X})$ are the block-columns of $M_{\mathbf{t}}(\mathcal{X})^{\mathcal{B}}$, which, by Lemma \ref{revB}, equals $M_{\mathrm{rev}(\mathbf{t})}(\mathrm{rev}(\mathcal{X}))$. 
Thus, the result for block-rows follows from the result for block-columns applied to $M_{\mathbf{t}}(\mathcal{X})^{\mathcal{B}}$.
\end{proof}


\subsection{Fiedler-like pencils families}

In this section we recall the families of Fiedler pencils, generalized Fiedler pencils, Fiedler pencils with repetition, and generalized Fiedler pencils with repetition (see also  \cite{AV04,GFPR,DTDM10,DTDM12,Fiedler03,ant-vol11}).
Notice that in this work these families are defined in terms of the grades of the matrix polynomials they are associated with, in contrast to \cite{AV04,GFPR,ant-vol11}, where they were originally defined in terms of the degrees.

We start by recalling the definition of Fiedler pencils.
Although they have been defined for both square and rectangular matrix polynomials \cite{DTDM12}, here we focus on the square case.
In this case, Fiedler pencils can be defined via elementary matrices.
\begin{definition}{\rm (Fiedler pencils)}
Let $P(\lambda)=\sum_{i=0}^k A_i\lambda^i\in\mathbb{F}[\lambda]^{n\times n}$, and let $\mathbf{q}$ be a permutation of $\{0:k-1\}$.
Then, the \emph{Fiedler pencil} associated with $P(\lambda)$ and $\mathbf{q}$ is
\[
F_{\mathbf{q}}(\lambda):=\lambda M_{-k}^P-M_{\mathbf{q}}^P.
\]
\end{definition} 
\begin{example}\label{ex:FP}
Let $P(\lambda)=\sum_{i=0}^6 A_i\lambda^i \in\mathbb{F}[\lambda]^{n\times n}$, and let $\mathbf{q}=(0,2,4,1,3,5)$.
Then, the pencil
\begin{equation} \label{eq:exampfiedler}
F_\mathbf{q} (\lambda) =
\begin{bmatrix}
\lambda P_6+P_5 & -I_n & 0 & 0 & 0 & 0 \\
P_4 & \lambda I_n & P_3 & -I_n & 0 & 0 \\
-I_n & 0 & \lambda I_n & 0 & 0 & 0 \\
0 & 0 & P_2 & \lambda I_n & P_1 & -I_n \\
0 & 0 & -I_n & 0 & \lambda I_n & 0 \\
0 & 0 & 0 & 0 & P_0 & \lambda I_n
\end{bmatrix}
\end{equation}
is the Fiedler pencil associated with $P(\lambda)$ and $\mathbf{q}$.
This pencil is the Fiedler pencil  considered in Section \ref{sec:informal}.
\end{example}

Next, we introduce the family of generalized Fiedler pencils and proper generalized Fiedler pencils. 
\begin{definition}\label{def:GFP}{\rm (GFP and proper GFP)}
Let $P(\lambda)=\sum_{i=0}^k A_i\lambda^i\in\mathbb{F}[\lambda]^{n\times n}$, let $\{C_0,C_1\}$ be a partition of $\{0:k\}$, where we allow $C_0$ or $C_1$ to be the empty set, and let $\mathbf{q}$ and $\mathbf{z}$ be permutations of $C_0$ and $-C_1$, respectively.
Then, the \emph{generalized Fiedler pencil (GFP)} associated with  $P(\lambda)$ and  $(\mathbf{q},\mathbf{z})$ is 
\[
K_{\mathbf{q},\mathbf{z}}(\lambda):=\lambda M_{\mathbf{z}}^P-M_{\mathbf{q}}^P.
\]
If $0\in C_0$ and $k\in C_1$, then the pencil $K_{\mathbf{q},\mathbf{z}}(\lambda)$ is said to be a \emph{proper generalized Fiedler pencil (proper GFP)} associated with $P(\lambda)$.
\end{definition}

We show in Example \ref{ex:GFP} one proper generalized Fiedler pencil  associated with a matrix polynomial of grade 6.
\begin{example}\label{ex:GFP}
Let $P(\lambda)=\sum_{i=0}^6 A_i\lambda^i \in\mathbb{F}[\lambda]^{n\times n}$, let $\mathbf{z}=(-1,-6,-5)$ and $\mathbf{q}=(3,4,2,0)$.
Then, the pencil
\[
K_{\mathbf{q},\mathbf{z}}(\lambda) = 
\begin{bmatrix}
-I_n & \lambda A_6 & 0 & 0 & 0 & 0 \\
\lambda I_n & \lambda A_5+A_4 & -I_n & 0 & 0 & 0 \\
0 & A_3 & \lambda I_n & A_2 & -I_n & 0 \\
0 & -I_n & 0 & \lambda I_n & 0 & 0 \\
0 & 0 & 0 & -I_n & 0 & \lambda I_n \\
0 & 0 & 0 & 0 & \lambda I_n & \lambda A_1+A_0
\end{bmatrix}
\]
is the proper generalized Fiedler pencil associated with $P(\lambda)$ and $(\mathbf{q},\mathbf{z})$.
This pencil is the proper generalized Fiedler pencil considered in Section \ref{sec:informal}.
\end{example}

\begin{remark}\label{rem:properGFP}
Notice that, for any given  proper GFP $K_{\mathbf{q},\mathbf{z}}(\lambda)$, there is always an index $h\in\{0:k-1\}$ and a tuple $\mathbf{m}$ such that 
\[
K_{\mathbf{q},\mathbf{z}}(\lambda)=\lambda M_{(\mathbf{m},-k:-h-1)}^P-M_{\mathbf{q}}^P,
\]
where $\widehat{\mathbf{q}}:=(-\rev(\mathbf{m}),\mathbf{q})$ is a permutation of $\{0:h\}$.
We call $(\widehat{\mathbf{q}},h)$ the simple pair associated with $K_{\mathbf{q},\mathbf{z}}(\lambda)$.
\end{remark}

We finish the section by recalling the definition of Fiedler pencils with repetition and generalized Fiedler pencils with repetition.
We note that, although we have assigned a special notation to denote Fiedler and generalized Fiedler pencils, we will not do so with FPR and GFPR since their construction involve too many parameters.
\begin{definition}{\rm (FPR and GPFR)}
\label{defGFPR}
Let $P(\lambda)=\sum_{i=0}^k A_i\lambda^i\in\mathbb{F}[\lambda]^{n\times n}$.
Let $h\in\left\{  0:k-1\right\}$, and let ${\mathbf{q}}$ and $\mathbf{z}$ be permutations of $\left\{  0:h\right\}
$ and $\left\{  -k:-h-1\right\}$, respectively. 
Let ${\boldsymbol\ell}_{q}$ and
${\mathbf{r}}_{q}$ be tuples with indices from $\left\{  0:h-1\right\}  $ such
that $({\boldsymbol\ell}_{q},{\mathbf{q,r}}_{q})$ satisfies the SIP. 
Let
${\boldsymbol\ell}_{z}$ and ${\mathbf{r}}_{z}$ be tuples with indices from
$\left\{  -k:-h-2\right\}  $ such that $({\boldsymbol\ell}_{z},{\mathbf{z,r}}_{z})$ satisfies the SIP. 
Let $\mathcal{X},$ $\mathcal{Y},$ $\mathcal{Z}$ and $\mathcal{W}$ be $n\times n$ matrix assignments for ${\boldsymbol\ell}
_{q},$ $\mathbf{r}_{q},$ ${\boldsymbol\ell}_{z}$ and $\mathbf{r}_{z},$ respectively.
Then, the pencil
\begin{equation}\label{fprgen}
L_P(\lambda)=M_{{\boldsymbol\ell}_{q},{\boldsymbol\ell}_{z}}(\mathcal{X},\mathcal{ Z})(\lambda M^P_{\mathbf{z}}
-M^P_{\mathbf{q}})M_{\mathbf{r}_{z},\mathbf{r}_{q}}(\mathcal{W}, \mathcal{Y})
\end{equation}
is called a \emph{generalized Fiedler pencil with repetition} associated with  $P(\lambda).$
When $\mathcal{X}$, $\mathcal{Z},$ $\mathcal{W}$ and $\mathcal{Y}$ are the trivial matrix  assignments for ${\boldsymbol\ell}_q$, ${\boldsymbol\ell}_z$, $\mathbf{r}_z$ and $\mathbf{r}_q$, respectively, associated with $P(\lambda)$, then $L_P(\lambda)$ is called a \emph{Fiedler pencil with repetition}.
\end{definition}
\begin{remark}
If ${\boldsymbol\ell}_q$, ${\boldsymbol\ell}_z$, $\mathbf{r}_z$ and $\mathbf{r}_q$ are all  empty tuples, then $L_P(\lambda)=K_{\mathbf{q},\mathbf{z}}(\lambda)$, that is, $L_P(\lambda)$ is a special type of GFP. 
In particular,  if $\mathbf{z}=(-k)$, then $L_P(\lambda)=F_{\mathbf{q}}(\lambda)$, that is, $L_P(\lambda)$ is a Fiedler pencil. 
Thus, the family of GFPR contains the Fiedler pencils, the FPR, and a subfamily of the GFP.
\end{remark}

We show in Example \ref{ex:FPR} one Fiedler pencil with repetition associated with a matrix polynomial of grade 6.
\begin{example}\label{ex:FPR}
Let $P(\lambda)=\sum_{i=0}^6 A_i\lambda^i \in\mathbb{F}[\lambda]^{n\times n}$, let $\mathbf{z}=(-6:-1)$,  $\mathbf{q}=(0)$ and $\mathbf{r}_z=(-6:-2,-6:-3,-6:-4,-6:-5,-6)$.
Notice that the index tuple $(\mathbf{z},\mathbf{r}_z)$ satisfies the SIP.
Then, the following pencil
\[
(\lambda M^P_{\mathbf{z}}-M^P_{\mathbf{q}})M_{\mathbf{r}_z}^P =
\begin{bmatrix}
0 & 0 & 0 & 0 & -A_6 & \lambda A_6 \\
0 & 0 & 0 & -A_6 & \lambda A_6-A_5 & \lambda A_5 \\
0 & 0 & -A_6 & \lambda A_6-A_5 & \lambda A_5-A_4 & \lambda A_4\\
0 & -A_6 & \lambda A_6-A_5 & \lambda A_5 -A_4 & \lambda A_4-A_3 & \lambda A_3 \\
-A_6 & \lambda A_6-A_5 & \lambda A_5-A_4 & \lambda A_4-A_3 & \lambda A_3-A_2 & \lambda A_2 \\
\lambda A_6 & \lambda A_5 & \lambda A_4 & \lambda A_3 & \lambda A_2 & \lambda A_1+A_0
\end{bmatrix}
\]
is a Fiedler pencil with repetition associated with $P(\lambda)$. 
This pencil is the Fiedler pencil with repetition considered in Section \ref{sec:informal}.
Moreover, it is also one of the pencils in the standard basis of the vector space $\mathbb{DL}(P)$.
\end{example}

We end this section presenting three structural lemmas for GFPR.
But first, taking into consideration the commutativity relations \eqref{commutativity}, notice that the GFPR associated with a matrix polynomial $P(\lambda)$ given by \eqref{fprgen} can be rewritten as follows
\[L(\lambda) = \lambda [M_{\boldsymbol\ell_q, \mathbf{r}_q }(\mathcal{\mathcal{X}}, \mathcal{\mathcal{Y}})(M_{\boldsymbol\ell_z}(\mathcal{\mathcal{Z}}) M^P_{\mathbf{z}}M_{ \mathbf{r}_z}(\mathcal{W}))] - [M_{\boldsymbol\ell_q}(\mathcal{X})M^P_{\mathbf{q}} M_{\mathbf{r}_q}(\mathcal{Y}) M_{\boldsymbol\ell_z, \mathbf{r}_z}(\mathcal{Z},\mathcal{W})].
\]

In Lemma \ref{lem:block_structure}, we give some information about the block-structure of the factors that appear in the matrix coefficients of the pencil above.
This lemma follows immediately just taking into account the range of indices in the tuples $\boldsymbol\ell_q$, $\boldsymbol\ell_z$, $\mathbf{r}_q$, $\mathbf{r}_z$, $\mathbf{z}$ and $\mathbf{q}$, so its proof is omitted.
\begin{lemma}\label{lem:block_structure}
Let 
$L(\lambda)=M_{{\boldsymbol\ell}_{q},{\boldsymbol\ell}_{z}}(\mathcal{X},\mathcal{ Z})(\lambda M^P_{\mathbf{z}}
-M^P_{\mathbf{q}})M_{\mathbf{r}_{z},\mathbf{r}_{q}}(\mathcal{W}, \mathcal{Y})$
be a GFPR associated with a matrix polynomial $P(\lambda)=\sum_{i=0}^k A_i\lambda^i\in\mathbb{F}[\lambda]^{n\times n}$.
Then, the following holds:
\begin{align*}
& M_{\boldsymbol\ell_q, \mathbf{r}_q }(\mathcal{X}, \mathcal{Y})= I_{n(k-h)} \oplus C_{22}, \quad \mbox{for some } C_{22} \in \mathbb{F}^{nh \times nh},\\
&M_{\boldsymbol\ell_z}(\mathcal{Z}) M_{\mathbf{z}}^PM_{ \mathbf{r}_z}(\mathcal{W}) = C_{11} \oplus I_{nh}, \quad \mbox{for some } C_{11} \in \mathbb{F}^{n(k-h) \times n(h-h)},\\
& M_{\boldsymbol\ell_q}(\mathcal{X}) M_{\mathbf{q}}^PM_{ \mathbf{r}_q}(\mathcal{Y}) =I_{n(k-h-1)} \oplus   D_{22} , \quad \mbox{for some } D_{22} \in \mathbb{F}^{n(h+1) \times n(h+1)}, \quad \mbox{and},\\
& M_{\boldsymbol\ell_z, \mathbf{r}_z}(\mathcal{Z}, \mathcal{W})= D_{11} \oplus I_{n(h+1)} , \quad \mbox{for some }D_{11} \in \mathbb{F}^{n(k-h-1) \times n(k-h-1)}.
\end{align*}
\end{lemma}

In the proof of Theorem \ref{thm:main_GFPR} (one of the main results in this paper), we will also use Lemma \ref{lem:structure}, where we present an important structural result for the GFPR, which follows from Lemma \ref{lem:block_structure}.
\begin{lemma}\label{lem:structure}{\rm \cite[Theorem 5.3]{GFPR}}
Let $P(\lambda)$ be an $n\times n$ matrix polynomial of grade $k$ and let
$$L_P(\lambda):=\lambda L_1-L_0=M_{\boldsymbol\ell_q,\boldsymbol\ell_z}(\mathcal{X},\mathcal{Z})(\lambda M_{\mathbf{z}}^P-M_{\mathbf{q}}^P) M_{\mathbf{r}_z,\mathbf{r}_q}(\mathcal{W},\mathcal{Y})$$
be a GFPR associated with $P(\lambda)$. 
Then:
\begin{enumerate}
\item[\rm(a)] $L_{1}=\diag(C_{11},C_{22})$, where $C_{11}$ is a $(k-h)\times(k-h)$ block-matrix which contains the blocks in the matrix assignments for the tuples $\mathbf{z}$,
$\boldsymbol\ell_{z}$, and $\mathbf{r}_{z}$, and $C_{22}$ is an $h\times h$ block-matrix which contains the blocks in the matrix assignments for the tuples $\boldsymbol\ell_{q}$ and
$\mathbf{r}_{q}$.
\item[\rm(b)] $L_{0}=\diag(D_{11},D_{22})$, where $D_{11}$ is a $(k-h-1)\times (k-h-1)$ block-matrix which contains the blocks in the matrix assignments for the tuples $\boldsymbol\ell_{z}$ and $\mathbf{r}_{z}$, and $D_{22}$ is a $(h+1)\times(h+1)$
block-matrix which contains the blocks in the matrix assignments for the tuples $\mathbf{q}$,
$\boldsymbol\ell_{q}$, and $\mathbf{r}_{q}$.
\end{enumerate}
\end{lemma}

As an immediate consequence of Lemma  \ref{lem:structure}, we obtain Lemma \ref{lem:splitting}, which will be one of our main tools in future sections.
\begin{lemma}\label{lem:splitting}
Let $P(\lambda)=\sum_{i=0}^k A_i\lambda^i\in\mathbb{F}[\lambda]^{n\times n}$ be a matrix polynomial of  grade $k$.
 Let $h\in \{0:k-1\}$ and let $\mathbf{q}$ be a permutation of $\{0:h\}$. 
 Let $\mathbf{z}$ be a permutation of $\{-k: -h-1\}$ and let $L_P(\lambda)$ be a GFPR as in \eqref{fprgen}.
Then, the pencil $L_P(\lambda)$ can be partitioned as
\[
L_P(\lambda) = \left[ \begin{array}{c|c|c} D_\mathbf{z}(\lambda) & y_\mathbf{z}(\lambda) & 0 \\ \hline x_\mathbf{z}(\lambda) & c(\lambda) & x_\mathbf{q}(\lambda) \\ \hline 0 & y_\mathbf{q}(\lambda) & D_\mathbf{q}(\lambda)   \end{array} \right ],
\]
where $ D_\mathbf{q}(\lambda) \in \mathbb{F}[\lambda]^{nh\times nh}$, $D_\mathbf{z}(\lambda)\in \mathbb{F}[\lambda]^{n(k-h-1) \times n(k-h-1)}$, $c(\lambda)\in\mathbb{F}[\lambda]^{n\times n}$, $x_\mathbf{z}(\lambda)\in \mathbb{F}[\lambda]^{n \times n(k-h-1)}$, $x_\mathbf{q}(\lambda) \in \mathbb{F}[\lambda]^{n \times nh}$, $y_\mathbf{z}(\lambda)\in \mathbb{F}[\lambda]^{n(k-h-1)\times n}$ and $y_\mathbf{q}(\lambda) \in \mathbb{F}[\lambda]^{nh \times n}$,
and where the pencils
\[
F(\lambda) := \left[ \begin{array}{cc} c(\lambda) & x_q(\lambda) \\ y_q(\lambda) & D_q(\lambda)\end{array} \right], \quad \mbox{and} \quad G(\lambda) := \left[ \begin{array}{cc} D_z(\lambda) & y_z(\lambda) \\ x_z(\lambda) & c(\lambda) \end{array} \right],
 \]
are GFPR associated with $Q(\lambda):= \lambda^{h+1} A_{h+1} + \lambda^h A_h + \cdots + \lambda A_1 + A_0$ and $Z(\lambda):= \lambda^{k-h} A_{k} + \lambda^{k-h-1}  A_{k-1} + \cdots +  \lambda A_{h+1}+ A_h$, respectively.
More precisely, we have
\[
F(\lambda) = M_{\boldsymbol\ell_q}(\mathcal{X}) ( \lambda M^Q_{-h-1} - M^Q_{\mathbf{q}}) M_{\mathbf{r}_q}(\mathcal{Y})
\]
and
\[
G(\lambda) = M_{\boldsymbol\ell_z}(\mathcal{Z}) ( \lambda M^Z_{h + \mathbf{z}} - M^Z_{0}) M_{\mathbf{r}_z}(\mathcal{W}).
\]
\end{lemma}

\section{Matrix pencils block-permutationally equivalent to extended block Kronecker pencils}

As we have mentioned before, the main goal of the paper is to prove that almost all Fiedler-like pencils  are  block-permutationally equivalent to an extended block Kronecker pencil.
To help us with this task,  we introduce in this section some useful notation and some technical results concerning pencils that are block-permutationally equivalent to block Kronecker pencils or to extended block Kronecker pencils.

\begin{definition}\label{def:permutationally_eq}
Let $L(\lambda)\in\mathbb{F}[\lambda]^{kn\times kn}$ be a $k\times k$ block-pencil with block-entries of size $n\times n$.  
Assume that there exist block-permutation matrices  $\Pi_{\boldsymbol\ell}^n$ and $\Pi_{\mathbf{r}}^n$  such that $C(\lambda):=(\Pi_{\boldsymbol\ell}^n)^{\mathcal{B}} L(\lambda) \Pi_{\mathbf{r}}^n$ is an extended $(p,n,q,n)$-block Kronecker pencil with $k=p+q+1$, partitioned as in \eqref{LDP}.
We call $M(\lambda)$ the \emph{body of $L(\lambda)$ }relative to $(\boldsymbol\ell,\mathbf{r}, p, q)$, and we call the \emph{body block-rows} (resp. \emph{body block-columns}) of $L(\lambda)$ relative to $(\boldsymbol\ell,{\mathbf{r}}, p, q)$ the block-rows (resp. block-columns) of $L(\lambda)$ that, after the permutations, occupy the first $q+1$ (resp. $p+1$) block-rows (resp. block-columns) of $C(\lambda)$. 
Additionally, we call the \emph{wing block-rows} (resp. 
\emph{wing block-columns}) of $L(\lambda)$ relative to $(\boldsymbol\ell,\mathbf{r}, p, q)$ the block-rows (resp. block-columns) of $L(\lambda)$ that are not body block-rows (resp. body block-columns) relative to $(\boldsymbol\ell,{\mathbf{r}}, p, q)$. 
\end{definition}

In the following example we illustrate all the notions introduced in  Definition \ref{def:permutationally_eq}.
\begin{example}\label{ex:permutationally_eq}
 Let us consider again the proper GFP associated with the polynomial $P(\lambda)=\sum_{i=0}^6 A_i\lambda^i\in\mathbb{F}[\lambda]^{n\times n}$ considered in Section \ref{sec:informal}, that is, the pencil 
\[
 K_{\mathbf{q},\mathbf{z}}(\lambda)=
\begin{bmatrix}
-I_n & \lambda A_6 & 0 & 0 & 0 & 0 \\
\lambda I_n & \lambda A_5+A_4 & -I_n & 0 & 0 & 0 \\
0 & A_3 & \lambda I_n & A_2 & -I_n & 0 \\
0 & -I_n & 0 & \lambda I_n & 0 & 0 \\
0 & 0 & 0 & -I_n & 0 & \lambda I_n \\
0 & 0 & 0 & 0 & \lambda I_n & \lambda A_1+A_0
\end{bmatrix}.
\]
 As we showed, we can transform  $K_{\mathbf{q},\mathbf{z}}(\lambda)$ into an (extended) block Kronecker pencil via  block-row and block-column  permutations, denoted by  $\Pi_{\boldsymbol\ell_2}^n$ and $\Pi_{\mathbf{r_2}}^n$, to obtain
\[
(\Pi_{\boldsymbol\ell_2}^n)^{\mathcal{B}} K_{\mathbf{q},\mathbf{z}}(\lambda) \Pi_{\mathbf{r_2}}^n =
\left[\begin{array}{ccc|ccc}
\lambda A_6 & 0 & 0 & -I_n & 0 & 0 \\
\lambda A_5+A_4 & 0 & 0 & \lambda I_n & -I_n & 0 \\
A_3 & A_2 & 0 & 0 & \lambda I_n & -I_n \\
0 & 0 & \lambda A_1+A_0 & 0 & 0 & \lambda I_n \\ \hline 
-I_n & \lambda I_n & 0 & 0 & 0 & 0 \\
0 & -I_n & \lambda I_n & 0 & 0 & 0
\end{array}\right],
\]
which is a (2,$n$,3,$n$)-block Kronecker pencil.

Note that the first, second, third and sixth block-rows of  $K_{\mathbf{q},\mathbf{z}}(\lambda)$ are its body block-rows relative to $(\mathbf{\boldsymbol\ell},\mathbf{r},2,3)$, and the second, fourth and sixth block-columns of $K_{\mathbf{q},\mathbf{z}}(\lambda)$ are its body block-columns relative to $(\boldsymbol\ell,\mathbf{r},2,3)$.
Moreover, the fourth and fifth block-rows of  $K_{\mathbf{q},\mathbf{z}}(\lambda)$ are its wing block-rows relative to $(\boldsymbol\ell,\mathbf{r},2,3)$, and the first, third, and fifth block-columns of  $K_{\mathbf{q},\mathbf{z}}(\lambda)$ are its wing block-columns relative to $(\boldsymbol\ell,\mathbf{r},2,3)$.
\end{example}

\begin{remark}\label{rem:permutationally_1}
Let $L(\lambda)\in\mathbb{F}[\lambda]^{kn\times kn}$ be a matrix pencil  block-permutationally equivalent to an extended $(p,n,q,n)$-block Kronecker pencil with $p+q+1=k$, and let us denote by $\Pi_{\mathbf{r}}^n$ and $\Pi_{\boldsymbol\ell}^n$ the block-permutation matrices that transform the pencil $L(\lambda)$ in the extended $(p,n,q,n)$-block Kronecker pencil $(\Pi_{\boldsymbol\ell}^n)^\mathcal{B}L(\lambda)\Pi_{\mathbf{r}}^n$.
Let $\{r_1,\hdots,r_p\}$, $\{r_{p+1},\hdots,r_k\}$, $\{c_1,\hdots,c_q\}$, $\{c_{q+1},\hdots,c_k\}$ be, respectively, the sets of  positions of the wing block-rows, body block-rows, wing block-columns and body block-columns of $L(\lambda)$, all relative to $(\boldsymbol\ell,\mathbf{r},p,q)$, and let $L_1(\lambda):= (\Pi_{\boldsymbol\ell}^n)^\mathcal{B} L(\lambda)$ and $L_2(\lambda):=L(\lambda)\Pi_{\mathbf{r}}^n$.
Then, the following  simple observations will be used freely.
\begin{itemize}
\item[\rm (i)] The pencils $L_1(\lambda)$ and $L_2(\lambda)$ are  block-permutationally equivalent to the extended $(p,n,q,n)$-block Kronecker pencil $(\Pi_{\boldsymbol\ell}^n)^\mathcal{B}L(\lambda)\Pi_{\mathbf{r}}^n$.
\item[\rm (ii)] The sets $\{1:q+1\}$ and $\{q+2:k\}$ are, respectively, the sets of  positions of the  body block-rows and the wing block-rows of $L_1(\lambda)$, both relative to $(\boldsymbol{\mathrm{id}},\mathbf{r},p,q)$.
\item[\rm (iii)] The sets $\{c_1,\hdots,c_q\}$, $\{c_{q+1},\hdots,c_k\}$ are, respectively, the sets of  positions of the wing block-columns and the  body block-columns of $L_1(\lambda)$, both relative to $(\boldsymbol{\mathrm{id}},\mathbf{r},p,q)$.
\item[\rm (iv)] The sets $\{1:p+1\}$ and $\{p+2:k\}$ are, respectively, the sets of  positions of the body block-columns and the  wing block-columns of $L_2(\lambda)$, both relative to $(\boldsymbol\ell,\boldsymbol{\mathrm{id}},p,q)$.
\item[\rm (v)] The sets $\{r_1,\hdots,r_p\}$, $\{r_{p+1},\hdots,r_k\}$ are, respectively, the sets of  positions of the wing block-rows and the body block-rows of $L_2(\lambda)$, both relative to $(\boldsymbol\ell,\boldsymbol{\mathrm{id}},p,q)$.
\end{itemize}
\end{remark}

The following two lemmas will be useful for the proofs of the main results of this  paper.
Their simple proofs are omitted.
\begin{lemma}\label{lem:permutationally_2}
Let $L_1(\lambda),L_2(\lambda)\in\mathbb{F}[\lambda]^{kn\times kn}$ be  matrix pencils, and assume that there exist block-permutation matrices $\Pi_{\mathbf{r}}^n$, $\Pi_{\boldsymbol\ell_1}^n$ and $\Pi_{\boldsymbol\ell_2}^n$ such that $(\Pi_{\boldsymbol\ell_1}^n)^\mathcal{B}L_1(\lambda)\Pi_{\mathbf{r}}^n$ and $(\Pi_{\boldsymbol\ell_2}^n)^\mathcal{B}L_2(\lambda)\Pi_{\mathbf{r}}^n$ are both extended $(p,n,q,n)$-block Kronecker pencils.
Then, the sets of  positions of the body block-columns and the wing block-columns of $L_1(\lambda)$ relative to $(\boldsymbol\ell_1,\mathbf{r},p,q)$ equal, respectively,  the sets of  positions of the body block-columns and the wing block-columns of $L_2(\lambda)$ relative to $(\boldsymbol\ell_2,\mathbf{r},p,q)$.
\end{lemma}

\begin{lemma}\label{lem:permutationally_3}
Let $L_1(\lambda),L_2(\lambda)\in\mathbb{F}[\lambda]^{kn\times kn}$ be matrix pencils, and assume that there exist block-permutation matrices $\Pi_{\mathbf{r}_1}^n$, $\Pi_{\mathbf{r}_2}^n$ and $\Pi_{\boldsymbol\ell}^n$ such that $(\Pi_{\boldsymbol\ell}^n)^\mathcal{B}L_1(\lambda)\Pi_{\mathbf{r}_1}^n$ and $(\Pi_{\boldsymbol\ell}^n)^\mathcal{B}L_2(\lambda)\Pi_{\mathbf{r_2}}^n$ are both extended $(p,n,q,n)$-block Kronecker pencils.
Then, the sets of  positions of the body block-rows and the wing block-rows of $L_1(\lambda)$ relative to $(\boldsymbol\ell,\mathbf{r}_1,p,q)$ equal, respectively,  the sets of  positions of the body block-rows and the wing block-rows of $L_2(\lambda)$ relative to $(\boldsymbol\ell,\mathbf{r}_2,p,q)$.
\end{lemma}

In the next two subsections we include two  technical lemmas that show how the multiplication of pencils block-permutationally equivalent to an (extended) block Kronecker pencil by elementary matrices, in certain situations, produce new pencils permutationally equivalent to an extended block Kronecker pencil.

\subsection{Product by elementary matrices associated with negative indices}

The proof of the main result for GFP (Theorem \ref{thm:informal1}  or Theorem \ref{thm:main:GFP}) requires one key auxiliary result. 
This result is Lemma \ref{lem:left_multiplication2}, which studies the effect of left-multiplications by elementary matrices associated with negative indices in some relevant situations.

\begin{lemma}\label{lem:left_multiplication2}
Let $P(\lambda)\in\mathbb{F}[\lambda]^{n\times n}$ be a matrix polynomial of grade $k$,  let $L(\lambda)\in\mathbb{F}[\lambda]^{kn\times kn}$ be a pencil  block-permutationally equivalent to a $(p,n,q,n)$-block Kronecker pencil with $p+q+1=k$, i.e., there exist  block-permutation matrices $\Pi_{\boldsymbol\ell}^n,\Pi_{\mathbf{r}}^n$ such that 
 \[
 (\Pi_{\boldsymbol\ell}^n)^\mathcal{B} L(\lambda) \Pi_{\mathbf{r}}^n =
\left[\begin{array}{c|c}
M(\lambda) & L_q(\lambda)^T\otimes I_n \\ \hline
L_p(\lambda)\otimes I_n & 0
\end{array}\right].
 \]
 Assume that  $M(\lambda)$ satisfies the AS condition for $P(\lambda)$.
 Let $\{r_1,\hdots,r_p\}$, $\{r_{p+1},\hdots,r_k\}$, $\{c_1,\hdots,c_q\}$, $\{c_{q+1},\hdots,c_k\}$ be, respectively, the sets of  positions of the wing block-rows, the body block-rows, the wing block-columns and the body block-columns of $L(\lambda)$, all relative to $(\boldsymbol\ell,\mathbf{r},p,q)$.
 Let $x\in\{1:k-1\}$ be such that  $k-x+1\in\{r_1,\hdots,r_p\}$ and $k-x\in\{r_{p+1},\hdots,r_k\}$.
Define $\widetilde{L}(\lambda):=M_{-x}(X)L(\lambda)$, where $X\in \mathbb{F}^{n\times n}$.
Then, the following statements hold.
\begin{enumerate}
\item[\rm (i)] There exists a block-permutation matrix $ \Pi_{\widetilde{\boldsymbol\ell}}^n$ such that
 \begin{equation}\label{eq:left_multiplication_negative}
 (\Pi_{\widetilde{\boldsymbol\ell}}^n)^\mathcal{B} \widetilde{L}(\lambda) \Pi_{\mathbf{r}}^n =
\left[\begin{array}{c|c}
\widetilde{M}(\lambda) & L_q(\lambda)^T\otimes I_n \\ \hline
L_p(\lambda)\otimes I_n & \phantom{\Big{(}} 0 \phantom{\Big{(}}
\end{array}\right]
 \end{equation}
is an extended $(p,n,q,n)$-block Kronecker pencil whose body $\widetilde{M}(\lambda)$ satisfies the AS condition for $P(\lambda)$.
\item[\rm (ii.a)] For $j=1:k$, if $j\neq k-x,k-x+1$, then the $j$th block-row of $\widetilde{L}(\lambda)$   is equal to the $j$th block-row of $L(\lambda)$. 
\item[\rm (ii.b)] The $(k-x)$th block-row of $\widetilde{L}(\lambda)$ is equal to the $(k-x+1)$th block-row of $L(\lambda)$. 
\item[\rm (ii.c)] The set $(\{r_1,\hdots,r_p\}\setminus\{k-x+1\})\cup\{k-x\}$ is the set of  positions of the wing block-rows  of $\widetilde{L}(\lambda)$ relative to $(\widetilde{\boldsymbol\ell},\mathbf{r},p,q)$.
\item[\rm (ii.d)] The set $(\{r_{p+1},\hdots,r_k\}\setminus\{k-x\})\cup\{k-x+1\}$ is the set of  positions of the body block-rows of $\widetilde{L}(\lambda)$ relative to $(\widetilde{\boldsymbol\ell},\mathbf{r},p,q)$.
\item[\rm (iii.a)] The set $\{c_1,\hdots,c_q\}$ is the set of  positions of the wing block-columns of $\widetilde{L}(\lambda)$ relative to $(\widetilde{\boldsymbol\ell},\mathbf{r},p,q)$.
\item[\rm (iii.b)] The set $\{c_{q+1},\hdots,c_k\}$ is the set of  positions of the body block-columns of $\widetilde{L}(\lambda)$ relative to $(\widetilde{\boldsymbol\ell},\mathbf{r},p,q)$.
\end{enumerate} 
\end{lemma}
\begin{proof}
Notice that the product on the left by the matrix $M_{-x}(X)$ only affects the $(k-x)$th and the $(k-x+1)$th block-rows of $L(\lambda)\Pi_{\mathbf{r}}^n$.
Moreover, since the $(k-x+1)$th block-row of $L(\lambda)$ is a wing block-row relative to $(\boldsymbol\ell,\mathbf{r},p,q)$, the submatrix of $L(\lambda)\Pi_{\mathbf{r}}^n$  consisting of its $(k-x)$th and $(k-x+1)$th block-rows is of the form 
\[
\begin{bmatrix}
R_1(\lambda) & R_2(\lambda) & R_3(\lambda) & R_4(\lambda) & R_5(\lambda) \\
0 & -I_n & \lambda I_n & 0 & 0 
\end{bmatrix},
\]
for some matrix pencils $R_1(\lambda)$, $R_2(\lambda)$, $R_3(\lambda)$, $R_4(\lambda)$ and $R_5(\lambda)$, where some of the block-columns containing a zero block other than the last one may not be present.
This, in turn, implies that the submatrix of $M_{-x}(X)L(\lambda)\Pi_{\mathbf{r}}^n$ formed by its $(k-x)$th and $(k-x+1)$th block-rows is of the form 
\[
\begin{bmatrix}
0 & -I_n & \lambda I_n & 0 & 0 \\
R_1(\lambda) & R_2(\lambda)-X & R_3(\lambda)-\lambda X & R_4(\lambda) & R_5(\lambda) 
\end{bmatrix}.
\]
Therefore, setting
\[
\mathbf{t}:=(1:k-x-1,k-x+1,k-x,k-x+2:k)
\]
and introducing the block-permutation matrix  $\Pi_{\widetilde{\boldsymbol\ell}}^n:=\Pi_{\mathbf{t}}^n\Pi_{\boldsymbol\ell}^n$, it is clear that \eqref{eq:left_multiplication_negative} holds for some pencil $\widetilde{M}(\lambda)$.
Furthermore,  note that $\widetilde{M}(\lambda)=M(\lambda)+(e_t\otimes I_n)\begin{bmatrix} 0 & -I_n & \lambda I_n & 0 & 0 \end{bmatrix}$,  for some $t\in\{0:q+1\}$.
Thus, the fact that the body of $\widetilde{L}(\lambda)$ satisfies the AS condition for $P(\lambda)$ follows from Lemma \ref{lemm:perturbation_body}.
Therefore, part (i) is true.
Parts (ii.a), (ii.b), (ii.c) and (ii.d) follows from applying Lemma \ref{lem:permutationally_3} to $L(\lambda)$ and $(\Pi_{\mathbf{t}}^n)^\mathcal{B}L(\lambda)$, together with the simple fact that left-multiplication by $M_{-x}(X)$ only affects the $(k-x)$th and $(k-x+1)$th block-rows of $L(\lambda)$, and replaces the $(k-x)$th block-row by the $(k-x+1)$th block-row.
Finally, parts (iii.a) and (iii.b) are direct consequences of Lemma \ref{lem:permutationally_2}.
\end{proof}

\subsection{Product by elementary matrices associated with nonnegative indices}

The proof of the main result for GFPR (Theorem \ref{thm:informal2}  or Theorem \ref{thm:main_GFPR}) requires one key auxiliary result. 
This result is Lemma \ref{lem:right_multiplication}, which studies the effect of right-multiplications by elementary matrices associated with nonnegative indices in some relevant situations.

\begin{lemma}\label{lem:right_multiplication}
Let $P(\lambda)\in\mathbb{F}[\lambda]^{n\times n}$ be a matrix polynomial of grade $k$,  let $L(\lambda)\in\mathbb{F}[\lambda]^{kn\times kn}$ be a pencil block-permutationally equivalent to an extended $(p,n,q,n)$-block Kronecker pencil with $p+q+1=k$, i.e., there exist block-permutation matrices $\Pi_{\boldsymbol\ell}^n,\Pi_{\mathbf{r}}^n$ such that 
 \[
 (\Pi_{\boldsymbol\ell}^n)^\mathcal{B} L(\lambda) \Pi_{\mathbf{r}}^n =
\left[\begin{array}{c|c}
M(\lambda) & K_2(\lambda)^T \\ \hline
K_1(\lambda) & 0
\end{array}\right],
 \]
where $K_1(\lambda)$ and $K_2(\lambda)$ are wing pencils, and assume that  $M(\lambda)$ satisfies the AS condition for $P(\lambda)$.
 Let $\{r_1,\hdots,r_p\}$, $\{r_{p+1},\hdots,r_k\}$, $\{c_1,\hdots,c_q\}$ and  $\{c_{q+1},\hdots,c_k\}$ be, respectively, the sets of  positions of the wing block-rows, the body block-rows, the wing block-columns and the body block-columns of $L(\lambda)$, all relative to $(\boldsymbol\ell,\mathbf{r},p,q)$.
 Assume, additionally, that $x$ is an index such that $k-x\in\{c_1,\hdots,c_q\}$.
Define $\widetilde{L}(\lambda):=L(\lambda)M_x(X)$, where $X\in \mathbb{F}^{n\times n}$.
Then, the following statements hold.
\begin{enumerate}
\item[\rm (i)] There exists a block-permutation matrix $ \Pi_{\widetilde{\mathbf{r}}}^n$ such that
 \begin{equation}\label{eq:right_multiplication}
 (\Pi_{\boldsymbol\ell}^n)^\mathcal{B} \widetilde{L}(\lambda) \Pi_{\widetilde{\mathbf{r}}}^n =
\left[\begin{array}{c|c}
\widetilde{M}(\lambda) & \widetilde{K}_2(\lambda)^T \\ \hline
K_1(\lambda) & \phantom{\Big{(}} 0 \phantom{\Big{(}}
\end{array}\right]
 \end{equation}
is an extended  $(p,n,q,n)$-block Kronecker pencil whose body $\widetilde{M}(\lambda)$ satisfies the AS condition for $P(\lambda)$.
\item[\rm (ii.a)] If $x=0$, the $k$th block-column of $\widetilde{L}(\lambda)$ is a wing block-column of $\widetilde{L}(\lambda)$ relative to $(\boldsymbol\ell,\widetilde{\mathbf{r}},p,q)$.
\item[\rm (ii.b)] If $x\neq 0$, then the $(k-x+1)$th block-column of $\widetilde{L}(\lambda)$ is a wing block-column of $\widetilde{L}(\lambda)$ relative to $(\boldsymbol\ell,\widetilde{\mathbf{r}},p,q)$, and it is equal to the $(k-x)$th block-column of $L(\lambda)$.
\item[\rm (ii.c)] If $x\neq 0$, then the $(k-x)$th block-column of $\widetilde{L}(\lambda)$ is a wing block-column of $\widetilde{L}(\lambda)$ relative to $(\boldsymbol\ell,\widetilde{\mathbf{r}},p,q)$ if and only if the $(k-x+1)$th block-column of $L(\lambda)$ is a wing block-column of $L(\lambda)$ relative to $(\boldsymbol\ell,\mathbf{r},p,q)$.
\item[\rm (iii.a)] For $j=1:q$, if $c_j\neq k-x,k-x+1$, then the $c_j$-th block-column of $\widetilde{L}(\lambda)$  is a wing block-column of $\widetilde{L}(\lambda)$  relative to $(\boldsymbol\ell,\widetilde{\mathbf{r}},p,q)$, and it is equal to the $c_j$-th block-column of $L(\lambda)$.
\item[\rm (iii.b)] For $j=q+1:k$, if $c_j\neq k-x+1$, then the $c_j$-th block-column of $\widetilde{L}(\lambda)$  is a body block-column of $\widetilde{L}(\lambda)$ relative to $(\boldsymbol\ell,\widetilde{\mathbf{r}},p,q)$.
\item[\rm (iv.a)] The set $\{r_1,\hdots,r_p\}$ is the set of  positions of the wing block-rows of $\widetilde{L}(\lambda)$ relative to $(\boldsymbol\ell,\widetilde{\mathbf{r}},p,q)$.
\item[\rm (iv.b)] The set $\{r_{p+1},\hdots,r_k\}$ is the set of  positions of the body block-rows of $\widetilde{L}(\lambda)$ relative to $(\boldsymbol\ell,\widetilde{\mathbf{r}},p,q)$.
\end{enumerate} 
\end{lemma}
\begin{proof}
First, let us assume that $x=0$.
By the hypotheses of the theorem, the $k$th block-column of $L(\lambda)$ is a wing block-column relative to $(\boldsymbol\ell,\mathbf{r},p,q)$.
This implies that the wing block-rows of $L(\lambda)$ and the block-columns of $L(\lambda)$ other than its $k$th block-column are not changed after the right-multiplication of $L(\lambda)$ by $M_0(X)$.
Thus, setting $\Pi_{\widetilde{\mathbf{r}}}^n=\Pi_{\mathbf{r}}^n$, we obtain that \eqref{eq:right_multiplication} holds with $\widetilde{M}(\lambda)=M(\lambda)$, for some matrix pencil $\widetilde{K}_2(\lambda)\in\mathbb{F}[\lambda]^{nq\times n(q+1)}$.
Moreover, we claim that $\widetilde{K}_2(\lambda)$ is a wing pencil.
This claim follows from the fact that $\widetilde{K}_2(\lambda)$ is obtained by multiplying $K_2(\lambda)$ by a block diagonal matrix whose diagonal blocks are $I_n$ except for the last one, which is equal to $X$, together with Remark \ref{rem:block_structure}.
Thus, part (i) is true when $x=0$.
Part (ii.a) also follows immediately.
Parts (iii.a) and (iii.b) follow from Lemma \ref{lem:permutationally_2} and the fact that the block-columns of $L(\lambda)$ and $\widetilde{L}(\lambda)$, other than their $k$th block-columns, are equal.
Finally, parts  (iv.a) and (iv.b) follow from  Lemma \ref{lem:permutationally_3}.

Let us assume, now, that $x\neq 0$.
By the hypotheses of the theorem, the $(k-x)$th block-column of  $L(\lambda)$ is a wing block-column relative to $(\boldsymbol\ell,\mathbf{r},p,q)$. 
Notice that the right-multiplication of $L(\lambda)$ by the matrix $M_x(X)$ only affects the $(k-x)$th and the $(k-x+1)$th block-columns of  $L(\lambda)$.
Then, to prove all the results, we have to distinguish two cases.

\noindent {\bf Case I}: Assume that the $(k-x+1)$th block-column of $L(\lambda)$ is a body block-column of $L(\lambda)$ relative to $(\boldsymbol\ell,\mathbf{r},p,q)$.
Thus, the $(k-x)$th and the $(k-x+1)$th block-columns of $(\Pi_{\boldsymbol\ell}^n)^\mathcal{B}L(\lambda)$ are, respectively, of the form $\begin{bmatrix} R(\lambda)^T & 0 \end{bmatrix}^T$  and $\begin{bmatrix}  B(\lambda)^T & C(\lambda)^T \end{bmatrix}^T$, with $B(\lambda),R(\lambda)\in\mathbb{F}[\lambda]^{ n(q+1)\times n}$ and $C(\lambda)\in\mathbb{F}[\lambda]^{np\times n}$.
Then, set 
\[
\mathbf{t} := (1:k-x-1,k-x+1,k-x,k-x+2:k),
\]
and let $\widehat{L}(\lambda):=\widetilde{L}(\lambda)\Pi^n_{\mathbf{t}}$.
Notice that the submatrices of $(\Pi_{\boldsymbol\ell}^n)^{\mathcal{B}}L(\lambda)$, $(\Pi_{\boldsymbol\ell}^n)^{\mathcal{B}}\widetilde{L}(\lambda)$ and $(\Pi_{\boldsymbol\ell}^n)^{\mathcal{B}}\widehat{L}(\lambda)$ formed by their $(k-x)$th and $(k-x+1)$th block-columns are, respectively, given by
\begin{align*}
&\begin{bmatrix}
R(\lambda) & B(\lambda) \\ 0 & C(\lambda)
\end{bmatrix}, \quad  
\begin{bmatrix}
R(\lambda) & B(\lambda) \\ 0 & C(\lambda)
\end{bmatrix}
\begin{bmatrix}
X & I_n \\ I_n & 0
\end{bmatrix}=  
\begin{bmatrix}
R(\lambda)X + B(\lambda) & R(\lambda) \\
C(\lambda) & 0 
\end{bmatrix}, \quad \mbox{and} \\
&\begin{bmatrix}
R(\lambda) & B(\lambda) \\ 0 & C(\lambda)
\end{bmatrix}
\begin{bmatrix}
X & I_n \\ I_n & 0
\end{bmatrix}
\begin{bmatrix}
0 & I_n \\ I_n & 0
\end{bmatrix}=
\begin{bmatrix}
R(\lambda) & R(\lambda)X + B(\lambda) \\
0 & C(\lambda)
\end{bmatrix}.
\end{align*}
This implies that all the wing block-columns and the wing block-rows of $L(\lambda)$ relative to $(\boldsymbol\ell,\mathbf{r},p,q)$ remain unchanged after the right-multiplication by $M_x(X)\Pi_{\mathbf{t}}^n$.
Therefore, setting  $ \Pi_{\widetilde{\mathbf{r}}}^n:= \Pi_{\mathbf{t}}^n \Pi_{\mathbf{r}}^n $, we obtain that \eqref{eq:right_multiplication} holds with $\widetilde{K}_2(\lambda)=K_2(\lambda)$ and $\widetilde{M}(\lambda) = M(\lambda)+ R(\lambda)X(e_t^T\otimes I_n)$, for some $t\in\{1:p+1\}$, where $e_t$ denotes the $t$th column of the $(p+1)\times (p+1)$ identity matrix.
Note that $(\Lambda_q(\lambda)\otimes I_n)R(\lambda)X=0$ since $R(\lambda)$ is a block-column of the wing pencil $K_2(\lambda)^T$.
Thus, from Lemma \ref{lemm:perturbation_body}, we obtain that the pencil $\widetilde{M}(\lambda)$ satisfies the AS condition for $P(\lambda)$, and part (i) follows.
Parts  (ii.b), (iii.a) and (iii.b) follow from Lemma \ref{lem:permutationally_2} applied to $L(\lambda)$ and $\widehat{L}(\lambda)$, together with the simple fact that the only difference between $\widehat{L}(\lambda)$ and $\widetilde{L}(\lambda)$  is that their $(k-x)$th and $(k-x+1)$th block-columns are permuted.
Parts (iv.a) and (iv.b) follow from Lemma \ref{lem:permutationally_3} applied to $L(\lambda)$ and $\widetilde{L}(\lambda)$.

\medskip

\noindent {\bf Case II}: Assume that the $(k-x+1)$th block-column of $L(\lambda)$ is a wing block-column relative to $(\boldsymbol\ell,\mathbf{r},p,q)$.
Thus, the $(k-x)$th and the $(k-x+1)$th block-columns of $(\Pi_{\boldsymbol\ell}^n)^\mathcal{B}L(\lambda)$ are, respectively, of the form  $\begin{bmatrix} R_1(\lambda)^T & 0 \end{bmatrix}^T$ and $\begin{bmatrix} R_2(\lambda)^T & 0 \end{bmatrix}^T$, with $R_1(\lambda),R_2(\lambda)\in\mathbb{F}[\lambda]^{n(q+1)\times n}$.
Then, notice that the  submatrix formed by the $(k-x)$th and the $(k-x+1)$th block-columns of the pencil $(\Pi_{\boldsymbol\ell}^n)^\mathcal{B}L(\lambda)M_x(X)$   is given by
\[
\begin{bmatrix}
R_1(\lambda) & R_2(\lambda) \\ 0 & 0
\end{bmatrix}
\begin{bmatrix}
X & I_n \\ I_n & 0
\end{bmatrix} = 
\begin{bmatrix}
R_1(\lambda)X + R_2(\lambda) & R_1(\lambda) \\
0 & 0 
\end{bmatrix}.
\]
Therefore, setting $\Pi_{\widetilde{\mathbf{r}}}^n = \Pi_{\mathbf{r}}^n$, we obtain that \eqref{eq:right_multiplication} holds with $\widetilde{M}(\lambda)=M(\lambda)$ and    $\widetilde{K}_2(\lambda)^T = K_2(\lambda)^T+ R_1(\lambda)X(e_t^T\otimes I_n)$, for some $t\in\{1:q\}$, where $e_t$ denotes the $t$th column of the $q\times q$ identity matrix.
Since $\widetilde{K}_2(\lambda)$ can be written as $BK_2(\lambda)$, for some nonsingular matrix $B$,  by Corollary \ref{cor:wing_pencils}, $\widetilde{K}_2(\lambda)$ is a  wing pencil, and part (i) follows.
Note that the body of $\widetilde{L}(\lambda)$ relative to  $(\boldsymbol\ell,\mathbf{r},p,q)$ and the body of $L(\lambda)$ relative to  $(\boldsymbol\ell,\mathbf{r},p,q)$ are the same, and therefore, the body of $\widetilde{L}(\lambda)$ satisfies the AS condition for $P(\lambda)$.
 Parts (ii.b), (iii.a) and (iii.b) follow from Lemma \ref{lem:permutationally_2} applied to $L(\lambda)$ and $\widetilde{L}(\lambda)$, together with the simple fact that the $(k-x)$th and the $(k-x+1)$th block-columns of $L(\lambda)$ and $\widetilde{L}(\lambda)$  are permuted.
Parts (iv.a) and (iv.b) follow immediately from  Lemma \ref{lem:permutationally_3}. 

Finally, notice that in the proofs of {\bf Case I}  and  {\bf Case II}, we have shown that the $(k-x)$th block-column of $\widetilde{L}(\lambda)$ is a wing block-column (resp. a body block-column) of $\widetilde{L}(\lambda)$ relative to $(\boldsymbol\ell,\widetilde{\mathbf{r}},p,q)$ whenever the $(k-x+1)$th block-column of $L(\lambda)$ is a wing block-column (resp. a body block-column) of $L(\lambda)$ relative to $(\boldsymbol\ell,\mathbf{r},p,q)$. 
Thus, part (ii.c) holds.
\end{proof}

\section{Fiedler pencils as block Kronecker pencils}
\label{sec:FP_as_minimal_bases_pencils}

In \cite[Theorem 4.5]{canonical} it was proven that all Fiedler pencils associated with a matrix polynomial $P(\lambda)$ are block Kronecker pencils with bodies satisfying the AS condition for $P(\lambda)$,  after permuting some of its block-rows and block-columns. 
The main  theorem of this section (Theorem \ref{FP-col-row}) gives a more detailed description of this block Kronecker form for Fiedler pencils that will be useful to show that almost all Fiedler-like pencils associated with $P(\lambda)$ are extended block Kronecker pencils  with bodies satisfying the AS condition for $P(\lambda)$, modulo permutations. 
However, we warn the reader that Theorem \ref{FP-col-row} is only valid when $n=m$, in contrast to \cite[Theorem 4.5]{canonical} which is valid also when $n\neq m$.

We start by relating the notation used in \cite{DTDM10} to work with and construct Fiedler pencils with the index tuple notation introduced in Section \ref{sec:index_tuples}.
 To this end, we recall the notion of consecutions and inversions of a permutation $\mathbf{q}$ of the set $\{0:k-1\}$.
\begin{definition}
Let  $k$ be a positive integer and let $\mathbf{q}$ be a permutation of the set $\{0:k-1\}$. 
For $i\in \{0:k-2\}$, we say that $\mathbf{q}$ has a \emph{consecution} at $i$ if $i$ 
occurs to the left of $i+1$ in $\mathbf{q}$, and that $\mathbf{q}$ has an \emph{inversion} at $i$ if $i$ occurs to the right  of $i+1$.  
\end{definition}

Next, we relate   the consecutions and inverstions of a permutation $\mathbf{q}$ of $\{0:k-1\}$ with the sets $\mathrm{heads}(\mathbf{q})$ and $\mathrm{heads}(\mathbf{\rev (q)})$ (recall that  $\mathfrak{h}(\mathbf{t})$ denotes the cardinality of the set $\mathrm{heads}(\mathbf{t})$).
In the proof of this lemma and in the rest of the section, we use the following notation.
We denote by  $\mathfrak{i}(\mathbf{q})$  and $\mathfrak{c}(\mathbf{q})$ the total number of inversions and consecutions of $\mathbf{q}$, respectively. 
We also denote by $\mathfrak{C}_q$ and $\mathfrak{I}_q$  the set of indices at which the permutation $\mathbf{q}$ has a consecution and an inversion, respectively.
\begin{lemma}\label{lem:prop-cons}
Let $\mathbf{q}$ be a permutation of $\{0:k-1\}$.
Then,
\begin{enumerate}[(i)]
\item $ \mathfrak{i}(\mathbf{q}) = \mathfrak{h}(\mathbf{q}) -1$ and  $ \mathfrak{c}(\mathbf{q})=\mathfrak{h}(\rev(\mathbf{q})) -1
$; \label{cons-1}
\item $k- \mathfrak{C}_q= \{k - h: h\notin \heads(\mathbf{q})\} $ and $ k - \mathfrak{I}_q=\{k-h: h\notin \heads(\rev(\mathbf{q}))\};$\label{cons-2}
\item $\mathfrak{I}_q \cup \mathfrak{C}_q = \{0:k-2\} \quad \textrm{and} \quad \mathfrak{I}_q\cap \mathfrak{C}_q 
= \emptyset.$\label{cons-3}
\end{enumerate}
\end{lemma}
\begin{proof}
Note that a permutation $\mathbf{q}$ of the set $\{0:k-1\}$ has an inversion at an index $i$ if and only if  $i \neq k-1$ and $i\in \heads(\mathbf{q})$; and $\mathbf{q}$ has a consecution at $i$ if and only if $i\neq k-1$ and $i \notin \heads(\mathbf{q})$, or equivalently, if $i\neq k-1$ and $i\in \heads(\rev(\mathbf{q}))$. 
Thus, taking into account that $k-1\in \heads(\mathbf{q})\cap  \heads(\rev(\mathbf{q}))$, we get parts \eqref{cons-1}
 and \eqref{cons-2}.
Moreover, part \eqref{cons-3} follows since $\mathbf{q}$ cannot have a consecution and an inversion at the same index and $\mathbf{q}$ has neither a consecution nor an inversion at $k-1$. 
\end{proof}

The next theorem is the main result of this section.
In this theorem, we recall that, modulo  block-permutations, Fiedler pencils are block Kronecker pencils, and identify the body block-rows, the body block-columns, the wing block-rows and the wing block-columns of any Fiedler pencil relative to some  block-permutations that transform it into a block Kronecker pencil.
\begin{theorem}\label{FP-col-row}
Let $P(\lambda)\in\mathbb{F}[\lambda]^{n\times n}$ be a matrix polynomial of degree $k$ as in \eqref{pol}. 
Let $\mathbf{q}$ be a permutation of $\{0:k-1\}$ and let 
$$F_{\mathbf{q}}(\lambda)= \lambda M^P_{-k} - M^P_{\mathbf{q}}$$
be the Fiedler pencil associated with $P(\lambda)$ and $\mathbf{q}$.  
Then, there exist block-permutation matrices $\Pi_{\mathbf{r}}^n$ and $\Pi_{\boldsymbol\ell}^n$ such that 
\begin{equation}\label{Fiedler-dual}
C(\lambda):=(\Pi_{\boldsymbol\ell}^n)^{\mathcal{B}} F_{\mathbf{q}}(\lambda) \Pi_{\mathbf{r}}^n= \left[ \begin{array}{c|c} M(\lambda ) &  L_{\mathfrak{h}(\rev(\mathbf{q}))-1}(\lambda)^T\otimes I_n\\ \hline  L_{\mathfrak{h}(\mathbf{q})-1}(\lambda) \otimes I_n & 0  \end{array} \right],
\end{equation}
is  a $(\mathfrak{h}(\mathbf{q})-1,n,\mathfrak{h}(\rev(\mathbf{q}))-1,n)$-block Kronecker pencil  with body $M(\lambda)$ satisfying the AS condition for $P(\lambda)$.
Moreover, the following statements hold.

\begin{enumerate}[(1)]
\item[\rm (a)] The wing block-columns of  $(\Pi_{\boldsymbol\ell}^n)^{\mathcal{B}} F_{\mathbf{q}}(\lambda)$ relative to $(\boldsymbol{\mathrm{id}}, \mathbf{r}, \mathfrak{h}(\mathbf{q})-1, \mathfrak{h}(\rev(\mathbf{q}) )-1)$ are  of the form $-e_i \otimes I_n + \lambda e_{i+1} \otimes I_n$, for $1\leq i \leq \mathfrak{h}(\rev(\mathbf{q}))-1$, 
and are  located in  positions $k-j$, where $j\in \{0:k-2\}$ and $j \notin \heads(\mathbf{q})$, or, equivalently,  $j\in \{0:k-2\}$ and $(\mathbf{q}, j)$ satisfies the SIP.

\item[\rm (b)] The wing block-rows of $F_{\mathbf{q}}(\lambda) \Pi_{\mathbf{r}}^n$ relative to $(\boldsymbol\ell,\boldsymbol{\mathrm{id}},\mathfrak{h}(\mathbf{q})-1,\mathfrak{h}(\rev(\mathbf{q}) )-1)$ are of the form $-e_i^T \otimes I_n + \lambda e_{i+1}^T \otimes I_n$, for $1\leq i < \mathfrak{h}(\mathbf{q})-1$,  and are  located in  positions $k-j$, where $j\in \{0:k-2\}$ and $j \notin \heads(\rev(\mathbf{q}))$, or, equivalently,  $j\in \{0:k-2\}$ and $(j, \mathbf{q})$ satisfies the SIP.

\item[\rm (c)]  The first block-row and the first block-column of $F_{\mathbf{q}}(\lambda)$ are, respectively, the first body block-row and the first body block-column of $F_{\mathbf{q}}(\lambda)$ relative to $(\boldsymbol\ell,\mathbf{r}, \mathfrak{h}(\mathbf{q})-1,\mathfrak{h}(\rev(\mathbf{q}))-1)$. Moreover, the block-entry of $M(\lambda)$ in position $(1,1)$ equals $\lambda A_k + A_{k-1}$.

\end{enumerate}
\end{theorem} 
\begin{proof}
Let $p:=\mathfrak{h}(\mathbf{q})-1$ and $q:=\mathfrak{h}(\rev(\mathbf{q}))-1$.
By \cite[Theorem 4.5]{canonical}, and taking into account  part (\ref{cons-1}) in Lemma \ref{lem:prop-cons}, there exist  block-permutation matrices  $\Pi_{\boldsymbol\ell}^n$ and $\Pi_{\mathbf{r}}^n$  such that  \eqref{Fiedler-dual} holds with $M(\lambda)$ with the so-called  staircase pattern for $\lambda A_k+A_{k-1},A_{k-2},\hdots,A_0$ (see Definition 4.3 and Theorem 4.5 in \cite{canonical}).
Thus, the Fiedler pencil $F_{\mathbf{q}}(\lambda)$ is  block-permutationally equivalent to a $(p,n,q,n)$-block Kronecker pencil.
Furthermore, it is not difficult to check that the staircase pattern of $M(\lambda)$ implies that it satisfies the AS condition for $P(\lambda)$.

We now prove  parts (a) and (b).
It is clear, by \eqref{Fiedler-dual}, that the wing block-columns of $(\Pi^n_{\boldsymbol\ell})^{\mathcal{B}} F_{\mathbf{q}}(\lambda)$ relative to $(\boldsymbol{\mathrm{id}}, \mathbf{r}, p, q)$ are of the form $-e_i \otimes I_n + \lambda e_{i+1} \otimes I_n$ for some $i\leq q$, and that the wing block-rows of $F_{\mathbf{q}}(\lambda)\Pi_{\mathbf{r}}^n$ relative to $(\boldsymbol\ell, \boldsymbol{\mathrm{id}}, p, q)$ are of the form $-e_i^T \otimes I_n +\lambda  e_{i+1}^T \otimes I_n$, for some $i\leq p$.
This implies the first claim in parts (a) and (b). 

To prove the second claim in parts (a) and (b), notice that the wing block-columns of  $F_{\mathbf{q}}(\lambda)$ relative to $(\boldsymbol\ell, \mathbf{r}, p, q)$ and the wing block-columns of  $(\Pi_{\boldsymbol\ell}^n)^{\mathcal{B}} F_{\mathbf{q}}(\lambda)$ relative to $(\boldsymbol{\mathrm{id}}, \mathbf{r}, p, q)$ are in exactly the same positions.
Similarly, the wing block-rows of $F_{\mathbf{q}}(\lambda)$ relative to $(\boldsymbol\ell, \mathbf{r}, p, q)$ and the wing block-rows of $F_{\mathbf{q}}(\lambda) \Pi_{\mathbf{r}}^n$ relative to $(\boldsymbol\ell,\boldsymbol{\mathrm{id}}, p, q)$ are located in exactly the same positions (see also Remark \ref{rem:permutationally_1}).
Thus, to prove the second claim in parts (a) and (b), we just need to identify the wing block-rows and the wing block-columns of $F_{\mathbf{q}}(\lambda)$ relative to  $(\boldsymbol\ell, \mathbf{r}, p, q)$.
By Lemma \ref{lem: column locations}, the block-columns  (resp. block-rows) of $M^P_{\mathbf{q}}$ of the form $e_i \otimes I_n$ (resp. $e_i^T \otimes I_n$), for some $1\leq i \leq k$, are precisely those  not in positions $k - \heads(\mathbf{q})$ (resp. not in positions $k -\heads(\rev(\mathbf{q})$) or, equivalently, by  part (ii) in Lemma \ref{lem:prop-cons}, those  in positions $k- \mathfrak{C}_{q}$ (resp. $ k- \mathfrak{I}_q$). 
Taking into account  part (c) in Lemma \ref{lem:prop-cons}, the set of positions of the block-columns of the form $e_i \otimes I_n$ and the set of positions of the block-rows  of the form $e_i^T \otimes I_n$ are disjoint and their union is $\{2:k\}$.  
Since  $M^P_{-k} = A_k \oplus I_{n(k-1)}$,  the blocks of $M^P_{-k}$ in positions $(i, i)$, with $i \in \{2:k\}$, are of the form $\lambda I_n$.  
Thus, the block-columns (resp. the block-rows) of $F_{\mathbf{q}}(\lambda)$ of the form $-e_i \otimes I_n + \lambda e_j \otimes I_n$ (resp. $-e_i^T \otimes I_n + \lambda e_j^T \otimes I_n$) for some $i, j\in \{1:k\}$ are precisely those in $k- \mathfrak{C}_q$  (resp. $k-\mathfrak{I}_q$). 
Hence,  the second claim in parts (a) and (b) follows. 
The equivalent condition for the position of the wing block-rows and  the wing block-columns in terms of tuples satisfying the SIP property follows from Lemma \ref{txSIP} and the definition of reversal tuple.

To prove part (c), notice that, since  $k-1$ is the largest index in $\mathbf{q}$, we have $k-1\in \heads(\mathbf{q}) \cap  \heads(\rev(\mathbf{q}))$.  
Therefore,  by parts (b) and (c), the first block-row and the first block-column of $F_{\mathbf{q}}(\lambda)$  are not, respectively, a wing block-row of $F_{\mathbf{q}}(\lambda) \Pi_{\mathbf{r}}^n$ relative to $(\boldsymbol\ell, \boldsymbol{\mathrm{id}}, p, q)$,  or a wing block-column of $(\Pi_{\boldsymbol\ell}^n)^{\mathcal{B}}F_{\mathbf{q}}(\lambda)$ relative to $(\boldsymbol{\mathrm{id}}, \mathbf{r}, p, q)$.
   Since the positions of the wing block-rows (resp. the wing block-columns) of $F_{\mathbf{q}}(\lambda) \Pi_{\mathbf{r}}^n$ (resp. $(\Pi_{\boldsymbol\ell}^n)^{\mathcal{B}}F_{\mathbf{q}}(\lambda)$) are the same as the positions of the wing block-rows (resp. wing block-columns) of $F_\mathbf{q}(\lambda)$,  the first claim   in part (c)  follows. 
The second claim follows from the fact that $M(\lambda)$ follows the staircase pattern for $\lambda A_k+A_{k-1},A_{k-2},\hdots,A_0$. 
\end{proof}

\begin{remark}\label{remark:permutation1}
We note that part (c) in Theorem \ref{FP-col-row} implies that the block-permutations $\Pi_{\mathbf{r}}^n$ and $\Pi_{\boldsymbol\ell}^n$ in \eqref{Fiedler-dual} are, respectively, of the form $I_n\oplus \Pi_{\mathbf{\widetilde{r}}}^n$ and $I_n\oplus \Pi_{\widetilde{\boldsymbol\ell}}^n$, for some permutations $\mathbf{\widetilde{r}}$ and $\widetilde{\boldsymbol\ell}$ of the set $\{1:k-1\}$.
\end{remark}


\section{The GFP as block-Kronecker pencils}
\label{sec:GF_as_block_Kron_pen}

In this section, we start by proving that the proper GFP are, up to permutations, block Kronecker pencils.
The precise statement of this result is  Theorem \ref{thm:main:GFP}, however, we postpone its proof to subsection \ref{sec:proof_main:GFP}. 
The case of nonproper GFP is considered in subsection \ref{nonGFP}.

Recall from Remark \ref{rem:properGFP} the concept of a simple pair associated with a proper GFP.

\begin{theorem}\label{thm:main:GFP}
Let $P(\lambda)=\sum_{i=0}^kA_i\lambda^i \in\mathbb{F}[\lambda]^{n\times n}$ be a matrix polynomial of grade $k$, and let
\[
K_{\mathbf{q},\mathbf{z}}(\lambda)= \lambda M_{\mathbf{z}}^P-M_{\mathbf{q}}^P
\]
be the GFP associated with $P(\lambda)$ and $(\mathbf{q}, \mathbf{z})$. 
Let $(\widehat{\mathbf{q}},h)$ be the simple pair associated with $K_{\mathbf{q},\mathbf{z}}(\lambda)$.
Then, there exist block-permutation matrices $\Pi_{\boldsymbol\ell}^n$ and $\Pi_{\mathbf{r}}^n$ such that
\begin{equation}\label{eq:main_result_GFP}
(\Pi_{\boldsymbol\ell}^n)^{\mathcal{B}}K_{\mathbf{q},\mathbf{z}}(\lambda) \Pi_{\mathbf{r}}^n = 
\left[\begin{array}{c|c}
M(\lambda) & L_{\mathfrak{h}(\rev(\widehat{\mathbf{q}}))+k-h-2}(\lambda)^T \\ \hline
L_{\mathfrak{h}(\widehat{\mathbf{q}})-1}(\lambda) & 0 
\end{array}\right]
\end{equation}
is a $(\mathfrak{h}(\widehat{\mathbf{q}})-1,n,\mathfrak{h}(\rev(\widehat{\mathbf{q}}))+k-h-2,n)$-block Kronecker pencil, where the body $M(\lambda)$ satisfies the AS condition for $P(\lambda)$.
Moreover, setting $p:=\mathfrak{h}(\widehat{\mathbf{q}})-1$ and $q:=\mathfrak{h}(\rev(\widehat{\mathbf{q}}))+k-h-2$, the following statements hold.
\begin{itemize}
\item[\rm (a)] The wing block-columns of $(\Pi_{\boldsymbol\ell}^n)^\mathcal{B}K_{\mathbf{q},\mathbf{z}}(\lambda)$ relative to $(\boldsymbol{\mathrm{id}},\mathbf{r},p,q)$ are of the form $-e_i\otimes I_n + \lambda e_{i+1}\otimes I_n$, for some $1\leq i\leq q$, and are located in positions $j\in\{1:k\}$ such that $(\mathbf{q},k- j)$ and $(\mathbf{z}+k,j-1)$ satisfy the SIP.
\item[\rm (b)] The wing block-rows of $K_{\mathbf{q},\mathbf{z}}(\lambda)\Pi_{\mathbf{r}}^n$ relative to $(\boldsymbol\ell,\boldsymbol{\mathrm{id}},p,q)$ are of the form $-e_i^T\otimes I_n + \lambda e_{i+1}^T\otimes I_n$, for some $1\leq i\leq p$, and are located in positions  $j\in\{1:k\}$ such that  $(k-j,\mathbf{q})$ and $(j-1,\mathbf{z}+k)$ satisfy the SIP.
\end{itemize}
\end{theorem}

\begin{remark}
We note that Theorem \ref{thm:main:GFP}, together with Theorem  \ref{thm:Lambda-dual-pencil-linearization}, implies that the (extended) block Kronecker pencil $(\Pi_{\boldsymbol\ell}^n)^{\mathcal{B}}K_{\mathbf{q},\mathbf{z}}(\lambda) \Pi_{\mathbf{r}}^n$ in \eqref{eq:main_result_GFP} is a strong linearization of the matrix polynomial $P(\lambda)=\sum_{i=0}^kA_i\lambda^i$.
\end{remark}

We start by showing that Theorem \ref{thm:main:GFP} holds for a subfamily of proper GFP.
\begin{theorem}\label{thm:main:GFP_aux}
Let $P(\lambda)=\sum_{i=0}^k A_i\lambda^i\in\mathbb{F}[\lambda]^{n\times n}$ be a matrix polynomial of grade $k$.
Let $h\in\{1:k-1\}$, let $\mathbf{q}$ be a permutation of the set $\{0:h\}$, and let $K_{\mathbf{q},\mathbf{z}}(\lambda)$ be the following proper generalized Fiedler pencil
\[
K_{\mathbf{q},\mathbf{z}}(\lambda) = \lambda M_{-k:-h-1}^P-M_{\mathbf{q}}^P,
\]
associated with $P(\lambda)$.
Then, Theorem \ref{thm:main:GFP} holds for $K_{\mathbf{q},\mathbf{z}}(\lambda) $.
\end{theorem}

\begin{proof}
Let $\mathbf{z}=(-k:-h-1)$. 
First note that the simple pair associated with $K_{\mathbf{q},\mathbf{z}}(\lambda)$ is $(\mathbf{q}, h)$.
By applying Lemma \ref{lem:splitting} to the pencil $K_{\mathbf{q},\mathbf{z}}(\lambda)$, which is a also a GFPR, we obtain that $K_{\mathbf{q},\mathbf{z}}(\lambda)$ can be partitioned as follows
\[
K_{\mathbf{q},\mathbf{z}}(\lambda)=
\left[\begin{array}{c|c|c}
D_\mathbf{z}(\lambda) & y_\mathbf{z}(\lambda) & 0 \\ \hline
x_\mathbf{z}(\lambda) &  c(\lambda) & x_\mathbf{q}(\lambda) \\ \hline
0 & y_\mathbf{q}(\lambda) & D_\mathbf{q}(\lambda)
\end{array}
\right],
\]
where $D_\mathbf{z}(\lambda)\in\mathbb{F}[\lambda]^{(k-h-1)n\times(k-h-1)n}$, $x_\mathbf{z}(\lambda)\in\mathbb{F}[\lambda]^{n\times (k-h-1)n}$, $y_\mathbf{z}(\lambda)\in\mathbb{F}[\lambda]^{(k-h-1)n\times n}$, $D_\mathbf{q}(\lambda)\in\mathbb{F}[\lambda]^{nh\times nh}$, $x_\mathbf{q}(\lambda)\in\mathbb{F}[\lambda]^{n\times nh}$, $y_\mathbf{q}(\lambda)\in\mathbb{F}[\lambda]^{nh\times n}$ and $c(\lambda)\in\mathbb{F}[\lambda]^{n\times n}$. 
Moreover, Lemma \ref{lem:splitting} also tells us that the pencil
\[
F(\lambda):=
\begin{bmatrix}
 c(\lambda) & x_\mathbf{q}(\lambda) \\
y_\mathbf{q}(\lambda) & D_\mathbf{q}(\lambda)
\end{bmatrix}
\]
is a  Fiedler pencil associated with the matrix polynomial $Q(\lambda):=\lambda^{h+1}A_{h+1}+\lambda^h A_h+\cdots+\lambda A_1+A_0$ since $F(\lambda)= \lambda M_{-h-1}^Q-M_{\mathbf{q}}^Q$.
Similarly, the pencil
\[
G(\lambda):=
\begin{bmatrix}
D_\mathbf{z}(\lambda) & y_\mathbf{z}(\lambda) \\
x_\mathbf{z}(\lambda) &  c(\lambda)
\end{bmatrix}
\]
is a proper GFP associated with the matrix polynomial $Z(\lambda):=\lambda^{k-h}A_k+\lambda^{k-h-1}A_{k-1}+\cdots+\lambda A_{h+1}+A_h$ since  $G(\lambda)=\lambda M_{h-k:-1}^Q-M_0^Q$.
Furthermore, a direct matrix multiplication similar to the computations necessary to prove Lemma \ref{lem: string product} shows that
\[
\begin{bmatrix}
D_\mathbf{z}(\lambda) \\ x_\mathbf{z}(\lambda)
\end{bmatrix} = L_{k-h-1}(\lambda)^T\otimes I_n, \quad \mbox{and} \quad 
\begin{bmatrix}
y_\mathbf{z}(\lambda) \\  c(\lambda)
\end{bmatrix} = 
\begin{bmatrix}
\lambda A_k \\ \lambda A_{k-1} \\ \vdots \\ \lambda A_{h+2} \\ \lambda A_{h+1}+A_h
\end{bmatrix}.
\]

Next, applying Theorem \ref{FP-col-row} to the pencil $F(\lambda)$ and to according Remark \ref{remark:permutation1}, we deduce that $c(\lambda)= \lambda A_{h+1}+A_h$ and obtain that there exist block-permutation matrices $\Pi_{\mathbf{r}_1}=I_n\oplus \Pi_{\widetilde{\mathbf{r}}}$ and $\Pi_{{\boldsymbol\ell_1}}=I_n\oplus\Pi_{\widetilde{\boldsymbol\ell}}$  such that 
\begin{align*}
\Pi_{\boldsymbol\ell_1}^{\mathcal{B}} F(\lambda) \Pi_{\mathbf{r}_1}=&
\begin{bmatrix}
I_n & 0 \\ 0 & \Pi_{\widetilde{\boldsymbol\ell}}^\mathcal{B}
\end{bmatrix}
\begin{bmatrix}
\lambda A_{h+1}+A_h & x_\mathbf{q}(\lambda) \\ y_\mathbf{q}(\lambda) & D_\mathbf{q}(\lambda) 
\end{bmatrix}
\begin{bmatrix}
I_n & 0 \\ 0 & \Pi_{\widetilde{\mathbf{r}}}
\end{bmatrix}  = : \\
&\left[\begin{array}{c|c}
 M(\lambda) & L_{q_1}(\lambda)^T\otimes I_n\\
 \hline  L_p(\lambda)\otimes I_n & 0 
 \end{array}\right] =: \\ 
&
\left[\begin{array}{cc|c} 
\lambda A_{h+1}+A_{h} & m_1(\lambda) & \ell_{q_1}(\lambda)^T\otimes I_n \\
m_2(\lambda) & \widetilde{M}(\lambda) & \widetilde{L}_{q_1}(\lambda)^T\otimes I_n \\ \hline
\ell_p(\lambda)\otimes I_n & \widetilde{L}_p(\lambda)\otimes I_n & \phantom{\Big{(}} 0 \phantom{\Big{(}}
\end{array}\right]
\end{align*}
 is a $(p,n,q_1,n)$-block Kronecker pencil for $Q(\lambda)$, with $p=\mathfrak{h}(\mathbf{q})-1$ and $q_1:=\mathfrak{h}(\rev(\mathbf{q}))-1$.
Then, notice that
 \begin{align}\label{eq:main_GFP_aux_eq1}
 \begin{split}
 \begin{bmatrix}
 I_{(k-h-1)n} \\ & I_n \\ & &  \Pi_{\widetilde{\boldsymbol\ell}}^\mathcal{B}
\end{bmatrix}
\begin{bmatrix}
D_\mathbf{z}(\lambda) & y_\mathbf{z}(\lambda) & 0 \\ 
x_\mathbf{z}(\lambda) & \lambda A_{h+1}+A_h & x_\mathbf{q}(\lambda) \\ 
0 & y_\mathbf{q}(\lambda) & D_\mathbf{q}(\lambda)
\end{bmatrix}
\begin{bmatrix}
I_{(k-h-1)n} \\ & I_n \\ & & \Pi_{\widetilde{\mathbf{r}}}
\end{bmatrix}& = \\
\begin{bmatrix}
D_\mathbf{z}(\lambda) & y_\mathbf{z}(\lambda) & 0 & 0 \\
x_\mathbf{z}(\lambda) & \lambda A_{h+1}+A_h & m_1(\lambda) & \ell_{q_1}(\lambda)^T\otimes I_n \\
0 & m_2(\lambda) & \widetilde{M}(\lambda) & \widetilde{L}_{q_1}(\lambda)^T\otimes I_n \\
0 & \ell_p(\lambda)\otimes I_n & \widetilde{L}_p(\lambda)\otimes I_n & 0
\end{bmatrix}&,
\end{split}
 \end{align}
which is block-permutationally equivalent to the matrix pencil
\begin{equation}\label{eq:main_GFP_aux_eq2}
\left[\begin{array}{cc|cc}
y_\mathbf{z}(\lambda) & 0 & D_\mathbf{z}(\lambda) & 0 \\
\lambda A_{h+1}+A_h & m_1(\lambda) & x_\mathbf{z}(\lambda) & \ell_{q_1}(\lambda)^T\otimes I_n \\
m_2(\lambda) & \widetilde{M}(\lambda) & 0 & \widetilde{L}_{q_1}(\lambda)^T\otimes I_n \\ \hline
\ell_p(\lambda)\otimes I_n & \widetilde{L}_p(\lambda)\otimes I_n & 0 & \phantom{\Big{(}} 0 \phantom{\Big{(}}
\end{array}\right].
\end{equation}
Thus, there exist matrix permutations $\Pi_{\boldsymbol\ell}^n$ and $\Pi_{\mathbf{r}}^n$ such that $(\Pi_{\boldsymbol\ell}^n)^{\mathcal{B}}K_{\mathbf{q},\mathbf{z}}(\lambda) \Pi_{\mathbf{r}}^n $ is a $(p, n, q_1+k-h-1)$-block-Kronecker pencil  with body
 \[
 H(\lambda)=\begin{bmatrix}
 y_{\mathbf{z}}(\lambda) & 0 \\
 \lambda A_{h+1}+A_h & m_1(\lambda) \\
 m_2(\lambda) & \widetilde{M}(\lambda)
 \end{bmatrix}.
 \]
 To prove that the pencil $H(\lambda)$ satisfies the AS condition for $P(\lambda)$, we decompose $H(\lambda)$ as
 \[
\underbrace{\begin{bmatrix}
y_\mathbf{z}(\lambda) & 0 \\ 0 & 0 \\ 0 & 0
\end{bmatrix}}_{H_1(\lambda)}+\underbrace{\begin{bmatrix}
0 & 0 \\ \lambda A_{h+1}+A_h & m_1(\lambda) \\
m_2(\lambda) & \widetilde{M}(\lambda)
\end{bmatrix}}_{H_2(\lambda)},
 \]
 and notice that
 \[
\mathrm{AS}(H,s)=\mathrm{AS}(H_1,s)+\mathrm{AS}(H_2,s) = A_s,
\]
for $s=0:k$, which follows from the structure of  the  block-vector $  y_\mathbf{z}(\lambda)$, the fact that the body of $\Pi_{\boldsymbol\ell_1}^{\mathcal{B}} F(\lambda) \Pi_{\mathbf{r}_1}$  satisfies the AS condition for $Q(\lambda)$, and the linearity of the antidiagonal sum.

To prove part (a) of Theorem \ref{thm:main:GFP}, notice first that Theorem \ref{FP-col-row} implies that the wing block-columns of the Fiedler pencil $F(\lambda)$  relative to $(\boldsymbol\ell_1, \mathbf{r}_1,  p, q_1)$ are in positions $h+1-j$, where $j\in\{0:h-1\}$ and $(\mathbf{q},j)$ satisfies the SIP. 
Then, notice that, for $q:=q_1+k-h-2$,   the wing block-columns of $(\Pi_{\boldsymbol\ell}^n)^{\mathcal{B}}K_{\mathbf{q},\mathbf{z}}(\lambda)$  relative to $(\boldsymbol{\rm id}, \mathbf{r}, p, q)$ are in positions $\{1:k-h-1\}\cup\{(k-h-1)+h+1-i{\red :} \; i\in \{0:h-1\}\mbox{ and } (\mathbf{q}, i) \; \textrm{satisfies  the }SIP\}$, or equivalently, in positions
$$\{1:k-h-1\}\cup\{j: \; j\in \{k-h-1:k\}\mbox{ and } (\mathbf{q}, k-j) \; \textrm{satisfies  the SIP}\}.$$
Note also that, in this case, $\mathbf{z}+k=0:k-h-1$ and $(\mathbf{z}+k, s)$ satisfies the SIP if and only if $s\in \{0:k-h-2\}$. 
 Moreover, we observe that all the wing  block-columns are of the form $-e_i \otimes I_n + \lambda e_{i+1} \otimes I_n$ for some $1 \leq i \leq q$,  which implies part (a).   
 
Finally, to prove part (b), recall from Theorem \ref{FP-col-row} that the wing block-rows of $F(\lambda)$ relative to $(\boldsymbol\ell_1, \mathbf{r}_1, p, q_1)$ are located in positions $h+1-j$, where $j\in\{0:h-1\}$ and $(j,\mathbf{q})$ satisfies the SIP.
Then, notice that, no index $s$ is such that $(s, \mathbf{z}+k)$ satisfies the SIP and  the wing block-rows of $K_{\mathbf{q},\mathbf{z}}(\lambda)\Pi_{\mathbf{r}}^n$  relative to $(\boldsymbol\ell, \boldsymbol{\rm id}, p, q)$ are located in positions $$\{(k-h-1)+h+1-j: \; j\in \{0:h-1\} \mbox{ and } (j, \mathbf{q}) \; \textrm{satisfies  the SIP}\}$$ and have the desired form, which implies part (b). 
\end{proof}


\subsection{Proof of Theorem \ref{thm:main:GFP}}\label{sec:proof_main:GFP}

Armed with Lemma \ref{lem:left_multiplication2} and Theorem  \ref{thm:main:GFP_aux}, we are in a position to prove Theorem \ref{thm:main:GFP}.

\medskip

\begin{proof}{\rm (of  Theorem \ref{thm:main:GFP})}
 Recall from Remark \ref{rem:properGFP} that the GFP $K_{\mathbf{q},\mathbf{z}}(\lambda)$ can be written in the form
\[
K_{\mathbf{q},\mathbf{z}}(\lambda) = M_\mathbf{m}^P(\lambda M_{-k:-h-1}^P-M_{\widehat{\mathbf{q}}}^P),
\]
where $\widehat{\mathbf{q}}=(-\rev(\mathbf{m}),\mathbf{q})$ is a permutation of $\{0:h\}$, and $\mathbf{m}$ is a tuple with indices from $\{-1:-h\}$.
Additionally, let us introduce the notation $\mathbf{m}=(-i_s,-i_{s-1},\hdots,-i_1)$ for the indices of the tuple $\mathbf{m}$, and let us write
\begin{equation}
 \label{eq:GFP_in_proof}
K_{\mathbf{q},\mathbf{z}}(\lambda) =
M_{(-i_s,-i_{s-1},\hdots,-i_1)}^P(\lambda M_{-k:-h-1}^P-M_{(i_1,\hdots,i_{s-1},i_s,\mathbf{q})}^P).
\end{equation}
The proof proceeds by induction on the number $s$ of factors in  $M_{(-i_s,-i_{s-1},\hdots,-i_1)}^P$.
When $s=0$, we have the pencil $\lambda M_{-k:-h-1}^P-M_{\widehat{\mathbf{q}}}^P$.
In this case, the result follows from Theorem \ref{thm:main:GFP_aux}. 
Assume, now, that the result is true for the proper GFP \begin{align*}
L(\lambda):=&M_{(-i_{s-1},\hdots,-i_1)}^P(\lambda M^P_{-k:-h-1}-M^P_{(i_1,\hdots,i_{s-1},i_s,\mathbf{q})})=\\ &\lambda M_{(-i_{s-1},\hdots,-i_1,-k:-h-1)}^P-M_{(i_s,\mathbf{q})}^P,
\end{align*}
 and let us show that the result is true for the pencil $K_{\mathbf{q},\mathbf{z}}(\lambda)=M_{-i_s}^PL(\lambda)$.
Since $(\widehat{\mathbf{q}},h)$ is also a simple pair associated with $L(\lambda)$, the inductive hypothesis implies that there exist block-permutation matrices $\Pi_{\widetilde{\boldsymbol\ell}}^n$ and $\Pi_\mathbf{r}^n$ such that
\begin{equation}
(\Pi_{\widetilde{\boldsymbol\ell}}^n)^{\mathcal{B}}L(\lambda) \Pi_{\mathbf{r}}^n = 
\left[\begin{array}{c|c}
M^\prime(\lambda) & L_{\mathfrak{h}(\rev(\widehat{\mathbf{q}}))+k-h-2}(\lambda)^T\otimes I_n \\ \hline
L_{\mathfrak{h}(\widehat{\mathbf{q}})-1}(\lambda)\otimes I_n & 0 
\end{array}\right],
\end{equation}
is an extended $(p, n, q, n)$-block Kronecker pencil, whose body $M^\prime(\lambda)$ satisfies the AS condition for $P(\lambda)$, where $p=\mathfrak{h}(\widehat{q})-1$ and $q=\mathfrak{h}(\rev(\widehat{\mathbf{q}}))+k-h-2$.
To prove that the result holds as well for $K_{\mathbf{q},\mathbf{z}}(\lambda)$, we will apply Lemma \ref{lem:left_multiplication2}, so we need to start by showing that the  $(k-i_s)$th and  $(k-i_s+1)$th block-rows of $L(\lambda)$ are, respectively, a body block-row and a wing block-row of $L(\lambda)$ relative to $(\boldsymbol{\widetilde{\ell}},\mathbf{r},p,q)$.
First, notice that $(i_s, (i_s,\mathbf{q}))$ does not satisfy the SIP.
Thus, part (b) of the inductive hypothesis implies that the  $(k-i_s)$th block-row of $L(\lambda)$ is a body block-row relative to $(\boldsymbol{\widetilde{\ell}},\mathbf{r},p,q)$.
Next, notice that each index in $\{0:k\}$ appears either in the tuple $(i_{s-1},\hdots,i_1,k:h+1)$ or in the tuple   $(i_s,\mathbf{q})$, and it appears in those tuples at most one time. 
This, in turn, implies that both $(i_s-1,(i_s,\mathbf{q}))$ and $(k-i_s,(k-i_{s-1},\hdots,k-i_1,0:k-h-1))$  satisfy the SIP.
By part (b) of the inductive hypothesis, it follows, then, that the $(k-i_s+1)$th block-row of $L(\lambda)$ is a wing block-row relative to $(\boldsymbol{\widetilde{\ell}},\mathbf{r},p,q)$.
Therefore, by part (i) of Lemma \ref{lem:left_multiplication2}, there exist block-permutation matrices $\Pi_{ \boldsymbol\ell}^n$ and $\Pi_\mathbf{r}^n$ such that \eqref{eq:main_result_GFP} holds with a body $M(\lambda)$ satisfying the AS condition for $P(\lambda)$. 

Now, we show that parts (a) and (b) hold for $K_{\mathbf{q},\mathbf{z}}(\lambda)$. 
 To this end, we introduce the notation $\mathbf{t}:=(k-i_{s-1},\hdots,k-i_1,0:k-h-1)$.
Notice that $\mathbf{z}+k=(\mathbf{m}, -k:-h-1)+k=(k-i_s,\mathbf{t})$.
We start by proving part (a).

Let us denote by $\{ c_1,\hdots,c_q \}$ the set of  positions of the wing block-columns of $L(\lambda)$ relative to $(\boldsymbol{\widetilde{\ell}},\mathbf{r},p,q)$.
From parts (iiia) and (iiib) of  Lemma \ref{lem:left_multiplication2}, we obtain that $\{ c_1,\hdots,c_q \}$ is also the set of positions of the wing block-columns of $K_{\mathbf{q},\mathbf{z}}(\lambda)$ relative to  $( \boldsymbol\ell,\mathbf{r},p,q)$.
Since $L(\lambda)$ and $K_{\mathbf{q},\mathbf{z}}(\lambda)$ have the same  number of wing block-columns  and are located in the same positions,  part (a) for $K_{\mathbf{q},\mathbf{z}}(\lambda)$ follows from part (a) for $L(\lambda)$ (which holds by  the inductive hypothesis) if we prove that  the tuples $(\mathbf{q},k-j)$ and  $((k-i_s,\mathbf{t}),j-1)$ satisfy the SIP  if, respectively, the tuples $((i_s,\mathbf{q}),k-j)$ and $(\mathbf{t},j-1)$ do. 
So, let us assume that $((i_s,\mathbf{q}),k-j)$ and $(\mathbf{t},j-1)$ satisfy the SIP for some $j\in \{1:k\}$. 
The assumptions clearly imply that the tuple $(\mathbf{q},k-j)$ satisfy the SIP, since this is a subtuple of consecutive indices of $((i_s,\mathbf{q}),k-j)$.
Then, notice that, since $(\mathbf{t},j-1)$ satisfies the SIP, since  each different index of $\mathbf{t}$ appears only one time in $\mathbf{t}$, and $k-i_s\notin  \mathbf{t}$, to prove that the tuple $((k-i_s,\mathbf{t}),j-1)$ satisfies the SIP, it is only necessary to check the case $j=k-i_s+1$, that is, we have to prove that $(k-i_s,\mathbf{t},k-i_s)$ satisfies the SIP when  $((i_s,\mathbf{q}),i_s-1)$ and $(\mathbf{t},k-i_s)$ satisfy the SIP.
The proof proceeds by contradiction.
Assume that $((i_s,\mathbf{q}),i_s-1)$ and $(\mathbf{t},k-i_s)$ satisfy the SIP and that $(k-i_s,\mathbf{t},k-i_s)$ does not satisfy the SIP.
The latter assumption implies that $k-i_s+1\notin\mathbf{t}$, which, in turn, implies that $i_s-1\in\mathbf{q}$. 
Since each different index in $(i_s,\mathbf{q})$ appears only once, the statement $i_s-1\in\mathbf{q}$ implies that $((i_s,\mathbf{q}),i_{s}-1)$ does not satisfy the SIP, which  contradicts our assumptions. 
Thus, (a) follows.

Finally, we prove that part (b) is true. 
Let us denote by $\{r_1,\hdots,r_p\}$ the set of  positions of the wing block-rows of $L(\lambda)$ relative to  $(\boldsymbol{\widetilde{\ell}},\mathbf{r},p,q)$, and recall that $k-i_s\notin \{r_1,\hdots,r_p\}$ and $k-i_s+1\in\{r_1,\hdots,r_p\}$.
From parts (ii.c) and (ii.d) of Lemma \ref{lem:left_multiplication2}, we obtain that the set of  positions of the wing block-rows of $K_{\mathbf{q},\mathbf{z}}(\lambda)$ relative to  $( \boldsymbol\ell,\mathbf{r},p,q)$ is $(\{r_1,\hdots,r_p\} \setminus \{k-i_s+1\})\cup\{k-i_s\}$.
Since $L(\lambda)$ and $K_{\mathbf{q},\mathbf{z}}(\lambda)$ have the same number of wing block-rows, part (b) for $K_{\mathbf{q},\mathbf{z}}(\lambda)$  follows from part (b) for $L(\lambda)$ (by the inductive hypothesis) if  the following three statements hold: (i) since $k-i_s$ is the position of a wing  block-row of $K_{\mathbf{q},\mathbf{z}}(\lambda)$ relative to $( \boldsymbol\ell,\mathbf{r},p,q)$, the tuples $(i_s,\mathbf{q})$ and $(k-i_s-1,(k-i_s,\mathbf{t}))$  satisfy the SIP; (ii) since $k-i_s+1$ is not a wing block-row of $K_{\mathbf{q},\mathbf{z}}(\lambda)$ relative to $( \boldsymbol\ell,\mathbf{r},p,q)$, one of the following tuples  $(i_s-1,\mathbf{q})$ and $(k-i_s,(k-i_s,\mathbf{t}))$, or both, does not satisfy the SIP; (iii) when $j\neq k-i_s,k-i_{s}+1$, if the tuples $(k-j,(i_s,\mathbf{q}))$ and $(j-1,\mathbf{t})$ satisfy the SIP, then $(k-j,\mathbf{q})$ and $(j-1,(k-i_s,\mathbf{t}))$ also  satisfy  the SIP.
Since (i) and (ii) are immediate to prove,  we  focus on proving (iii).
Assume that $(k-j,(i_s,\mathbf{q}))$ and $(j-1,\mathbf{t})$ satisfy the SIP, and let $j\neq k-i_s,k-i_s+1$.
First, if $j\neq k-i_s-1$, then the tuple $(i_s,k-j,\mathbf{q})\sim(k-j,i_s,\mathbf{q})$ satisfies the SIP, which, in turn, implies that its subtuple $(k-j,\mathbf{q})$ satisfies the SIP as well.
Second, if $j=k-i_s-1$, the assumption that $(i_s+1,(i_s,\mathbf{q}))$ satisfies the SIP, together with the fact that each different index of $(i_s,\mathbf{q})$ appears only one time, implies that $(i_{s}+1,\mathbf{q})$ satisfies the SIP as well. 
Finally, since $k-i_s\notin\mathbf{t}$ and $j-1\neq k-i_s,k-i_s-1$, we immediately obtain that $(j-1,(k-i_s,\mathbf{t}))$ satisfies the SIP.
\end{proof}

\subsection{GFP that are not proper}\label{nonGFP}
Recall that a GFP as in Definition \ref{def:GFP} is not proper when $0\in C_1$ and/or $k\in C_0$.
Unlike proper GFP,  nonproper GFP associated with a matrix polynomial $P(\lambda)=\sum_{i=0}^k A_i \lambda^i$ are strong linearizations of $P(\lambda)$ only if $A_0$ and/or $A_k$ is nonsingular.
This implies that nonproper GFP are never strong linearizations of singular matrix polynomials.
This drawback makes them the least interesting subfamily of GFP. 
However, nonproper GFP find applications in the problem of constructing symmetric linearizations of symmetric matrix polynomials of even grade \cite{AV04}.
These linearizations can be constructed when the trailing and/or the leading coefficient of the matrix polynomial is  nonsingular. 
For example, consider a symmetric matrix polynomial $P(\lambda)=\sum_{i=0}^4A_i\lambda^i\in \mathbb{F}[\lambda]^{n\times n}$.
If $A_0$ is nonsingular, then the nonproper GFP
\[
\lambda M^P_{(-4,-2-0)}-M^P_{(3,1)}=
\begin{bmatrix}
\lambda A_4+A_3 & -I_n & 0 & 0 \\
-I_n & 0 & \lambda I_n & 0 \\
0 & \lambda I_n & \lambda A_2+A_1 & -I_n \\
0 & 0 & -I_n & -\lambda A_0^{-1}
\end{bmatrix}
\]
is a symmetric strong linearization of $P(\lambda)$. 
On the other hand, if $A_4$ is nonsingular, then the nonproper GFP
\[
\lambda M^P_{(-3,-1)}-M^P_{(4,2,0)}=
\begin{bmatrix}
-A_4^{-1} & \lambda I_n & 0 & 0 \\
\lambda I_n & \lambda A_3+A_2 & -I_n & 0 \\
0 & -I_n & 0 & \lambda I_n \\
0 & 0 & \lambda I_n & \lambda A_1+A_0
\end{bmatrix}
\]
is a symmetric strong linearization of $P(\lambda)$.
Notice that neither of the above nonproper GFP  can be permuted into an extended block Kronecker pencil with a body satisfying the AS condition for $P(\lambda)$.
Therefore, we cannot state a result like Theorem \ref{thm:main:GFP} for GFP that are not proper. 
Nonetheless, we state a weaker result in Theorem \ref{thm:nonproper}. 
Before we state this theorem we give an example.
\begin{example}
Let $P(\lambda)=\sum_{i=0}^4 A_i \lambda^i$,  where $A_0$ and $A_4$ are nonsingular matrices. 
Consider the nonproper GFP given by 
$$ L_1(\lambda)= \lambda M^P_{-0, -2, -4} - M^P_{3, 1}.$$ 
Note that, by the commutativity relations \eqref{commutativity}, we get
$$L_1(\lambda) =( \lambda M^P_{-2, -4} - M^P_{3, 1, 0}) M^P_{-0},$$
where $\lambda M^P_{-2, -4} - M^P_{3, 1,0}$ is a proper GFP. 
Consider now the nonproper GFP
$$L_2(\lambda)= \lambda M^P_{-0, -1} - M^P_{4, 3, 2}.$$
Again, using the commutativity relations, we have
$$L_2(\lambda)= M^P_{-0} M^P_{4} (\lambda M^P_{-4, -1} - M^P_{0, 3, 2}),$$
where $\lambda M^P_{-4, -1} - M^P_{0, 3, 2}$  is also a proper GFP.
 Since the matrices $M_{0}^P$ and $M_{-4}^P M_{0}^P$ are block-diagonal, applying Theorem \ref{thm:main:GFP} to the proper GFP $L_1(\lambda)M_{0}^P$ and $M_{-4}^P M_{0}^PL_2(\lambda)$, we conclude that the non proper GFP considered in this example, are, up to block-permutations and product of nonsingular block-diagonal matrices, extended block-Kronecker pencils.
\end{example}

It is easy to see that, in general, given a nonproper GFP $K_{\mathbf{q},\mathbf{z}}(\lambda)$ associated with $P(\lambda)=\sum_{i=0}^kA_i\lambda^i \in\mathbb{F}[\lambda]^{n\times n}$,  due to the commutativity relations \eqref{commutativity} of the elementary matrices, there exist block-diagonal matrices $L$ and $R$ such that 
\[
K_{\widetilde{\mathbf{q}},\widetilde{\mathbf{z}}}(\lambda):= L K_{\mathbf{q},\mathbf{z}}(\lambda) R
\]
is a proper GFP associated with $P(\lambda)$ and some permutations $\widetilde{\mathbf{q}}$ and $\widetilde{\mathbf{z}}$, where   $R:=\diag(R_1,I_n,\hdots,I_n,R_2)$ and $L:=\diag(L_1,I_n,\hdots,I_n,L_2)$ are  block-diagonal $kn\times kn$ matrices with 
\begin{equation}\label{R1}
R_2:=
\left\{\begin{array}{lr}
A_0, & \mbox{if }-0,-1\in\mathbf{z}\,\,\mbox{ and $-0$ is to the right of $-1$ in $\mathbf{z}$,} \\
A_0, & \mbox{if }-0\in\mathbf{z}\mbox{ and }-1\notin\mathbf{z},\\
I_n, & \mbox{otherwise},
\end{array}\right.
\end{equation}
\begin{equation}\label{L1}
L_2:=
\left\{\begin{array}{lr}
A_0, & \mbox{if }-0,-1\in\mathbf{z}\,\,\mbox{ and $-0$ is to the left of $-1$ in $\mathbf{z}$,} \\
I_n, & \mbox{otherwise},
\end{array}\right.
\end{equation}
\begin{equation}\label{R2}
R_1:=
\left\{\begin{array}{lr}
A_k, & \mbox{if }k,k-1\in\mathbf{q}\,\,\mbox{ and $k$ is to the right of $k-1$ in $\mathbf{z}$,} \\
A_k, & \mbox{if }k\in\mathbf{q}\mbox{ and }k-1\notin\mathbf{q},\\
I_n, & \mbox{otherwise},
\end{array}\right.
\end{equation}
and
\begin{equation}\label{L2}
L_1:=
\left\{\begin{array}{lr}
A_k, & \mbox{if }k,k-1\in\mathbf{q}\,\,\mbox{ and $k$ is to the left of $k-1$ in $\mathbf{z}$,} \\
I_n, & \mbox{otherwise}.
\end{array}\right.
\end{equation}

We call the pencil $K_{\widetilde{\mathbf{q}},\widetilde{\mathbf{z}}}(\lambda)$ above a \emph{proper GFP associated with $K_{\mathbf{q},\mathbf{z}}(\lambda)$}.
Applying Theorem \ref{thm:main:GFP} to a  proper GFP associated with $K_{\mathbf{q},\mathbf{z}}(\lambda)$, we obtain Theorem \ref{thm:nonproper}.
\begin{theorem}\label{thm:nonproper}
Let $P(\lambda)=\sum_{i=0}^kA_i\lambda^i \in\mathbb{F}[\lambda]^{n\times n}$ be a matrix polynomial of grade $k$, and let
\[
K_{\mathbf{q},\mathbf{z}}(\lambda)= \lambda M_{\mathbf{z}}^P-M_{\mathbf{q}}^P,
\]
be a nonproper GFP associated with $P(\lambda)$. 
Let $R:=\diag(R_1,I_n,\hdots,I_n,R_2)$ and $L:=\diag(L_1,I_n,\hdots,I_n,L_2)$ be $kn\times kn$ block-diagonal matrices as defined in \eqref{R1}
--\eqref{L2}.
Then, there exist block-permutation matrices $\Pi_{\boldsymbol\ell}^n$ and $\Pi_{\mathbf{r}}^n$ such that
\begin{equation}
(\Pi_{\boldsymbol\ell}^n)^{\mathcal{B}}LK_{\mathbf{q},\mathbf{z}}(\lambda)R \Pi_{\mathbf{r}}^n = 
\left[\begin{array}{c|c}
M(\lambda) & K_2(\lambda)^T \\ \hline
K_1(\lambda) & 0 
\end{array}\right]
\end{equation}
is a block Kronecker pencil whose body $M(\lambda)$ satisfies the AS condition for $P(\lambda)$.
\end{theorem} 

Thus, Theorem \ref{thm:nonproper} says that GFP  that are not proper are block Kronecker pencils, up to permutation and product by  nonsingular  block-diagonal matrices.

\section{The GFPR as extended block Kronecker pencils}\label{sec:GFPR_as_minimal_bases_pencils}

We prove in this section that all GFPR associated with a matrix polynomial $P(\lambda)$ are, up to permutations of block-rows and block-columns, extended block Kroneckers pencils with bodies satisfying the AS condition for $P(\lambda)$. 
This result is stated in Theorem \ref{thm:main_GFPR}, which is one of the main results of this paper.
However, we postpone its proof to Section \ref{sec:thm:main_GFPR}.

\begin{theorem}\label{thm:main_GFPR}
Let $P(\lambda)=\sum_{i=0}^k A_i\lambda^i \in\mathbb{F}[\lambda]^{n\times n}$ be a matrix polynomial of grade $k$, and let 
\begin{equation}\label{LP-gen}
L_P(\lambda)=M_{{\boldsymbol\ell}_{q},{\boldsymbol\ell}_{z}}(\mathcal{X},\mathcal{ Z})(\lambda M^P_{\mathbf{z}}
-M^P_{\mathbf{q}})M_{\mathbf{r}_{z},\mathbf{r}_{q}}(\mathcal{W}, \mathcal{Y})
\end{equation}
be a GFPR associated with $P(\lambda)$.
  Then, there exist block-permutation matrices $\Pi_{\boldsymbol\ell}^n,\Pi_{\mathbf{r}}^n$ such that 
 \[
 (\Pi_{\boldsymbol\ell}^n)^\mathcal{B} L_P(\lambda) \Pi_{\mathbf{r}}^n =
\left[\begin{array}{c|c}
M(\lambda) & K_2(\lambda)^T \\ \hline
K_1(\lambda) & 0
\end{array}\right]
 \]
 is an extended $(\mathfrak{h}(\mathbf{q})+\mathfrak{h}(k+\mathbf{z})-2,n,\mathfrak{h}(\rev(\mathbf{q}))+\mathfrak{h}(\rev(k+\mathbf{z}))-2,n)$-block Kronecker pencil, whose body $M(\lambda)$ satisfies the AS condition for $P(\lambda)$.
 Moreover, if $K_1(\lambda)$ and $K_2(\lambda)$ are minimal bases, then $ (\Pi_{\boldsymbol\ell}^n)^\mathcal{B} L_P(\lambda) \Pi_{\mathbf{r}}^n$  is a strong linearization of $P(\lambda)$.
\end{theorem}

\begin{remark}
We note that the fact that  $ (\Pi_{\boldsymbol\ell}^n)^\mathcal{B} L_P(\lambda) \Pi_{\mathbf{r}}^n$ is a strong linearization of $P(\lambda)$ if $K_1(\lambda)$ and $K_2(\lambda)$ are minimal bases follows from Theorem   \ref{thm:Lambda-dual-pencil-linearization}. 
Thus, in this section, we focus on proving the first claim in Theorem \ref{thm:main_GFPR}.
\end{remark}

The main idea  behind the proof of  Theorem \ref{thm:main_GFPR} is that right  multiplications by elementary matrices  preserve, in some relevant cases, the property of being  block-permutationally equivalent to an extended block Kronecker pencil whose body satisfies the AS condition for some $P(\lambda)$, as we showed in Lemma \ref{lem:right_multiplication}. With the help of this result, next we present and prove Theorems \ref{right-canonical-1} and \ref{left-canonical}, which will be key to prove Theorem \ref{thm:main_GFPR}.
Theorems \ref{right-canonical-1} and \ref{left-canonical} can be seen as particular instances of the general result in Theorem \ref{thm:main_GFPR} together with some structural information concerning  block-rows and block-columns of the particular GFPR they focus on.

\begin{theorem}\label{right-canonical-1}
Let $P(\lambda)\in\mathbb{F}[\lambda]^{n\times n}$ be a matrix polynomial  of grade $k\geq 2$.
Let $\mathbf{q}$ be a permutation of $\{0:k-1\}$ and let 
\[
L_P(\lambda)= M_{\boldsymbol\ell_q}(\mathcal{X})(\lambda M^P_{-k} - M^P_{\mathbf{q}}) M_{\mathbf{r}_q}(\mathcal{Y}),
\]
be a GFPR associated with $P(\lambda)$. 
Then, there exist block-permutation matrices   $\Pi_{\boldsymbol\ell}^n$ and $\Pi_{\mathbf{r}}^n$ such that 
\begin{equation}\label{r-case}
(\Pi_{\boldsymbol\ell}^n)^{\mathcal{B}} L_P(\lambda)\Pi_{\mathbf{r}}^n= \left [ \begin{array}{c|c} M(\lambda) &  K_2(\lambda)^T \\ \hline K_1(\lambda) & 0 \end{array} \right],
\end{equation}
is an extended $(\mathfrak{h}(\mathbf{q})-1,n,\mathfrak{h}(\rev(\mathbf{q}))-1,n)$-block Kronecker pencil, whose body $M(\lambda)$ satisfies the AS condition for $P(\lambda)$.
Moreover,  the following statements hold.
\begin{enumerate}[(1)]
\item[\rm(a)] The first block-row and the first block-column of $L_P(\lambda)$ are, respectively,  the first body block-row and the first body block-column of $L_P(\lambda)$ relative to $(\boldsymbol\ell,\mathbf{r}, \mathfrak{h}(\mathbf{q})-1,\mathfrak{h}(\rev(\mathbf{q}) )-1)$. 
Moreover, the block-entry of $M(\lambda)$ in position $(1,1)$ equals $\lambda A_k + A_{k-1}$.

\item[\rm (b)] The wing block-columns  of $(\Pi_{\boldsymbol\ell}^n)^{\mathcal{B}} L_P(\lambda)$ relative to $(\boldsymbol{\mathrm{id}}, \mathbf{r},  \mathfrak{h}(\mathbf{q})-1,\mathfrak{h}(\rev(\mathbf{q}) )-1)$ of the form $-e_i \otimes I_n + \lambda e_{i+1} \otimes I_n$, for $1\leq  i \leq  \mathfrak{h}(\rev(\mathbf{q}))-1$, are precisely those located in positions $k-j$, where $j \in \{0:k-2\} $ and $j \notin \heads(\boldsymbol\ell_q, \mathbf{q}, \mathbf{r}_q)$, or, equivalently, $j \in \{0:k-2\}$ and $(\boldsymbol\ell_q, \mathbf{q}, \mathbf{r}_q, j)$ satisfies the SIP.

\item[\rm (c)] The wing block-rows of $L_P(\lambda)\Pi_{\mathbf{r}}^n$ relative to $(\boldsymbol\ell, \boldsymbol{\mathrm{id}},  \mathfrak{h}(\mathbf{q})-1,\mathfrak{h}(\rev(\mathbf{q}) )-1)$  of the form $-e_i^T \otimes I_n + \lambda e_{i+1}^T \otimes I_n$, for $1\leq i \leq  \mathfrak{h}(\mathbf{q})-1$, are precisely those located in positions $k-j$, where $j \in \{0:k-2\} $ and $j \notin \heads(\rev(\mathbf r_q), \rev(\mathbf{q}), \rev(\boldsymbol\ell_q))$, or, equivalently, $j \in \{0:k-2\}$ and $(\rev(\mathbf r_q), \rev(\mathbf{q}), \rev(\boldsymbol\ell_q),j)$ satisfies the SIP.


\end{enumerate}
\end{theorem}
\begin{proof}
Let $p:=\mathfrak{h}(\mathbf{q})-1$ and $q:= \mathfrak{h}(\rev(\mathbf{q}))-1$. 
We prove the result by induction on the number of indices in $\boldsymbol\ell_q$ and in $\mathbf{r}_q$. 
If both $\boldsymbol\ell_q$ and $\mathbf{r}_q$ are empty, then $L_P(\lambda)$ is a Fiedler pencil and the result follows by Theorem \ref{FP-col-row}.

Assume that $\boldsymbol\ell_q$ is such that the result holds  for $\boldsymbol\ell_q$ and for tuples $\mathbf{r}_q'$ with at most $t$ indices, with $t \geq 0$. 
The case of Fiedler pencils above demonstrates that some such $\boldsymbol\ell_q$ and $t$ exist, namely, $\boldsymbol\ell_q = \varnothing$ and $t=0$. 
Now, suppose that $\mathbf{r}_q=(\mathbf{r}_q^\prime, x)$ has $t+1$ indices, where $x\in \{0: k-2\}$.
Let $\mathcal{Y}=(\mathcal{Y}^\prime, Y_0)$ be an $n\times n$ matrix assignment for $\mathbf{r}_q$, where $\mathcal{Y}^\prime$ and $Y_0$ are, respectively,  the matrix assignments for $\mathbf{r}_q^\prime$ and  $x$ induced by $\mathcal{Y}$, and let  $L^\prime(\lambda):= M_{\boldsymbol\ell_q}(\mathcal X)(\lambda M^P_{-k} - M^P_{\mathbf{q}}) M_{\mathbf{r}_q^\prime}(\mathcal{Y}^\prime)$.
Since $\mathbf{r}_q^\prime$ has $t$ indices, by the inductive hypothesis, there exist block-permutation matrices $\Pi^n_{\mathbf{r}^\prime}$ and $\Pi^n_{\boldsymbol\ell}$ such that
\begin{equation}\label{eq:thm_aux1}
(\Pi^n_{\boldsymbol\ell})^{\mathcal{B}} L^\prime(\lambda) \Pi^n_{\mathbf{r}^\prime} =
\left[\begin{array}{c|c}
M^\prime(\lambda) & K_2^\prime(\lambda)^T \\ \hline
K_1(\lambda) & 0
\end{array}\right]
\end{equation} 
is an extended $(p,n,q,n)$-block Kronecker pencil whose body $M^\prime(\lambda)$ satisfies the AS condition  for $P(\lambda)$.
Since the index tuple $(\boldsymbol\ell_q,\mathbf{q},\mathbf{r}^\prime_q,x)$ satisfies the SIP, by (b) applied to $L'(\lambda)$, the $(k-x)$th block-column of $L^\prime(\lambda)$ is one of its wing block-columns relative to $(\boldsymbol\ell,\mathbf{r}^\prime,p,q)$.
Thus, by Lemma \ref{lem:right_multiplication}, there exists a  block-permutation matrix $\Pi_{\mathbf{r}}^n$ such that \eqref{r-case} holds,  with $M(\lambda)$ satisfying the AS condition for $P(\lambda)$.

Now, we prove part (a). 
By (iv.b) of Lemma \ref{lem:right_multiplication}, together with part (a) for $L'(\lambda)$ (by the inductive hypothesis), we deduce that the first block-row of $L_P(\lambda)$ is a body block-row relative to $(\boldsymbol\ell,\mathbf{r}, p,q)$.
Since $x\leq k-2$,  by (a) for  $L'(\lambda)$ and by  part (iii.b)  of Lemma \ref{lem:right_multiplication}, the first block-column of $L_P(\lambda)$ is  a body block-column  relative to $(\boldsymbol\ell,\mathbf{r}, p,q)$. 
Moreover, since the right-multiplication by the matrix $M_x(Y_0)$ does not affect the first block entry of $L^\prime(\lambda)$ because $x\leq k-2$, the block-entry of $M(\lambda)$ in position $(1,1)$ equals $\lambda A_k + A_{k-1}$.
Therefore, part (a) is true for $L_P(\lambda)$.

Next, we prove part (b). 
We have to distinguish several cases.

Assume, first, that $x=0$. Then,
\begin{align}\label{U1-U2}
(\Pi^n_{\ell})^{\mathcal{B}}L_P(\lambda)=( \Pi^n_{\ell})^{\mathcal{B}} L ^\prime(\lambda)  M_{ 0}(Y_0) & = \left[\begin{array}{c|c} N_1(\lambda) & u(\lambda) \end{array}\right]  (I_{n(k-1)}\oplus Y_0 )\\
&=\left[ \begin{array}{c|c} N_1(\lambda) & u(\lambda)  Y_0 \end{array} \right],
\end{align}
where $N_1(\lambda)$  consists of the first $k-1$ block-columns  of $(\Pi^n_{\ell})^{\mathcal{B}} L^\prime(\lambda)$ and, hence, $u(\lambda)$ is the $k$th block-column of $( \Pi^n_{\ell})^{\mathcal{B}} L^\prime(\lambda)$.
Part (iii.a) in Lemma \ref{lem:right_multiplication} implies that the wing block-columns of $L_P(\lambda)$ relative to $(\boldsymbol\ell,\mathbf{r},p,q)$ and the wing block-columns of $L^\prime(\lambda)$ relative to $(\boldsymbol\ell,\mathbf{r}^\prime,p,q)$, other than the  $k$th block-column, are equal and are located at the same positions. 
Thus, the wing block-columns of $(\Pi^n_{\boldsymbol\ell})^\mathcal{B}L_P(\lambda)$ relative to $(\boldsymbol{\mathrm{id}},\mathbf{r},p,q)$ and the wing block-columns of $(\Pi^n_{\boldsymbol\ell})^\mathcal{B}L^\prime(\lambda)$ relative to $(\boldsymbol{\mathrm{id}},\mathbf{r}^\prime,p,q)$, other than the $k$th block-column, are equal and are located at the same positions as well.  

Since $0$ is an index of Type II relative to $(\boldsymbol\ell_q,\mathbf{q},\mathbf{r}_q^\prime)$ by Proposition \ref{prop: Type I}, and in view of Remark \ref{rem: Type I and II heads}, we have
\[
\mathrm{heads}(\boldsymbol\ell_q,\mathbf{q},\mathbf{r}_q^\prime, x)=\mathrm{heads}(\boldsymbol\ell_q,\mathbf{q},\mathbf{r}_q^\prime)\cup\{x\}.
\]
This, together with Lemma \ref{txSIP} and part (iii.a) of Lemma \ref{lem:right_multiplication}, implies that, for part (b)  to hold for $L_P(\lambda)$, we only need to show that the $k$th block-column of $( \Pi^n_{\ell})^{\mathcal{B}} L_P(\lambda)$ is not of the form $-e_i \otimes I_n + \lambda e_{i+1} \otimes I_n$. 
Notice  that the induction hypothesis and (b) imply that the $k$th block-column of $(\Pi^n_{\boldsymbol\ell})^\mathcal{B}L^\prime(\lambda)$ is  of the form $-e_{i}\otimes I_n+\lambda e_{i+1}\otimes I_n$, for some $1\leq i\leq q$, which, in turns, implies that the $k$th block-column of $(\Pi^n_{\boldsymbol\ell})^\mathcal{B}L_P(\lambda)$ is not generically  of the form $-e_{i}\otimes I_n+\lambda e_{i+1}\otimes I_n$.

Assume, now, that $x\neq 0$. 
We have
\begin{align}\label{U1-U2}
 ( \Pi^n_{\ell})^{\mathcal{B}} L_P(\lambda)& =(\Pi^n_{\ell})^{\mathcal{B}} L^\prime(\lambda)  M_{x}(Y_0)\\ & = \left[\begin{array}{c|c|c} N_1(\lambda) & U(\lambda) & N_2(\lambda) \end{array}\right]  (I_{n(k-x-1)}\oplus \left[ \begin{array}{cc} Y_0 & I_n \\ I_n & 0 \end{array} \right] \oplus I_{n(x-1)})\\
&=\left[ \begin{array}{c|c|c} N_1(\lambda) & U(\lambda) \cdot \left[ \begin{array}{cc} Y_0 & I_n \\ I_n & 0 \end{array} \right]  & N_2(\lambda) \end{array} \right],
\end{align}
where $N_1(\lambda)$ and $N_2(\lambda)$ consist, respectively, of the block-columns $1:k-x-1$ and $k-x+2 : k$ of $(\Pi^n_{\ell})^{\mathcal{B}} L^\prime(\lambda)$. 
Let us denote by $u_1(\lambda)$ and $u_2(\lambda)$ the first and second block-columns of $U(\lambda)$, which are, respectively, the $(k-x)$th  and the $(k-x+1)$th block-columns of  $(\Pi^n_{\ell})^{\mathcal{B}} L^\prime(\lambda) $. 
Since $({\boldsymbol\ell}_q, \mathbf{q}, \mathbf{r}_q^\prime, x)$ satisfies the SIP, by the inductive hypothesis, the $(k-x)$th block-column of $(\Pi^n_{\ell})^{\mathcal{B}} L^\prime(\lambda)$, that is,  $u_1(\lambda)$,  is a wing block-column of $(\Pi^n_{\ell})^{\mathcal{B}} L^\prime(\lambda)$ relative to $(\mathbf{id},\mathbf{r}^\prime, p, q)$ of the form $-e_i \otimes I_n + \lambda e_{i+1} \otimes I_n$, with $i \leq q$. 
Observe also that the second block-column $u_2^\prime(\lambda) $ of $U(\lambda) \left[ \begin{smallmatrix} Y_0 & I_n \\ I_n & 0 \end{smallmatrix} \right] $, which is the  $(k-x+1)$th block-column of  $( \Pi^n_{\ell})^{\mathcal{B}} L_P(\lambda)$,  equals $u_1(\lambda)$ and, therefore, it is of the form  $-e_i \otimes I_n + \lambda e_{i+1} \otimes I_n$, with $i \leq q$.   
Moreover, the first block-column $u_1^\prime(\lambda)$ of  $U(\lambda) \left[ \begin{smallmatrix} Y_0 & I_n \\ I_n & 0 \end{smallmatrix} \right] $, which is the   $(k-x)$th block-column of $( \Pi^n_{\ell})^{\mathcal{B}} L_P(\lambda)$, equals $u_2(\lambda)+u_1(\lambda) Y_0$. 
Thus, the  $(k-x)$th block-column of $( \Pi^n_{\ell})^{\mathcal{B}} L_P(\lambda)$ is, generically, not of the form $-e_i \otimes I_n + \lambda e_{i+1} \otimes I_n$.

We consider two cases, namely, the index $x$ is a Type I or a Type II index relative to the tuple $(\boldsymbol\ell_q,\mathbf{q},\mathbf{r}_q^\prime)$.

Case I: Assume that $x$ is a Type I index relative to $(\boldsymbol\ell_q,\mathbf{q},\mathbf{r}_q^\prime)$.  
In view of Remark \ref{rem: Type I and II heads}, we have
\[
\mathrm{heads}(\boldsymbol\ell_q,\mathbf{q},\mathbf{r}_q^\prime,x)=(\mathrm{heads}(\boldsymbol\ell_q,\mathbf{q},\mathbf{r}_q^\prime)\cup\{x\})\setminus \{x-1\}.
\]
This, together with Lemma \ref{txSIP} and part (iii.a) of Lemma \ref{lem:right_multiplication} implies that, for (b) to hold for $L_P(\lambda)$ we only need to show that the $(k-x)$th block-column of $( \Pi^n_{\ell})^{\mathcal{B}} L_P(\lambda)$ is not of the form $-e_i \otimes I_n + \lambda e_{i+1} \otimes I_n$ (which we already proved above) and  the $(k-x+1)$th block-column is a wing column of the form $-e_i \otimes I_n + \lambda e_{i+1} \otimes I_n$,  which follows from (ii.b) in Lemma \ref{lem:right_multiplication} and the comments above. 

Case II: Assume that $x$ is a Type II index relative to $(\boldsymbol\ell_q,\mathbf{q},\mathbf{r}_q^\prime)$. 
In view of Remark \ref{rem: Type I and II heads}, we have
\[
\mathrm{heads}(\boldsymbol\ell_q,\mathbf{q},\mathbf{r}_q^\prime,x)=\mathrm{heads}(\boldsymbol\ell_q,\mathbf{q},\mathbf{r}_q^\prime)\cup\{x\}.
\]
This, together with Lemma \ref{txSIP} and part (iii.a) of Lemma \ref{lem:right_multiplication}, implies that, for part (b)  to hold for $L_P(\lambda)$, we only need to show that the $(k-x)$th block-column of $( \Pi^n_{\ell})^{\mathcal{B}} L_P(\lambda)$ is not of the form $-e_i \otimes I_n + \lambda e_{i+1} \otimes I_n$, which we already proved above.

Finally, we prove part (c).
Notice that part (iv.a) in Lemma \ref{lem:right_multiplication} implies that the wing block-rows of $L_P(\lambda)$ relative to $(\boldsymbol\ell,\mathbf{r},p,q)$ and the wing block-rows of $L^\prime(\lambda)$ relative to  $(\boldsymbol\ell,\mathbf{r}^\prime ,p,q)$ are located at the same positions.
This, in turn, implies that the wing block-rows of $L^\prime(\lambda)\Pi^n_{\mathbf{r}^\prime}$ relative to $(\boldsymbol\ell,\boldsymbol{\mathrm{id}},p,q)$ and the wing block-rows of $L_P(\lambda)\Pi^n_{\mathbf{r}}$ relative to  $(\boldsymbol\ell,\boldsymbol{\mathrm{id}} ,p,q)$ are located also at the same positions.
Moreover, in view of \eqref{r-case} and \eqref{eq:thm_aux1}, the wing block-rows of  $L^\prime(\lambda)\Pi^n_{\mathbf{r}^\prime}$ and $L_P(\lambda)\Pi^n_{\mathbf{r}}$ are equal. 
Additionally,  Lemma \ref{xtSIP}  implies $\mathrm{heads}(x,\rev(\mathbf{r}^\prime_q),\rev(\mathbf{q}),\rev(\boldsymbol\ell_q)) = \mathrm{heads}(\rev(\mathbf{r}^\prime_q),\rev(\mathbf{q}),\rev(\boldsymbol\ell_q))$.
Thus, $j\in\mathrm{heads}(x,\rev(\mathbf{r}^\prime_q),\rev(\mathbf{q}),\rev(\boldsymbol\ell_q))$ if and only if $j\in\mathrm{heads}(\rev(\mathbf{r}^\prime_q),\rev(\mathbf{q}),\rev(\boldsymbol\ell_q))$. 
This, together with Lemma \ref{txSIP}, implies part (c).

It now follows that, when this lemma holds for a GFPR of the form $M_{\boldsymbol\ell_q}(\mathcal{X})(\lambda M^P_{-k} - M^P_{\mathbf q})$, then it also holds for $M_{\boldsymbol\ell_q}(\mathcal{X})(\lambda M^P_{-k} - M^P_{\mathbf q})M_{\mathbf r_q}(\mathcal Y)$ for arbitrary $\mathbf r_q$ such that $(\boldsymbol\ell_q, \mathbf q, \mathbf r_q)$ satisfies the SIP and arbitrary matrix assignment $\mathcal{Y}$. 
Since this lemma holds for Fiedler pencils $\lambda M^P_{-k} - M^P_{\mathbf q}$, we have proven that it holds for $(\lambda M^P_{-k} - M^P_{\mathbf q})M_{\mathbf r_q}(\mathcal Y)$ for arbitrary $\mathbf r_q$ such that $(\mathbf q, \mathbf r_q)$ satisfies the SIP and arbitrary  $\mathcal{Y}$.
So we still have to show that the lemma holds for $M_{\boldsymbol\ell_q}(\mathcal{X})(\lambda M^P_{-k} - M^P_{\mathbf q})M_{\mathbf r_q}(\mathcal Y)$, when $\boldsymbol\ell_q \neq \emptyset$.
To this end, let  $Q(\lambda)$ be a GFPR associated with $P(\lambda)$ for which this lemma holds.  
We now show that this lemma also holds for $Q(\lambda)^{\mathcal{B}}$. 

Let $\Pi^n_{\boldsymbol\ell}$ and $\Pi^n_{\mathbf{r}}$ be block-permutation matrices such that $(\Pi^n_{\boldsymbol\ell})^{\mathcal B} Q(\lambda) \Pi^n_{\mathbf{r}}$ is an extended $(p,n,q,n)$-block Kronecker pencil whose body satisfies the AS condition for  $P(\lambda)$  partitioned as in \eqref{LDP}. 
Then, 
\begin{align*}
((\Pi^n_{\boldsymbol\ell})^{\mathcal B} Q(\lambda)\Pi^n_{\mathbf{r}})^{\mathcal B} &= (\Pi^n_{\mathbf{r}})^{\mathcal{B}} Q(\lambda)^{\mathcal{B}} \Pi^n_{\boldsymbol\ell} \\
&= \left [ \begin{array}{c|c} M(\lambda)^{\mathcal B} &  K_1(\lambda)^{\mathcal B} \\ \hline (K_2(\lambda)^{T})^{\mathcal B} & \phantom{\Big{(}} 0 \phantom{\Big{(}} \end{array} \right].
\end{align*}

It is not difficult to show that the pencil above is an extended $(q,n,p,n)$-block Kronecker pencil.
Moreover, since block-transposition maps each antidiagonal of a block-matrix into itself, the pencil $M(\lambda)^{\mathcal B}$ satisfies the AS condition for $P(\lambda)$, as $M(\lambda)$ does. 
Property (a) holds for $Q(\lambda)^{\mathcal{B}}$ because its block-rows and block-columns are, respectively, the block-columns and the block-rows of $Q(\lambda)$, for which property (a) holds by assumption.
 Furthermore, properties (b) and (c) for $Q(\lambda)$ imply, respectively, that properties (c) and (b) hold for $ Q(\lambda)^{\mathcal B}$. 
Now, since this lemma holds for $Q(\lambda) = (\lambda M^P_{-k} - M^P_{\mathbf q}) M_{\mathbf r_q}(\mathcal Y)$ as noted above, it also holds for $Q(\lambda)^{\mathcal B} = M_{\rev(\mathbf r_q)}(\rev(\mathcal Y))(\lambda M^P_{-k} - M^P_{\rev(\mathbf q)})$ (recall Lemma \ref{revB}). 
Consider some arbitrary $\boldsymbol\ell_q$ such that $( \boldsymbol\ell_q, \mathbf q, \mathbf r_q)$ satisfies the SIP and an arbitrary matrix assignment $\mathcal{X}$ for $\boldsymbol\ell_q$. 
By the comments above, this lemma also holds for $M_{\rev(\mathbf r_q)}(\rev(\mathcal{Y}))(\lambda M^P_{-k} - M^P_{\rev(\mathbf q)})M_{\rev(\boldsymbol\ell_q)}(\rev(\mathcal X))$, and hence  for $M_{\boldsymbol\ell_q}(\mathcal X)(\lambda M^P_{-k} - M^P_{\mathbf q})M_{\mathbf r_q}(\mathcal Y)$. 
This establishes the lemma for arbitrary $\boldsymbol\ell_q$ and $\mathbf r_q$.
\end{proof}

\begin{remark}\label{remark:permutation2}
We note that part (c) in Theorem \ref{right-canonical-1} implies that the block-permutations $\Pi_{\mathbf{r}}^n$ and $\Pi_{\boldsymbol\ell}^n$ in \eqref{r-case} are, respectively, of the form $I_n\oplus \Pi_{\mathbf{\widetilde{r}}}^n$ and $I_n\oplus \Pi_{\widetilde{\boldsymbol\ell}}^n$, for some permutations $\mathbf{\widetilde{r}}$ and $\widetilde{\boldsymbol\ell}$ of the set $\{1:k-1\}$.
\end{remark}

Another auxiliary result to prove Theorem \ref{thm:main_GFPR} is Theorem \ref{left-canonical}.
In the proof of this result, we will make use of the following properties of elementary matrices:
\begin{align}
&R_{k,n}M_i(B) R_{k,n}  = M_{-k+i}(B), \quad \textrm{for $i=0:k-1$ and arbitrary $B$},\label{line3} \\
&R_{k,n} M_{-i}(B) R_{k,n} = M_{k-i}(B),\quad \textrm{ for $i=1:k$ and arbitrary $B$},\label{line4}
\end{align}
where $R_{s,n}$ is the block sip matrix defined in \eqref{eq:sip}.
Moreover, to keep the notation simple, we will omit the second index of $R_{s,n}$ and just write $R_s$, since it is going to remain constant and equal to $n$ throughout the whole proof.

\begin{theorem}\label{left-canonical}
Let $P(\lambda)=\sum_{i=0}^k A_i \lambda^i\in\mathbb{F}[\lambda]^{n\times n}$ be a matrix polynomial of  grade $k \geq 2$. 
Let $\mathbf{z}$ be a permutation of $\{-k:-1\}$ and let 
\[
L_P(\lambda)= M_{ \boldsymbol\ell_z}(\mathcal{Z}) (\lambda M^P_{\mathbf{z}} - M^P_{0}) M_{ \mathbf{r}_z}(\mathcal{W}),
\]
be a GFPR associated with $P(\lambda)$. 
Then, there exist block-permutation matrices $\Pi_{\boldsymbol\ell}^n$ and $\Pi_{\mathbf{r}}^n$ such that 
\begin{equation}
\label{z-canonical}
  \begin{array}{cl}
  C(\lambda):=(\Pi_{\boldsymbol\ell}^n)^{\mathcal{B}} L_P(\lambda)\Pi_{\mathbf{r}}^n=
  \left[
    \begin{array}{c|c}
0 & L_1(\lambda) \\ \hline \phantom{\Big{(}} L_2(\lambda)^T \phantom{\Big{(}} \phantom{\Big{(}}&  N(\lambda)  
      \end{array}
    \right]&
    \begin{array}{l}
      \left. \vphantom{L_{1}(\lambda)} \right\} {\scriptstyle (\mathfrak{h}(k+\mathbf{z})-1)n}\\
      \left. \vphantom{\phantom{\Big{(}} L_2(\lambda)^T}\right\} {\scriptstyle \mathfrak{h}(\rev(k+\mathbf{z}))n}
    \end{array}\\
    \hphantom{C(\lambda):=(\Pi_{\boldsymbol\ell}^n)^{\mathcal{B}} L(\lambda)\Pi_{\mathbf{r}}^n=}
    \begin{array}{cc}
      \underbrace{\hphantom{L_2(\lambda)^T}}_{(\mathfrak{h}(\rev(k+\mathbf{z}))-1)n}&\underbrace{\hphantom{N(\lambda)}}_{\mathfrak{h}(k+\mathbf{z})n}
    \end{array}
  \end{array}
  \>,
\end{equation}
where the pencil $N(\lambda)$ satisfies the AS condition for $P(\lambda)$, and where $L_1(\lambda)$ and $L_2(\lambda)$ are wing pencils. 
Moreover, the last block-row of $C(\lambda)$ is the last block-row of $(\Pi^n_{\boldsymbol\ell})^{\mathcal{B}}L_P(\lambda) $, and the last block-column of $C(\lambda)$ is the last block-column of  $L_P(\lambda)\Pi^n_{\mathbf{r}}$, respectively. 
Additionally, the block-entry of $N(\lambda)$ in position $(\mathfrak{h}(\rev(k+\mathbf{z})), \mathfrak{h}(k+\mathbf{z}))$ equals $\lambda A_1+A_0$.
\end{theorem}
\begin{proof}
 Let $\widehat{P}(\lambda) := \rev(-P(\lambda))=\sum_{i=0}^k -A_{k-i} \lambda^i$ and let $L_P(\lambda)=: \lambda L_1 - L_0$. 
 Let us consider the pencil
\begin{align*}
\widehat{L}(\lambda) :=& R_k  \rev( - L_P(\lambda) )R_k = \lambda R_k L_0 R_k - R_k L_1R_k=\\
&\lambda R_k M_{\boldsymbol\ell_z, \mathbf{r}_z}(\mathcal{Z}, \mathcal{W}) M^P_{0} R_k - R_k M_{\boldsymbol\ell_z}(\mathcal{Z}) M^P_{\mathbf{z}} M_{\mathbf{r}_z}(\mathcal{W})R_k.
\end{align*}

Since $\mathbf{z}$ is a permutation of $\{-k:-1\}$, there exists a permutation $\mathbf{q}$ of $\{0:k-1\}$ such that $\mathbf{z}=-k+\mathbf{q}$. 
Taking into account \eqref{line3}--\eqref{line4} and the fact that $R_k$ is nonsingular with $R_k^{-1}=R_k$, it is not difficult to see that
\[
\widehat{L}(\lambda)= M_{k+\boldsymbol\ell_z}(\mathcal{Z}) ( \lambda M_{-k}^{\widehat{P}} - M_{\mathbf{q}}^{\widehat{P}}) M_{k+\mathbf{r}_z}(\mathcal{W}),
\]
which is a GFPR associated with $\widehat{P}(\lambda)$. 
Notice that if $(\boldsymbol\ell_z, \mathbf{z}, \mathbf{r}_z)$ satisfies the SIP, so does $k+(\boldsymbol\ell_z, \mathbf{z}, \mathbf{r}_z)$. 
By Theorem \ref{right-canonical-1}, there exist block-permutation matrices $\Pi^n_{\boldsymbol\ell^\prime}$ and $\Pi^n_{\mathbf{r}^\prime}$ such that 
\begin{equation}\label{partitionb}
(\Pi^n_{\boldsymbol\ell^\prime})^{\mathcal{B}} \widehat{L}(\lambda) \Pi^n_{\mathbf{r}^\prime}= \left[ \begin{array}{c|c} M(\lambda) & K_2(\lambda)^T \\ \hline K_1(\lambda) & 0 \end{array} \right].
\end{equation}
is an extended $(\mathfrak{h}(\mathbf{q})-1,n,\mathfrak{h}(\rev(\mathbf{q}))-1,n)$-block Kronecker pencil, whose body $M(\lambda)$ satisfies the AS condition for $\widehat{P}(\lambda)$.
Then, notice that
\begin{align*}
- \rev(R_k (\Pi^n_{\boldsymbol\ell^\prime})^{\mathcal{B}} \widehat{L}(\lambda) \Pi^n_{\mathbf{r}^\prime}R_k) &= - \rev[R_k (\Pi^n_{\boldsymbol\ell^\prime})^{\mathcal{B}} R_k \rev(-L(\lambda)) R_k \Pi^n_{\mathbf{r}^\prime} R_k]\\
&= \rev[(R_k (\Pi^n_{\boldsymbol\ell^\prime})^{\mathcal{B}} R_k) \rev(L(\lambda)) (R_k \Pi^n_{\mathbf{r}^\prime} R_k)] \\
&=(R_k (\Pi^n_{\boldsymbol\ell^\prime})^{\mathcal{B}} R_k) L(\lambda) (R_k \Pi^n_{\mathbf{r}^\prime} R_k)\\
&=( \Pi^n_{\boldsymbol\ell})^{\mathcal{B}} L_P(\lambda) \Pi^n_{\mathbf{r}},
\end{align*}
where $\Pi^n_{\boldsymbol\ell}:=R_k \Pi^n_{\boldsymbol\ell^\prime} R_k$ and $\Pi^n_{\mathbf{r}}:=R_k \Pi^n_{\mathbf{r}'} R_k$ are block-permutation matrices. 
On the other hand, setting $p:=\mathfrak{h}(\mathbf{q})-1$ and $q:=\mathfrak{h}(\rev(\mathbf{q}))-1$, from \eqref{partitionb}, we get
\begin{align*}
-\rev(R_k (\Pi^n_{\boldsymbol\ell^\prime})^{\mathcal{B}} \widehat{L}(\lambda) \Pi^n_{\mathbf{r}^\prime} R_k) & = - \rev\left ( R_k \left[ \begin{array}{c|c} M(\lambda) & K_2(\lambda)^T \\ \hline K_1(\lambda) & 0 \end{array} \right] R_k \right)\\
&=  \left [\begin{array}{cc} 0 & -\rev(R_pK_1(\lambda)R_{p+1}) \\ -\rev(R_qK_2(\lambda)R_{q+1})^T & -\rev(R_{q+1}M(\lambda)R_{p+1}) \end{array}  \right ].
\end{align*}
Note that, if $K(\lambda)$ is a  wing pencil, then 
$ -\rev(R K(\lambda) R) $ is also a wing  pencil.
 Thus, letting $N(\lambda):=-\rev(R_{q+1}M(\lambda)R_{p+1})$,  $L_1(\lambda):=-\rev(R_pK_1(\lambda)R_{p+1}) $ and  $L_2(\lambda):=-\rev(R_qK_2(\lambda)R_{q+1})$, we obtain that \eqref{z-canonical} holds.
 
Writing $M(\lambda)=\lambda H_1+H_0$, next, we show that $N(\lambda)= -\lambda R_{q+1}H_0R_{p+1}-R_{q+1}H_1R_{p+1}$ satisfies the AS condition for $P(\lambda)$.
As shown above, the pencil $M(\lambda)$ is a $(q+1)\times (p+1)$ block-pencil with blocks of size $n\times n$, since $(\Pi^n_{\boldsymbol\ell^\prime })^{\mathcal{B}} \widehat{L}(\lambda) \Pi^n_{\mathbf{r}^\prime}$ is  an extended $(p, n, q, n)$-block Kronecker pencil. 
Moreover, $p+q+1=k$ and $M(\lambda)$ satisfies the AS condition for $\widehat{P}(\lambda)$. 
We observe that $[R_{q+1}M(\lambda)R_{p+1}]_{ij} = [M(\lambda)]_{q+2-i, p+2-j}$.  
Thus, using the notation in Definition \ref{AS-def}, we have
\begin{align*}
AS(N, s)& = \sum_{i+j=k+2-s} [-R_{q+1}H_0R_{p+1}]_{ij} + \sum_{i+j=k+1-s} [-R_{q+1}H_1R_{p+1}]_{ij} \\
&= -\left\{\sum_{i+j=k+2-s} [H_0]_{q+2-i, p+2-j} + \sum_{i+j=k+1-s} [H_1]_{q+2-i, p+2-j} \right\}\\
&= -\left\{ \sum_{i+j=k+1-(k-s)} [H_0]_{i, j} + \sum_{i+j=k+2-(k-s)} [H_1]_{i, j} \right\}=- (-A_{k-(k-s)})=A_s,\\
\end{align*} 
where the third equality follows using the fact that $p+q+1=k$ and the fourth equality follows from the fact that $M(\lambda)$ satisfies the AS condition for $\widehat{P}(\lambda)$. 

Finally, notice that, by part (a) in Theorem \ref{right-canonical-1}, the first block-row  and the first block-column of the pencil $\widehat{L}(\lambda)$ are, respectively,  the first block-row and the first block-column  of the pencil $(\Pi^n_{\boldsymbol\ell^\prime})^{\mathcal{B}} \widehat{L}(\lambda) \Pi^n_{\mathbf{r}^\prime}$, and the block-entry in position $(1,1)$ of $M(\lambda)$ is $- \lambda A_{0}-A_1$. 
Since $N(\lambda)=-\rev(R_{q+1}M(\lambda)R_{p+1})$, the block-entry in position $(q+1, p+1)$ of $N(\lambda)$ is $\lambda A_1 + A_0$.  
Moreover, since $L_P(\lambda) = - \rev(R_k \widehat{L}(\lambda) R_k)$ and $C(\lambda)=( \Pi^n_{\boldsymbol\ell})^{\mathcal{B}}L_P(\lambda) \Pi^n_{\mathbf{r}} = - \rev [R_k (\Pi^n_{\boldsymbol\ell^\prime})^{\mathcal{B}} \widehat{L}(\lambda) \Pi^n_{\mathbf{r}^\prime} R_k]$, the claim about the last block-row and the last block-column of $L_P(\lambda)$ and $C(\lambda)$ follows.
\end{proof}

\begin{remark}
We will refer to any pencil of the form \eqref{z-canonical} as a \emph{reversed} extended block Kronecker pencil.
\end{remark}

\begin{remark}\label{remark:permutation3}
We note that  Theorem \ref{right-canonical-1} implies that the block-permutations $\Pi_{\mathbf{r}}^n$ and $\Pi_{\boldsymbol\ell}^n$ in \eqref{z-canonical}  are, respectively, of the form $\Pi_{\mathbf{\widetilde{r}}}^n\oplus I_n$ and $\Pi_{\widetilde{\boldsymbol\ell}}^n\oplus I_n$, for some permutations $\mathbf{\widetilde{r}}$ and $\widetilde{\boldsymbol\ell}$ of the set $\{1:k-1\}$.
\end{remark}

\subsection{Proof of Theorem \ref{thm:main_GFPR}}\label{sec:thm:main_GFPR}

Armed with Lemma  \ref{lem:splitting} together with Theorems \ref{right-canonical-1} and \ref{left-canonical}, we are in a position to prove Theorem \ref{thm:main_GFPR}.

\medskip

\begin{proof}(of Theorem \ref{thm:main_GFPR})
Let $P(\lambda)=\sum_{i=0}^k A_i\lambda^i\in\mathbb{F}[\lambda]^{n\times n}$ be a matrix polynomial of degree $k$.
 Let $h\in \{0:k-1\}$ and let $\mathbf{q}$ be a permutation of $\{0:h\}$. 
 Let $\mathbf{z}$ be a permutation of $\{-k: -h-1\}$ and let $L_P(\lambda)$ be a GFPR as in \eqref{fprgen}.
By Lemma \ref{lem:splitting}, the pencil $L_P(\lambda)$ can be written in the following form:
\[
L_P(\lambda) = \left[ \begin{array}{c|c|c} D_\mathbf{z}(\lambda) & y_\mathbf{z}(\lambda) & 0 \\ \hline x_\mathbf{z}(\lambda) & c(\lambda) & x_\mathbf{q}(\lambda) \\ \hline 0 & y_\mathbf{q}(\lambda) & D_\mathbf{q}(\lambda)   \end{array} \right ],
\]
where $ D_\mathbf{q}(\lambda) \in \mathbb{F}[\lambda]^{nh\times nh}$, $D_\mathbf{z}(\lambda)\in \mathbb{F}[\lambda]^{n(k-h-1) \times n(k-h-1)}$, $c(\lambda)\in \mathbb{F}[\lambda]^{n\times n}$, $x_\mathbf{z}(\lambda)\in \mathbb{F}[\lambda]^{n \times n(k-h-1)}$, $x_\mathbf{q}(\lambda) \in \mathbb{F}[\lambda]^{n \times nh}$, $y_\mathbf{z}(\lambda)\in \mathbb{F}[\lambda]^{n(k-h-1)\times n}$ and $y_\mathbf{q}(\lambda) \in \mathbb{F}[\lambda]^{nh \times n}$. 
In addition, let us denote
\[
F(\lambda) := \left[ \begin{array}{cc}  c(\lambda) & x_\mathbf{q}(\lambda) \\ y_\mathbf{q}(\lambda) & D_\mathbf{q}(\lambda)\end{array} \right], \quad \mbox{and} \quad G(\lambda) := \left[ \begin{array}{cc} D_\mathbf{z}(\lambda) & y_\mathbf{z}(\lambda) \\ x_\mathbf{z}(\lambda) &  c(\lambda) \end{array} \right].
 \]
Then, by Theorem \ref{lem:splitting}, we get that
\[
F(\lambda) = M_{\boldsymbol\ell_q}(\mathcal{X}) ( \lambda M^Q_{-h-1} - M^Q_{\mathbf{q}}) M_{\mathbf{r}_q}(\mathcal{Y})
\]
is a GFPR associated with $Q(\lambda):= \lambda^{h+1} A_{h+1} + \lambda^h A_h + \cdots + \lambda A_1 + A_0$. 
Thus, from part (a) in Theorem \ref{right-canonical-1}, we obtain that $c(\lambda)=\lambda A_{h+1}+A_h$.
Also by Theorem \ref{right-canonical-1}  and according to Remark \ref{remark:permutation2}, there exist block-permutation matrices $\Pi_{\mathbf{r}_1}=I_n\oplus \Pi_{\widetilde{\mathbf{r}}_1}$ and $\Pi_{\boldsymbol\ell_1}=I_n\oplus\Pi_{\widetilde{\boldsymbol\ell}_1}$  such that 
\begin{align*}
\Pi_{\boldsymbol\ell_1}^{\mathcal{B}} F(\lambda) \Pi_{\mathbf{r}_1}=&
\begin{bmatrix}
I_n & 0 \\ 0 & \Pi_{\widetilde{\boldsymbol\ell}_1}^{\mathcal{B}}
\end{bmatrix}
\begin{bmatrix}
\lambda A_{h+1}+A_h & x_\mathbf{q}(\lambda) \\ y_\mathbf{q}(\lambda) & D_\mathbf{q}(\lambda) 
\end{bmatrix}
\begin{bmatrix}
I_n & 0 \\ 0 & \Pi_{\widetilde{\mathbf{r}}_1}
\end{bmatrix} = \\
&\begin{bmatrix}
\lambda A_{h+1}+A_h & x_\mathbf{q}(\lambda)\Pi_{\widetilde{\mathbf{r}}_1}  \\ 
\Pi_{\widetilde{\boldsymbol\ell}_1}^{\mathcal{B}} y_\mathbf{q}(\lambda) & \Pi_{\widetilde{\boldsymbol\ell}_1}^{\mathcal{B}}D_\mathbf{q}(\lambda) \Pi_{\widetilde{\mathbf{r}}_1}
\end{bmatrix} =
\left[\begin{array}{c|c}
 M(\lambda) & K_2(\lambda)^T\\
 \hline  K_1(\lambda) & 0 
 \end{array}\right] =: \\
 &
\left[\begin{array}{cc|c} 
\lambda A_{h+1}+A_{h} & m_1(\lambda) & k_2(\lambda)^T \\
m_2(\lambda) & \widetilde{M}(\lambda) & \widetilde{K}_2(\lambda)^T \\ \hline
k_1(\lambda) & \widetilde{K}_1(\lambda) & \phantom{\Big{(}} 0 \phantom{\Big{(}}
 \end{array}\right]
\end{align*}
 is an extended $(\mathfrak{h}(\mathbf{q})-1,n,\mathfrak{h}(\rev(\mathbf{q}))-1,n)$-block Kronecker pencil for $Q(\lambda)$.
 
Furthermore,  by Theorem \ref{lem:splitting},  the pencil
\[
G(\lambda) = M_{\boldsymbol\ell_z}(\mathcal{Z}) ( \lambda M^Z_{h + \mathbf{z}} - M^Z_{0}) M_{\mathbf{r}_z}(\mathcal{W})
\]
is a GFPR associated with $Z(\lambda):= \lambda^{k-h} A_{k} + \lambda^{k-h-1}  A_{k-1} + \cdots +  \lambda A_{h+1}+ A_h$.
 By Theorem \ref{left-canonical}  and according to Remark \ref{remark:permutation3}, there exist block-permutation matrices $\Pi_{\mathbf{r}_2}= \Pi_{\widetilde{\mathbf{r}}_2}\oplus I_n$ and $\Pi_{\boldsymbol\ell_2}=\Pi_{\widetilde{\boldsymbol\ell}_2}\oplus I_n$ such that 
\begin{align*}
\Pi_{\boldsymbol\ell_2}^{\mathcal{B}} G(\lambda)\Pi_{\mathbf{r}_2}= &
\begin{bmatrix}
\Pi_{\widetilde{\boldsymbol\ell}_2}^{\mathcal{B}} & 0 \\
0 & I_n
\end{bmatrix}
\begin{bmatrix}
D_\mathbf{z}(\lambda) & y_\mathbf{z}(\lambda) \\ x_\mathbf{z}(\lambda) & \lambda A_{h+1}+A_h
\end{bmatrix}
\begin{bmatrix}
\Pi_{\widetilde{\mathbf{r}}_2} & 0 \\ 0 & I_n
\end{bmatrix} = \\
&\begin{bmatrix}
\Pi_{\widetilde{\boldsymbol\ell_2}}^{\mathcal{B}}D_\mathbf{z}(\lambda)\Pi_{\widetilde{\mathbf{r}}_2} & \Pi_{\widetilde{\boldsymbol\ell_2}}^{\mathcal{B}}y_\mathbf{z}(\lambda) \\
x_\mathbf{z}(\lambda)\Pi_{\widetilde{\mathbf{r}}_2} & \lambda A_{h+1}+A_h
\end{bmatrix}=
\left[ \begin{array} {c|c} 0 & L_1(\lambda) \\ \hline 
\phantom{\Big{(}} L_2(\lambda)^T \phantom{\Big{(}} & N(\lambda) \end{array} \right] =:\\
&\left[\begin{array}{c|cc}
0 & \widetilde{T}_1(\lambda) & t_1(\lambda) \\ \hline
\widetilde{T}_2(\lambda)^T & \widetilde{N}(\lambda) & \phantom{\Big{(}} n_1(\lambda) \phantom{\Big{(}} \\
t_2(\lambda)^T & n_2(\lambda) & \lambda A_{h+1}+A_h
\end{array}\right]
\end{align*}
is a reversed block Kronecker pencil for $Z(\lambda)$.
Then, notice that
\begin{align*}
&\begin{bmatrix}
\Pi_{\widetilde{\boldsymbol\ell}_2}^{\mathcal{B}} \\ & I_n \\ & & \Pi_{\widetilde{\boldsymbol\ell}_1}^{\mathcal{B}}
\end{bmatrix}
\begin{bmatrix}
D_\mathbf{z}(\lambda) & y_\mathbf{z}(\lambda) & 0 \\ 
x_\mathbf{z}(\lambda) & \lambda A_{h+1}+A_h & x_\mathbf{q}(\lambda) \\
0 & y_\mathbf{q}(\lambda) & D_\mathbf{q}(\lambda)
\end{bmatrix}
\begin{bmatrix}
\Pi_{\widetilde{\mathbf{r}}_2} \\ & I_n \\ & & \Pi_{\widetilde{\mathbf{r}}_1} 
\end{bmatrix} = \\
&\begin{bmatrix}
0 & \widetilde{T}_1(\lambda) & t_1(\lambda) & 0 & 0 \\
\widetilde{T}_2(\lambda)^T & \widetilde{N}(\lambda) & n_1(\lambda) & 0 & 0 \\
t_2(\lambda)^T & n_2(\lambda) & \lambda A_{h+1}+A_h & m_1(\lambda) & k_2(\lambda)^T \\
0 & 0 & m_2(\lambda) & \widetilde{M}(\lambda) & \widetilde{K}_2(\lambda)^T \\
0 & 0 & k_1(\lambda) & \widetilde{K}_1(\lambda) & 0
\end{bmatrix},
\end{align*}
which is block-permutationally equivalent to the pencil
\[
\left[\begin{array}{ccc|cc}
\widetilde{N}(\lambda) & n_1(\lambda) & 0 & \widetilde{T}_2(\lambda)^T & 0 \\
n_2(\lambda) & \lambda A_{h+1}+A_h & m_1(\lambda) & t_2(\lambda)^T & k_2(\lambda)^T \\
0 & m_2(\lambda) & \widetilde{M}(\lambda) & 0 & \widetilde{K}_2(\lambda)^T \\ \hline
\widetilde{T}_1(\lambda) & t_1(\lambda) & \phantom{\Big{(}} 0 \phantom{\Big{(}} & 0 & 0 \\
0 & k_1(\lambda) & \widetilde{K}_1(\lambda) & 0 & 0
\end{array}\right]=:
\left[\begin{array}{c|c}
H(\lambda) & S_2(\lambda)^T \\ \hline
S_1(\lambda)^T & \phantom{\Big{(}} 0 \phantom{\Big{(}}
\end{array}\right],
\]
where $S_1(\lambda)$ and $S_2(\lambda)$ are wing pencils by Theorem \ref{thm:minimal_bases_concatenation}.
Thus, to finish the proof, we only need to check that the pencil $H(\lambda)$ satisfies the AS condition for $P(\lambda)$.
Indeed, decomposing the pencil $H(\lambda)$ as follows
\[ 
\underbrace{\begin{bmatrix}
\widetilde{N}(\lambda) & n_1(\lambda) & 0 \\
n_2(\lambda) & \lambda A_{h+1}+A_h & 0 \\
0 & 0 & 0 
\end{bmatrix}}_{=:H_1(\lambda)}+
\underbrace{\begin{bmatrix}
0 & 0 & 0 \\
0 & \lambda A_{h+1}+A_h & m_1(\lambda) \\
0 & m_2(\lambda) & \widetilde{M}(\lambda) 
\end{bmatrix}}_{=:H_2(\lambda)}-
\underbrace{\begin{bmatrix}
0 & 0 & 0 \\
0 & \lambda A_{h+1}+A_h & 0 \\
0 & 0 & 0 
\end{bmatrix}}_{=:H_3(\lambda)},
\]
and using the linearity of the block antidiagonal sum, we  obtain
\[
\mathrm{AS}(H,s)=\mathrm{AS}(H_1,s)+\mathrm{AS}(H_2,s)-\mathrm{AS}(H_3,s), \quad s=0:k.
\]
Finally, using that $M(\lambda)$ and $N(\lambda)$ satisfy the AS condition for $Q(\lambda)$ and $Z(\lambda)$, respectively, it follows easily from the equation above that the pencil $H(\lambda)$  satisfies the AS condition for $P(\lambda)$.
\end{proof}

\medskip

We end this section with an interesting remark.
Given a GFPR, the tuples $\mathbf{q}$ and $\mathbf{z}$ in  \eqref{fprgen}  may not be uniquely defined.
When this  situation occurs, Theorem \ref{thm:main_GFPR} may lead to express the GFPR as an extended block Kronecker pencil in more that one way.  
We illustrate this phenomenon in Example \ref{ex:nonunique}, where we apply Theorem \ref{thm:main_GFPR} to the pencils in the standard basis of $\mathbb{DL}(P)$  (which consists of GFPR \cite{GFPR,4m-vspace}) associated with a matrix polynomial $P(\lambda)$ of degree 3. 
Moreover, when appropriate, we give the different ways in which the same pencil can be expressed as an extended block Kronecker pencil and show that each of those representations corresponds to a distinct factorization \eqref{fprgen} of the matrix pencil .

\begin{example}\label{ex:nonunique}
Here, we use the symbol $\sim$ to denote that two pencils are permutationally equivalent.
Let $P(\lambda)= \sum_{i=0}^3 A_i \lambda^i\in\mathbb{F}[\lambda]^{n\times n}$.
Let us denote by $D_1(\lambda, P)$, $D_2(\lambda, P)$ and $D_3(\lambda, P)$ the pencils in the standard basis of $\mathbb{DL}(P)$.
We recall that the pencil $D_1(\lambda, P)$ is expressed as a product of elementary matrices as follows
\[
D_1(\lambda, P)=:
\begin{bmatrix}
\lambda A_3+A_2 & A_1 & A_0 \\
A_1 & -\lambda A_1+A_0 & -\lambda A_0 \\
A_0 & -\lambda A_0 & 0
\end{bmatrix}=\lambda M_{-3,0:1, 0} - M_{0:2, 0:1, 0}.
\]
Next, we include three different  representations  of $D_1(\lambda, P)$ as an extended block Kronecker pencil (coherent with Theorem \ref{thm:main_GFPR}) and the corresponding factorization \eqref{fprgen} of $D_1(\lambda, P)$ in each case.
\begin{align*}
&\left[ \begin{array}{cc|c} \lambda A_3+A_2 & A_1 & A_0 \\ A_1 & -\lambda A_1+A_0 & -\lambda A_0\\ \hline A_0 & -\lambda A_0 & 0   \end{array} \right] \sim M_0( \lambda M_{-3} - M_{1:2, 0})M_{1,0}, \\
& \left[ \begin{array}{c|cc} \lambda A_3+A_2 & A_1 & A_0 \\ A_1 & -\lambda A_1+A_0 & -\lambda A_0\\  A_0 & -\lambda A_0 & 0   \end{array} \right] \sim ( \lambda M_{-3} - M_{0:2}) M_{0:1, 0}\quad \mbox{and} \\
&\left[ \begin{array}{ccc} \lambda A_3+A_2 & A_1 & A_0 \\ \hline A_1 & -\lambda A_1+A_0 & -\lambda A_0\\  A_0 & -\lambda A_0 & 0   \end{array} \right] \sim M_{0:1,0}( \lambda M_{-3} - M_{2, 1, 0}).
\end{align*}
Now, we recall that the pencil $D_2(\lambda, P)$ is expressed as a product of elementary matrices as follows 
\[
D_2(\lambda, P):=
\left[ \begin{array}{ccc} -A_3 & \lambda A_3 & 0 \\ \lambda A_3 & \lambda A_2 +A_1 & A_0\\ 0 & A_0 & -\lambda A_0  \end{array} \right]=
 \lambda M_{0,-3:-2, -3} - M_{-3,0:1, 0}.
 \]
The possible representations of the pencil $D_2(\lambda, P)$ as an extended block Kronecker pencil, together with the corresponding factorization of $D_2(\lambda, P)$,  are
\begin{align*}
& \left[ \begin{array}{cc|c}  \lambda A_3 & \lambda A_2 +A_1 & A_0\\ 0 & A_0 & -\lambda A_0 \\\hline -A_3 & \lambda A_3 & 0  \end{array} \right] \sim  M_{-3} (\lambda M_{-2, -3} - M_{0:1}) M_0, \quad \mbox{and}\\
&\left[ \begin{array}{c|cc}  \lambda A_3 & 0 & -A_3\\ \lambda A_2+A_1 & A_0 & \lambda A_3 \\ A_0 & -\lambda A_0 & 0  \end{array} \right] \sim  (\lambda M_{-3: -2} - M_{0:1}) M_{-3, 0}.
\end{align*}
Finally, we recall that the pencil $D_3(\lambda, P)$ can be expressed as a product of elementary matrices as follows
\[
D_3(\lambda, P):=  \left[ \begin{array}{ccc} 0 & -A_3 & \lambda A_3\\ -A_3 & \lambda A_3-A_2 & \lambda A_2 \\ \lambda A_3 & \lambda A_2 & \lambda A_1 +A_0  \end{array} \right] = \lambda M_{-3:-1, -3:-2, -3} - M_{0, -3:-2, -3}.
\]
Two possible representations of $D_3(\lambda, P)$ as an extended block Kronecker pencil, together with the corresponding factorization of $D_3(\lambda, P)$,  are
\begin{align*}
&\left[ \begin{array}{cc|c} \lambda A_3 -A_2 & \lambda A_2 & - A_3\\ \lambda A_2 & \lambda A_1+ A_0 & \lambda A_3 \\ \hline -A_3 & \lambda A_3 & 0  \end{array} \right] \sim  M_{-3} ( \lambda M_{-2:-1, -3} - M_{0} ) M_{-2:-3}, \quad \mbox{and} \\
&\left[ \begin{array}{c|cc} \lambda A_3  & -A_3 & 0\\ \lambda A_2 & \lambda A_3 - A_2 & - A_3 \\ 
\lambda A_1+A_0  & \lambda A_2 &  \lambda A_3  \end{array} \right] \sim  ( \lambda M_{-3:-1} - M_{0} ) M_{-3:-2, -3}.
\end{align*}

Using Theorem \ref{thm:Lambda-dual-pencil-linearization}, it is easy to see that , if $A_0$ is nonsingular, then $D_1(\lambda, P)$ is a strong linearization of $P(\lambda)$, if $A_3$ is nonsingular, then $D_3(\lambda, P)$ is a strong linearization of $P(\lambda)$ and, if both $A_0$ and $A_3$ are nonsingular, then $D_2(\lambda, P)$ is a strong linearization of $P(\lambda)$. 
Thus, these pencils are only strong linearizations when $P(\lambda)$ is regular, which is coherent with Remark \ref{rem:only_for_regular}.
\end{example}

\section{Conclusions}\label{sec:conclusion}

In the last decade, many new families of linearizations for a matrix polynomial have been constructed. 
The families of square Fiedler-like pencils, that include the Fiedler pencils, the generalized Fiedler pencils, the Fiedler pencils with repetition and the generalized Fiedler pencils with repetition, consist of pencils with good properties
but whose definition involves products of the so-called elementary matrices.
 This fact is a disadvantage when proving theorems about these pencils since the proofs can become very involved. 
 Also, each of these families, although related, are studied separately in the literature.
Recently, a new family of pencils associated with a matrix polynomial $P(\lambda)$, called the block minimal bases pencils, was introduced. 
The pencils in this family are defined in terms of their block-structure and it is straightforward to determine when they are strong linearizations of $P(\lambda)$.
 A subfamily of this family, called the block-Kronecker pencils was also identified and it was proven that the Fiedler pencils are in this family, up to permutations of rows and
columns. 
In this paper, we provide a unified approach to the study of Fiedler-like pencils via block minimal bases pencils. 
We show that not all Fiedler-like pencils are block minimal bases pencils, and provide a generalization of the family of block-Kronecker pencils that we call the extended
block-Kronecker pencils.
 We also prove that, with the exception of the non proper GFP, all Fiedler-like pencils are extended block Kronecker pencils, up to permutations of rows and columns, and that, under some generic nonsingularity conditions, they are minimal bases pencils, again up to permutations.

 \vspace{0.5cm}
\noindent  \textbf{Acknowledgements.}
 We would like to express our gratitude to Heike Fa{$\ss$}bender and Philip Saltenberger for pointing out the connections of some results in Section 3 in this paper with some results in their paper ``Block-Kronecker Ansatz Spaces for Matrix Polynomials'', which was written independently. 
We would also like to thank NSF for the support to the ``UCSB Summer Research Experiences for Undergraduates'' program that funded two undergraduate students, R. Saavedra and B. Zykoski, to work on this project during Summer 2016. Additionally, we thank the Young Mathematicians Conference (Ohio, USA, August, 2016) for supporting the participation of these students in that conference where they presented the results of this paper.


\end{document}